\documentclass[12pt]{article}

\RequirePackage{amsthm,amsmath,amsfonts,amssymb}
\RequirePackage[numbers]{natbib}
\usepackage[utf8]{inputenc}
\usepackage{enumerate}
\usepackage{lmodern}
\usepackage[T1]{fontenc}
\usepackage{verbatim}
\usepackage{textcomp}

\usepackage[english]{babel}
\usepackage[a4paper,vmargin={3.5cm,3.5cm},hmargin={2.5cm,2.5cm}]{geometry}
\usepackage[font=sf, labelfont={sf,bf}, margin=1cm]{caption}
\usepackage[pdftex]{hyperref}
\usepackage[pdftex]{color,graphicx}
\usepackage{xcolor}

\usepackage{amsmath,amsfonts,amssymb,amsthm,mathrsfs,mathtools,bbm}
\usepackage{empheq}
\newtheorem{theorem}{Theorem}[section]

\newtheorem{lemma}[theorem]{Lemma}
\newtheorem{proposition}[theorem]{Proposition}

\theoremstyle{definition}
\newtheorem{remark}[theorem]{Remark}
\newtheorem{definition}[theorem]{Definition}

\newtheorem{assumption}[theorem]{Assumption}

\def\RR{\mathbb{R}}
\def\R{\mathbb{R}}
\def\ZZ{\mathbb{Z}}
\newcommand{\bN}{\mathbf{N}}
\newcommand{\Z}{\mathbb{Z}}

\newcommand{\tF}{\tilde{F}}

\newcommand{\cN}{{\cal N}}

\newcommand{\LL}{\mathscr L}
\newcommand{\tX}{\tilde{X}}

\def\NN{\mathbb{N}}\def\N{\mathbb{N}}

\def\PP{\mathbb{P}}\def\Pr{\mathbb{P}}
\def\EE{\mathbb{E}} \newcommand{\E}{\mathbb{E}}
\def\Var{\mathbb{V}\mathrm{ar}}
\def\Cov{\mathbb{C}\mathrm{ov}}
\def\1{{\bf 1}}
\def\al{\alpha}

\def\ep{\varepsilon}

 \newcommand{\toD}{\overset{{\cal D}}\longrightarrow}
 \newcommand{\toL}{\overset{L^1}\longrightarrow}
 
\newcommand{\eps}{\varepsilon}
\renewcommand{\emptyset}{\varnothing}
\newcommand{\fmax}{f_{\rm max}}

\renewcommand{\eta}{{\cal P}}

\def\cC{\mathcal{C}}
\def\cK{\mathcal{K}}
\def\cX{\mathcal{X}}
\newcommand{\X}{\mathcal{X}}
\newcommand{\cY}{\mathcal{Y}}
\newcommand{\A}{\mathcal{A}}

\def\cL{\mathcal{L}}
\def\cH{\mathcal{H}} 
\def\bX{\mathbf{X}}

\def\dtv{d_{\mathrm{TV}}}
\def\dk{{d_\mathrm{K}}}
\def\scr{\mathscr}
\newcommand{\bea}{\begin{eqnarray}}
\newcommand{\eea}{\end{eqnarray}}
\newcommand{\bean}{\begin{eqnarray*}}
\newcommand{\eean}{\end{eqnarray*}}

\numberwithin{equation}{section}

\DeclareMathOperator\diam{diam}
\DeclareMathOperator\dist{dist}
\DeclareMathOperator\cross{Cross}
\def\Bin{\mathrm{Bin}}
\newcommand{\Po}{{\mathcal P}}

\renewcommand{\P}{{\cal P}}

\begin{document}

\title{\bf On the components of random geometric graphs in the dense limit}
\author{Mathew D. Penrose and Xiaochuan Yang}

\author{ Mathew D. Penrose$^{1,2}$ and Xiaochuan Yang$^{1,3}$  \\
{\normalsize{\em University of Bath and Banque Internationale \`a
Luxembourg S.A.}} }

 \footnotetext{ Supported by EPSRC grant EP/T028653/1 }

         \footnotetext[2]{Corresponding author: Department of Mathematical Sciences, University of Bath,
 Bath BA2 7AY, UK: \texttt{m.d.penrose@bath.ac.uk} }
        \footnotetext[3]{Banque Internationale \`a Luxembourg S.A.,
69 Route d'Esch,
L-2953 Luxembourg:
                \texttt{xiaochuan.j.yang@gmail.com} }

\date{\today}

\maketitle

\begin{abstract}
	Consider the geometric graph on $n$ independent uniform random points in a connected compact region $A$ of $\R^d, d \geq 2$, with $C^2$ boundary, or in the unit square, with distance parameter $r_n$.  Let $K_n$ be the number of components of this graph, and $R_n$ the number of vertices not in the giant component.  Let $S_n$ be the number of isolated vertices.  We show that if $r_n$ is chosen so that $n(r_n)^d$ tends to infinity but slowly enough that $\E[S_n]$ also tends to infinity, then $K_n$, $R_n$ and $S_n$ are all asymptotic to $\mu_n$ in probability as $n \to \infty$ where (with $|A|$, $\theta_d$ and $|\partial A|$ denoting the volume of $A$, of the unit $d$-ball,  and the  perimeter of $A$ respectively) $\mu_n := ne^{-\pi n (r_n)^d/|A|}$ if $d=2$ and $\mu_n := ne^{-\theta_d n (r_n)^d/|A|} + (\theta_{d-1})^{-1} |\partial A| (r_n)^{1-d} e^{- \theta_d n (r_n)^d/(2|A|)}$ if $d\geq 3$.  We also give variance asymptotics and central limit theorems for $K_n$ and $R_n$ in this limiting regime when $d \geq 3$, and  for Poisson input with $d \geq 2$.  We extend these results (substituting $\E[S_n]$ for $\mu_n$) to a class of non-uniform distributions on $A$.
\end{abstract}

%\begin{keyword}[class=MSC]
%	\kwd[Primary ]{60D05}%Geometric prob. and stochastic geometry}
%	\kwd{60F25}%Lp-limit theorems}
%	\kwd{60F05}%Central limit and other weak limits}
%	\kwd[; secondary ]{05C80}%random graphs}
%\end{keyword}
%
%\begin{keyword}
%\kwd{Random geometric graph}
%\kwd{dense limiting regime}
%\end{keyword}

%\end{frontmatter}
%%%%%%%%%%%%%%%%%%%%%%%%%%%%%%%%%%%%%%%%%%%%%%
%% Please use \tableofcontents for articles %%
%% with 50 pages and more                   %%
%%%%%%%%%%%%%%%%%%%%%%%%%%%%%%%%%%%%%%%%%%%%%%
\tableofcontents

%%%%%%%%%%%%%%%%%%%%%%%%%%%%%%%%%%%%%%%%%%%%%%
%%%% Main text entry area:

\section{Introduction}
\label{secintro}

Given a compact set $A $ with a nice boundary
in Euclidean space $\R^d$, $ d \geq 2$,
the {\em random geometric graph} (RGG) based on a random point set
$\cX \subset A$ is
the graph $G(\cX,r)$ with vertex set $\X$ and edges between each pair
of points distant at most $r$ apart, in the Euclidean metric, for a specified 
distance parameter $r >0$. Such graphs are important in a variety
of applications (see \cite{Pen03}), including
modern topological data analysis (TDA), where the topological
properties of the graph are used to help understand the topology
of $A$.

In this paper we consider the
{\em number of components} of the graph $G$, denoted $K(G)$, where
$G=G(\X,r)$ with $\X $ a random sample of $n$ points in $A$ (denoted $\X_n$)
or the corresponding Poisson process (denoted $\eta_n$, and defined 
more formally later). 
In particular,
we investigate asymptotic properties for large $n$ with
 $r=r(n)$ specified and decaying to zero according
to a certain limiting regime (see \eqref{e:supcri}, 
\eqref{e:supcriupper} below).
Our results add significantly to the existing literature about the
limit theory of Betti numbers, an area that
has received intensive recent attention in %topological data analysis (TDA)
TDA. Indeed, the number of components of $G(\cX,r)$ is the 0-th Betti number of the occupied Boolean set $\cup_{x\in\cX}  B_{r/2}(x)$,
where $B_{r}(x)$ or $B(x,r)$ denotes the closed Euclidean ball of radius $r$
centred on $x$. 
%, one of the most fundamental invariants of a topological space. 
Given the sample $\X$, keeping track of $K(G(\X,r))$ while varying $r$
corresponds to the 0-th persistent homology, which leads to sparse topological descriptors in a 2D persistence diagram. See \cite{BK18,BK22} for related
geometric models of TDA. 
%{\bf Say something about persistence diagrams?}

%We are also concerned with the giant component - the component of $G(\cX_n, r)$ or $G(\eta_t, r)$ with the largest order. For the graphs
%we consider, most of the vertices lie in the giant, so for more

We  are also concerned with the {\em giant component} - the component of 
$G(\cX, r)$ with the largest order. For the graphs
we consider, most of the vertices lie in the giant component, so for more
detailed information we consider the total number of vertices $R(G(\cX,r))$ 
that are {\em not} in the giant component of $G(\cX,r)$.  
To be precise,
given a finite graph $G$ of order $n$,
list the orders of its components
in decreasing order as $L_1(G),L_2(G), \ldots, L_{K(G)}(G)$. Set
$R(G) := n-L_1(G)$.

We shall consider the limiting behaviour of
$K_n:= K(G(\X_n,r_n))$, $K'_n:= K(G(\eta_n,r_n))$,
$R_n:= R(G(\X_n,r))$
and $R'_n := R(G(\eta_n,r_n))$ as $n \to \infty$ with $r_n$ specified
for all $n \geq 1$.
 Let $\theta_d$ %(or $\theta$ for short)
 denote the volume of the unit radius ball in $\RR^d$,
 i.e. $\theta_d := \pi^{d/2}/\Gamma(1+d/2)$.
 For points uniformly distributed in $A$
 (which we call the {\em uniform case}), 
the main limiting regime for $r_n$ that we consider here is to assume that
as $n \to \infty$,
\bea
	\label{e:supcri}
nr_n^d\to + \infty;
%\mbox{ and }
\\
	\label{e:supcriupper}
\gamma_n := n (\theta_d/\lambda(A)) 
r_n^d - (2-2/d) ( \log n - \1_{\{d \geq 3\}} \log \log n) \to  -\infty, 
%\to 0 \mbox{ as } n\to\infty. 
%nr^d\to\infty \mbox{ and } nr^d/ \log n \to 0 \mbox{ as } n\to\infty. 
%	\lim_{ n\to \infty} (nr^d) = \infty \mbox{ and }  b : =
%	\limsup_{n \to \infty}  nr^d/ \log n \in [0,\infty).
	%\to 0 \mbox{ as } n\to\infty. 
\eea
where $\lambda$ denotes the Lebesgue measure on $\RR^d$
(note that throughout this paper, we adopt the convention
that if a symbol has both a subscript 
and a superscript, then the subscript is to be read first, so
$r_n^d$ means $(r_n)^d$.).
We call this the {\em intermediate} or 
{\em mildly dense} regime because the average
vertex degree is of order $\Theta(nr_n^d)$ and therefore grows
to infinity as $n$ becomes large, but only slowly in this regime.

 Other limiting regimes of $r$ are better  understood. 
	In the {\em thermodynamic regime} where $nr_n^d\to a \lambda(A)$
	with $a \in (0,\infty)$
	as $n \to \infty$, it holds 
	as $n \to \infty$ 
		that
\begin{align}
\frac{K_n}{n}\overset{\Pr}\longrightarrow c(a); ~~~~~~~~~ 
\frac{R_n}{n}\overset{\Pr}\longrightarrow c'(a),
	%~~~~~~~~ {\rm as} ~~ n \to \infty,
\label{0616a}
\end{align}
where $c(a)\in (0,1)$ is given explicitly in \cite[Theorem 13.25]{Pen03},
		and $c'(a) \in (0,1]$ is given less explicitly in
		\cite[Theorem 11.9]{Pen03}. If $a $ lies below
		a certain percolation threshold $a_c : = a_c(d) \in (0,\infty)$
		then $c'(a) =1$.
		 Central limit theorems for $K_n$ and
		 %(in the case where $A$ is a cube and $f$ is constant on $A$)
		 for
		$R_n$ in this regime are also proved in \cite{Pen03}
(these results hold for $K'_n$ and $R'_n$ as well as for $K_n$ and $R_n$). 

		In the {\em sparse regime} where $nr_n^d\to 0$,
		the average vertex degree goes to 0 and we
		still have \eqref{0616a}  with $c(0) = c'(0) =1$.
		This can be deduced from 
		the fact that   $c(a)\to 1$ as $a\to 0$ (which can be
		deduced from the formula in \cite{Pen03}), along with 
  coupling arguments.

%\end{itemize}

On the other hand, if
%$n r^d - (2-2/d) (\log n - \1_{\{d \geq 3\}} \log \log n) \to \infty$
$\gamma_n \to + \infty$,
and $\partial A$ is smooth or $A$ is a convex polygon,
%in $d=2$,
it follows from \cite[Theorem 1.1]{PY21} that 
%in the {\em connectivity regime}
%with $n r^d = b \log n$ for some constant $b >0 $, 
%{\bf (This $b$ was called $b/\theta$
%in the $k$-components paper. Maybe this discrepancy is OK)},
%provided $b$ exceeds some critical value $b_c$,
with probability tending to 1 as $n \to \infty$,
$G(\X_n,r_n) $ is fully connected 
%so that $K(G(\X_n,r_n))=1$ and $R(G(\X_n,r_n))=0$.
so that $K_n=1$ and $R_n=0$.
%If  $n r_n^d - (2-2/d) (\log n - \1_{\{d \geq 3 \}} \log \log n) $
%tends to a finite limit then we would expect
%$K_n-1$ to be asymptotically  Poisson with parameter 
%is determined by geometric characteristics of $A$; we shall confirm
%this for certain choices of $A$ in Theorem \ref{t:PoLimUn}.
We here call this limiting regime the {\em connectivity regime}
(in \cite{Pen03} this terminology was used slightly differently).

As well as the mildly dense regime  (\ref{e:supcri}), \eqref{e:supcriupper},
in this paper we also consider the case where $\gamma_n $ is bounded away from $-\infty$ and
$+\infty$ as $n \to \infty$; we call this the {\em critical regime
for connectivity}. Thus we consider
the whole range of possible limiting behaviours for $r$ in
between the  thermodynamic and connectivity regimes.
%is more general than
%the   logarithmic regime, in that we allow 
%$nr^d/\log n$ to tend to zero or to oscillate as $n \to \infty$.

 In TDA one is interested in understanding (for a fixed sample $\cX_n$)
 the number of components of $G(\cX_n,r)$ in the whole range of values from
 $r=0$, right
 up to the connectivity threshold (i.e. the smallest $r$ such
 that $G(\X_n,r)$ is connected).
% where connectivity is achieved (which happens at around
% $r =((b_c \log n)/n)^{1/d}$).
 Therefore it seems well worth trying to understand 
 %$K(G(\X_n,r_n))$ 
 $K_n$ in the mildly dense regime, as well as in other regimes.
 %at the same level of detail 
 %as it is already understood in the thermodynamic regime.
 Likewise, studying $R_n$ in this regime helps
 us understand the rate at which  the giant component
 swallows up the whole vertex set as $r$ approaches the
 connectivity threshold.

 Our main results for the uniform case refer to constants
 $\mu_n$ defined by
 \bea
	\mu_n := n e^{-n \theta_d %\lambda(A)^{-1}
	r_n^d/\lambda(A)} + \theta_{d-1}^{-1} |\partial A|
	r_n^{1-d} e^{-n \theta_d %\lambda(A)^{-1}
	r_n^d/(2 \lambda(A))} \1\{d \geq 3\},
	\label{e:defI'}
	\eea
where $\partial A$ denotes the topological boundary of $A$ and
	$|\partial A|$ denotes the $(d-1)$-dimensional
	Hausdorff measure of $\partial A$.
	
	We say $A$
	has a $C^2$ boundary (for short, $\partial A \in C^2$) 
	if for each $x \in \partial A$ there exists
a neighbourhood $U$ of $x$ and a real-valued function $\phi$ that is
defined on an open set in $\R^{d-1}$ and twice continuously differentiable,
such that $\partial A \cap U$,  after a rotation,
is the graph of the function $\phi$.
If we assume only that $\phi$ is Lipschitz-continuously differentiable
we say that 
$A$ {\em has a $C^{1,1}$ boundary} (for short,
$\partial A \in C^{1,1}$). Thus if $\partial A \in C^2$
then $\partial A \in C^{1,1}$.

%We say that $A$ {\em has a $C^{1,1}$ boundary} (for short,
%$\partial A \in C^{1,1}$) if
%for each $x \in \partial A$
%there exists a neighbourhood $U$ of $x$ and a real-valued function $f$ that
%is  defined on an open set in $\R^{d-1}$
%and Lipschitz-continuously differentiable, such
%that  $\partial A \cap U$, after a rotation, is the graph of the
%function $f$.
%kkkk

We can now present our main results for the uniform case.
	In all of our results we assume either 
	that $d=2$ and $A= [0,1]^2$, or 
	that $d \geq 2$ and
	$A \subset \R^d$ is compact and connected
	with $\partial A \in C^2$   
	 and %$A$ is the closure of its interior (for short:
	 $A = \overline{A^o}$, where  for
	any $D \subset \R^d$ we let 
 $D^o$ denote the interior of $D$ and  $\overline{D}$ the closure of $D$.  
	Also we assume $r_n \in (0,\infty)$ is given for all $n \geq 1$.
	Let $N(0,1)$ denote a standard normal random variable,
 and for $t \in (0,\infty)$  
 let $Z_t$ be a Poisson random variable with mean $t$.
	Let $\toD$, respectively $\toL$,
	denote convergence in distribution, respectively in the $L^1$ norm.
	Define 
	\begin{align}
		\sigma_A := |\partial A|/(\lambda(A)^{1-1/d});
		~~~
	c_{d,A} := \theta_{d-1}^{-1} (\theta_d/(2-2/d))^{1-1/d} \sigma_A.
		\label{e:defsigmac}
	\end{align}
	The ratio $\sigma_A$ is sometimes called the {\em isoperimetric ratio} of $A$.

\begin{theorem}[Basic results for the uniform case]
	\label{t:basicunif}
	Let $\xi_n$ denote either $K_n-1 $ or $ R_n$,
	and let $\xi'_n$ denote either $K'_n-1$ or $R'_n$.

	(a) Suppose $(r_n)_{n \geq 1}$ satisfy \eqref{e:supcri} and
	\eqref{e:supcriupper}. Then 
	in the uniform case, as $n \to \infty$  we have the
	convergence results: $\mu_n \to \infty$, and
	$(\xi_n/\mu_n) \toL 1 $, and $(\xi'_n/\mu_n) \toL 1$. Also
	$\mu_n^{-1} \Var [\xi'_n] \to 1$, and 
	$\mu_n^{-1/2} (\xi'_n - \E[\xi'_n]) \toD N(0,1)$.
	If $d \geq 3$ then
	$\mu_n^{-1} \Var [\xi_n] \to 1$, and 
	$\mu_n^{-1/2} (\xi_n - \E[\xi_n]) \toD N(0,1)$.

	(b) Suppose instead that $\gamma_n \to \gamma \in \R $
	as $n \to \infty$ (with $\gamma_n$ defined at \eqref{e:supcriupper}.) 
	Then as $n \to \infty$,
	$\xi_n \toD Z_{e^{-\gamma}} $ if $d =2$, and
	$\xi_n \toD Z_{c_{d,A} e^{-\gamma/2}} $ if $d \geq 3$,
	and likewise for $\xi'_n$.
\end{theorem}
Note that if $d \geq 3$ and $\lim_{n \to \infty}(n \theta_d r_n^d /
\log n) = b
\in [0,\infty)$,
then if $b/\lambda(A) < 2/d$ the first term in the right hand side of
\eqref{e:defI'}  dominates, while if $2/d < b/\lambda(A) < 2-2/d$
then the second term in the right hand side of \eqref{e:defI'}
 dominates but we still have $\gamma_n \to -\infty$ from
\eqref{e:supcriupper}.

In Section \ref{s:stateresults} 
we shall provide a more detailed version of Theorem \ref{t:basicunif}:
we shall give estimates on the
rates of convergence, and also
generalize to allow for non-uniformly distributed points in $A$.

To the best of our knowledge, the only previous results on $K_n$ and
$R_n$ in the mildly dense regime 
are  by Ganesan \cite{Gan13} in the special case of $d=2$ and $A=[0,1]^2$, where he proved that there exists a constant $c >0$ such that as $n \to \infty$, 
\begin{align}
\PP[ K_n\le n e^{ - c nr_n^2} ]  \to 1;
~~~~~~
\PP[ R_n\le n e^{ - c nr_n^2} ]  \to 1.
	\label{e:Ganesan}
\end{align}
In other words, the proportionate number of components
and the proportionate number of vertices not in the largest component
decay exponentially in $nr_n^2$
but the exact exponent is not identified;
Ganesan's proof, while ingenious, does not provide much of a clue as to
the optimal value of $c$ satisfying (\ref{e:Ganesan}), or whether this
optimal value is the same for $K_n$ and for $R_n$. Moreover, his proof
of the second part of (\ref{e:Ganesan}) does not appear to generalize to higher
dimensions.

One possible reason why the mildly dense limiting regime was not previously
well understood is an apparently strong dependence between contributions
from different regions of space; one has to look a long way from a given vertex to tell whether it lies in the giant component. A second reason is the importance of boundary effects in this regime and the necessity of dealing with the curved boundary of $A$ quantitatively; note the factor of $|\partial A|$ in
the  definition of $\mu_n $ at \eqref{e:defI'}. Another reason
is that in contrast to the thermodynamic regime, it seems not
 to be possible to re-scale space to obtain a limiting Poisson process
to work with, as was often done in previous works on
these kinds of limit theorems, for example \cite{PY03}.
In Section \ref{ss:overview}
we shall provide an overview of the methods we develop to deal
with these issues.

Our results show that the phenomenon of exponential decay is common to all
dimensions and more general sets $A$, and we identify the optimal value of $c$
in \eqref{e:Ganesan}.
Furthermore, we prove a central limit theorem (CLT) for the fluctuations of
$K'_n$ and $R'_n$ (for all $d \geq 2$) and for
$K_n$, $R_n$ (for $d \geq 3$). 
Our CLT is `weakly quantitative' in the sense that we provide bounds
on the rate of convergence to normal,
although our bounds might not be optimal.
%Previously, the CLT had been
%proved only in the sparse and thermodynamic regimes \cite{Pen03} 
%{\bf (maybe not the sparse regime)}
%with error bounds given in \cite{LrPY20}.  

We expect that our approach can shed some light on the limiting behaviour of higher dimensional homology, and higher Betti numbers,
of random geometric complexes in the mildly dense regime for which
the correct scaling is so far not well understood; see the
last paragraph of 
\cite[Section 2.4.1]{BK22}. This is beyond the scope of this paper and we leave it for future work. 

This paper contains a lot of notation for the reader to keep track of.
To assist with this, we provide an index of notation as an  appendix.

%{\bf Index of notation?}
\section{Statement of results}
\label{s:stateresults}

  We now  describe our setup more precisely.
Let $d \in \mathbb N$ and
%let 
$A\subset\RR^d$.
%be a compact connected set.
Throughout,
we make the following set of assumptions on the pair  $(d,A)$.
 
\begin{assumption}
	\label{a:WA}
	A is compact, connected and nonempty with
	$A = \overline{A^o}$.
Moreover, either  $d \geq 2 $ and  $ \partial A \in C^2$ or $d=2$ and
$A=[0,1]^2$.  
\end{assumption}
Let $f: \RR^d\to [0,\infty)$ a probability density function with support $A$.
Set $f_0:= \inf_{A} f(x)$, $f_1 := \inf_{\partial A}f$, and 
$\fmax:= \sup_{A} f(x)$.
We shall always assume
 that $f_0>0$ and that $f$ is continuous on $A$ 
(so in particular $\fmax < \infty$).
We refer to the special case where $f$ is constant on $A$ as the {\em uniform
case} but in general we allow possibly non-constant $f$.  
%
%We use $\lambda$ to denote the Lebesgue measure on $\RR^d$,
We use $\nu$ to denote the measure with density $f$, i.e. $\nu(dx) = f(x)dx$.
Clearly in the uniform case $f_0 = \lambda(A)^{-1}$. 

%we shall sometimes make the stronger assumption that
%$f$ is constant on $A$.)
%\xy{(bf check this suffices for the literature review part, continuity??)}
 Let
 $(X_1,X_2,\ldots)$ be a sequence of independent random vectors in
 $\RR^d$ with common density $f$, and for $n\in\NN$ set
 $\cX_n:=\{X_1,\ldots,X_n\}$,  which is a binomial point process.
 Also let $(Z_t)_{t >0}$ be a unit intensity Poisson counting process,
 independent of $(X_1,X_2,\ldots)$, so that 
 for $n \in [1,\infty)$,  
  $Z_n$ is a Poisson random variable with mean $n$,
 and set
 $\eta_n:=\cX_{Z_n}$. Then $\eta_n$
 is a Poisson point process with intensity measure $n \nu $.
We use $n$ to denote both the number of points in $\cX_n$ and the average number of points in a Poisson sample $\eta_n$ with the 
convention that $n \in \N$ in the former case and
$n\in [1,\infty)$ in the latter case.

  %\item[(iii)] 
%	  Writing $r(n)$ for $r_n$, define 
%	  \begin{align}
%		I_n:= n \int_A \exp(-n \nu(B_{r(n)}(x))) \nu(dx).
%		\label{e:def:I_n}
%	  \end{align}
%	  If $r_n$ is chosen so $I_n \to 0$ as $n \to \infty$ then
%		by the union bound, 
%		the probability that there are any isolated vertices
%		in $G(\X_n,r_n)$
%		tends to zero. Hence, using
%		(1.5) from \cite[Theorem 1.2(iii)]{PY21} 
%		we can deduce that
%		$\Pr[K_n=1] \to 1$ 
%		(and hence $R_n \overset{\PP}\longrightarrow 0$)
%		as $n \to \infty$.
%		This is called the {\em connectivity regime} {\bf
%		[check terminology]}

 We are concerned with the quantities
 %RGGs
 %$G(\cX_n,r_n)$ and $G(\eta_n,r_n)$; in particular with
$K_n:= K(G(\cX_n,r_n))$ and
$R_n:= R(G(\cX_n,r_n))$ and
 their Poisson counterparts 
$K'_n:=K(G(\eta_n,r_n)$) 
and $R'_n:=R(G(\eta_n,r_n)$),
with $r_n \in (0,\infty)$ specified for each $n$.
%with $r=r(n)$ given.

%{\bf Say something about persistence diagrams?}

%\subsection{Terminology}

%We say $D$ has {\em weakly smooth boundary}
%{\bf [never used?]} if
%$\lambda(D\setminus D^{(s)})= O(s)$ as $s\to 0$, where $D^{(s)}:= \{x \in D: 
%B_s(x)\subset D\}$.
%is defined to be the collection of   $x\in D$ distant at least $r$ from the boundary $\partial D$ of $D$.
%Note that that if $\partial D \in C^2$, then $\partial D$ it is weakly smooth.
%If $D$ is a polytope, its boundary is weakly smooth but not in  $C^2$.
%strongly smooth,
%in our terminology.

Given $g: (0,\infty)\to \R$,
and $h: (0,\infty)\to (0,\infty)$,
we write $g(x) = O(h(x))$ 
if we have $\limsup |g(x)|/h(x)<\infty$, and write $g(x)=o(h(x))$ if
$\limsup |g(x)|/h(x)=0$, $g(x) = \Omega(h(x))$ if $\liminf(g(x)/h(x)) >0$.
We write
$g(x) = \Theta(h(x))$ if both $g(x) = O(h(x))$ and $g(x)=\Omega(h(x))$, 
and  $g(x) \sim h(x) $ if $\lim(g(x) / h(x)) = 1$.
Here, the limit is taken either as $x\to 0$ or $x\to \infty$, to be specified in each appearance. 

To present quantitative CLTs, we recall that
for random variables $X,Y$,
%three notions of distance which measure the proximity between real-valued random variables. 
the Kolmogorov distance $\dk$
and the total variation distance $\dtv$  
between them
are defined respectively by
\begin{align*}
\dk(X,Y) := \sup_{z\in\RR} |\PP[X\le z] - \PP[Y\le z]|; 
~~~
%\\
\dtv(X,Y) := \sup_{A\in \mathcal B(\RR)} |\PP[X\in A] - \PP[Y\in A]|,
\end{align*}
where the second supremum is taken over all Borel measurable subsets of $\RR$.
Note that
 convergence in the Kolmogorov distance implies convergence in
 distribution.

 \subsection{Results for general $f$}
 \label{ss:Resultsgenf}

We now give our results for the component count 
and the number of vertices not in the giant component
in the general case with $f$ not assumed necessarily to be constant on
$A$.  
For  general $f$, instead of $\mu_n$ defined at \eqref{e:defI'},
our results refer to constants $I_n$ defined by
	  \begin{align}
		I_n:= n \int_A \exp(-n \nu(B_{r_n}(x))) \nu(dx).
		\label{e:def:I_n}
	  \end{align}
		We define the 
		{\em critical regime for connectivity}
		to be when $r_n$ is chosen so that $I_n=\Theta(1)$ as 
		$n \to \infty$, and the
		\emph{mildly dense} regime to be when $r_n$ is chosen so that
(\ref{e:supcri}) holds but also $I_n \to \infty$ as $n \to \infty$.
As discussed later in Remark \ref{rk:unif1},
the latter condition turns out to be equivalent to \eqref{e:supcriupper} in
the uniform case.

\begin{theorem}[First order moment asymptotics for general $f$]
	\label{t:momLLN}
	Suppose $(d,A)$ satisfy Assumption \ref{a:WA}, 
	that $f$ is continuous on $A$ with $f_0 >0$,
	and that $r_n$ satisfies \eqref{e:supcri} and also $I_n \to \infty$
	as $n \to \infty$.
	Let $\xi_n$ denote any of $K_n-1$, $R_n$,
	%and let $\xi'_n$ denote either 
	$K'_n-1$ or $R'_n$, and let $\zeta_n$ be either $\xi_n$
	or $\xi_n +1$. Then as $n \to \infty$
	we have
	\begin{align}
		\E[\xi_n] = I_n (1+ O((nr_n^d)^{1-d}));
		 \label{e:meanasymp}
		\\
		\E[| 	(\zeta_n /I_n) -1 |] =  O((nr_n^d)^{1-d} + I_n^{-1/2}).
		\label{e:basicLLN}
		\end{align}
		In particular $(\zeta_n/I_n) \toL 1$.
	\end{theorem}
We can use the $L^1$ convergence in Theorem \ref{t:momLLN}, together with
an asymptotic analysis of $I_n$, to determine the optimal exponent $c$ 
in Ganesan's result \eqref{e:Ganesan}.
First we introduce some further notation.  Given $(r_n)_{n \geq 1}$ we define
	\bea
	b^+ := \limsup_{n \to \infty} (n \theta_d r_n^d/\log n); 
	~~~~~~~~~~~
	b^- := \liminf_{n \to \infty} (n \theta_d r_n^d/\log n).
	\label{e:defb}
	\eea
	and $b:=  b^+= b^-$ whenever $b^+=b^-$. Loosely speaking,
	$b$  is the logarithmic
	growth rate of the degree of a typical vertex, at least
	in the uniform case with $\lambda(A)=1$. 
	We identify two critical values for $b$, namely
		\begin{align}
			b_c := \max \Big( \frac{1}{f_0}, 
			\frac{2-2/d}{f_1} \Big); 
			~~~~~
			b'_c := \begin{cases}
				(d(f_0 - f_1/2))^{-1}  & {\rm ~if ~ }
				f_0 > f_1/2;
				\\
				+ \infty & {\rm ~ if} ~ f_0 \leq f_1/2
			\end{cases}
			\label{e:bcdef}
		\end{align}
		 (so in the uniform case $b_c = (2-2/d)/f_0$ and $b'_c
		= 2/(df_0)$, and hence $b'_c < b_c$ if $d \geq 3$).
		The following result shows $b_c$ is the critical value of
		the logarithmic growth rate $b$ above which
		$I_n \to 0$, and below which $I_n \to \infty$.
	\begin{proposition}
		\label{p:bc}
		If $b^+ < b_c$ then $I_n \to \infty$ as $n \to \infty$.
		Conversely, if $b^- > b_c$ then 
		$I_n \to 0$ as $n \to \infty$, and
		if $ \liminf_{n \to \infty} I_n >0 $ then $b^+ \leq b_c$.
	\end{proposition}
	The next result arises from the fact that
	%if $b'_c < b_c$, then 
	for $b < b'_c$ the main contribution to $I_n$,
	and hence to $K_n$ or $R_n$, comes
	from the interior of $A$, while for $ b'_c < b < b_c$
	the main contribution to $I_n$ comes from near the boundary of $A$.
	Given random variables $(Y_n)_{n \geq 1}$ we write
	$Y_n = o_\Pr(1)$ to mean %$Y_n \toP 0$ (i.e.,
	$Y_n \to 0$ in probability as $n \to \infty$. 
	\begin{theorem}
		\label{t:exprate}
		Under the conditions of Theorem \ref{t:momLLN},
		%Suppose $nr^d \to \infty $ and  $I_n  \to \infty$.
		%Let $\zeta_n$ be either $K_n$ or $R_n$.
		%If also $b^+ \leq b_c$ then
		as $n \to \infty$
		we have
	%	Finally, if $d=2$ then 
		%we have the convergence in
		%probability
		\begin{align}
		%	(nr^d)^{-1} \log (\zeta_n/n) \toP - \theta f_0 
		%	~~~~ {\rm i.e.}~~
			\zeta_n & = 
			n \exp(-  \theta_d f_0 n r_n^d (1+ o_\Pr(1)))
			~~~~ & {\rm if}~~ b^+ \leq b'_c;
			\label{e:loginP1}
			\\
		%\end{align}
		%while if $ b^- > b'_c $ then
		%\begin{align}
			%(nr_n^d)^{-1} \log (\zeta_n/n^{1-1/d}) \toP 
			%- \theta f_1/2 
			\zeta_n & = 
			n^{1-1/d} \exp \Big(-  \theta_d f_1 n r_n^d 
			\big(\tfrac12+ o_\Pr(1)\big) \Big)
			~~~~ & {\rm if}~~ b^- \geq b'_c.
			\label{e:loginP2}
		\end{align}
		In particular, if $b^+ = b^- = b$ then
		$\zeta_n = n^{ 1-  \min(f_0 b,(1/d) + f_1b/2) +o_\Pr(1)}$
		.
	\end{theorem}
	If $d=2$ then since $f_1 \geq f_0$ we have $ d (f_0 -f_1/2)
	\leq f_0$ and $ b'_c \geq f_0^{-1} = b_c$.
Thus, if $nr_n^2 \to \infty$ and $I_n \to \infty$,
then \eqref{e:loginP1} applies and
Ganesan's result (\ref{e:Ganesan}) 
for $A=[0,1]^2$
holds whenever $c< \pi f_0$, and fails whenever
$c> \pi f_0$ (in the latter case the probabilities in (\ref{e:Ganesan})
tend to zero). A similar remark holds when $d \geq 3$, provided also
 $b^+ < b'_c$.

Next we give distributional results. 
The first one
says that in the critical regime for connectivity,
both $K_n-1$ and $R_n$ (along with $K'_n -1$ and $R'_n$)
are asymptotically Poisson.

\begin{theorem}[Poisson convergence in the connectivity regime for
	general $f$]
	\label{t:PoLim}
	Suppose that Assumption \ref{a:WA} applies,  
	$f$ is continuous on $A$ with $f_0 >0$, and  %suppose
	that $I_n = \Theta(1) $ as
		$n \to \infty$.
		Let $\xi_n$ denote any of  $K_n-1$, $R_n$ 
	$K'_n-1$ or $R'_n$. Then
	$\dtv(\xi_n, Z_{I_n}) = O((\log n)^{1-d})$
	as $n \to \infty$. In particular, if
	$\lim_{n \to \infty} I_n = c$ 
		for some $c \in (0,\infty)$, then
	$\xi_n \toD Z_c$ as $n \to \infty$.
\end{theorem}

Our next result demonstrates asymptotic
normality of $K'_n $ and of $R'_n$ for $d \geq 2$, and of $K_n$ and
$R_n$ for $d \geq 3$, in the whole of the mildly dense limiting
regime for $r$.

\begin{theorem}[Variance asymptotics and CLT for general $f$]
	\label{t:momCLT}
	Suppose that Assumption \ref{a:WA} applies,
	$f$ is continuous on $A$ with $f_0 >0$,
	and that $r_n$ satisfies \eqref{e:supcri} and also $I_n \to \infty$
	as $n \to \infty$.
	Let $\xi_n$ denote either $K_n-1$ or $R_n$;
%	and 
	let $\xi'_n$ denote either 
	$K'_n-1$ or $R'_n$. Then as $n \to \infty$
	we have 
	\begin{align}
		 \Var[\xi'_n] &= I_n (1+ O((nr_n^d)^{(1-d)/2})); 
		 \label{e:VarPo}
		 \\
		 \dk(I_n^{-1/2} (\xi'_n - \E[\xi'_n]),
		N(0,1)) & = O((nr_n^d)^{(1-d)/3} + I_n^{-1/2}).
	%	\sim \E[\xi'_n ] \sim \Var[\xi'_n] \sim I_n,
		%{\rm if~} d\geq 3,
		\label{e:CLTPo}
	\end{align}
	If $d \geq 3$ then also
		%{\rm if~} d\geq 3,
	%If $\xi_n $ is either $K_n-1$ or $R_n$ then
	\begin{align}
		\Var[\xi_n] & = I_n (1+ O(nr_n^d)^{1-d/2});
		%\Var[\xi_n] = I_n (1+ O(nr_n^d)^{1-d/2}).
		\label{e:EVarsim}
		\\
		\dk(I_n^{-1/2} (\xi_n - \E[\xi_n]),
		N(0,1)) & = O((nr_n^d)^{(2-d)/3} + I_n^{-1/2}).
	%	\sim \E[\xi'_n ] \sim \Var[\xi'_n] \sim I_n,
		\label{e:CLTbin}
	\end{align}
	%If also $d \geq 3$ or $\xi_n = S_n$
	%then also $\Var[\xi_n ] \sim I_n$ and
	% $(\xi_n - \E[\xi_n] )/\sqrt{I_n} \toD N(0,1)$ as $n \to \infty$.
%
 %and if $d \geq 3$ or $\xi_n= S_n$, likewise for $\xi_n$.
\end{theorem}

\begin{remark}
	\label{rk:nonunif}
\begin{enumerate}
	\item
		In view of Theorem \ref{t:PoLim},
		our limiting regime for $r_n$ for Theorems  
	 \ref{t:momLLN},
		\ref{t:exprate}
		and 
		\ref{t:momCLT}
		(namely, $nr_n^d \to \infty$ and $I_n \to \infty$)
		covers the whole range of limiting regimes
		between the thermodynamic regime and the critical
		regime for connectivity.

	\item
		It should be possible to relax the
		condition $\partial A \in C^2$
		to $\partial A \in C^{1,1}$ in all of our results.
		This would involve similarly relaxing the conditions for
	%	making a similar improvement to
		%certain
		%results from other papers that were stated under the $C^2$
		%condition which we use in our proofs here, in particular
		%\cite[Lemma 2.5]{PY21} and
		\cite[Lemma 3.5]{Pen99} which is used in the proof of Lemma
\ref{l:uniqueness} here. This should be possible using ideas from the proof
		of Lemma \ref{l:E9} here. All of the other lemmas in which we use the $C^{2}$ condition hold under the weaker $C^{1,1}$ condition. 
\end{enumerate}
\end{remark}

\subsection{Results for the uniform case}
\label{ss:ResUnif}

 In the uniform case,  we can replace $I_n$ with 
 the quantity $\mu_n$ defined at \eqref{e:defI'}.
 Indeed, in Propositions 
 \ref{p:average2d} and \ref{p:average3d+}
 we shall show 
 that in the uniform case, $I_n = \mu_n(1+ O((nr_n^2)^{-1/2}))$ 
as $n \to \infty$
 if
 $d =2$ and
 $I_n = \mu_n \Big(1+ O\Big( \big(\frac{\log (nr_n^d)}{nr_n^d}\big)^2 \Big) \Big)$
as $n \to \infty$
 if $d \geq 3$. Therefore in the convergence results arising from
 Theorems \ref{t:momLLN}, \ref{t:PoLim}
 and \ref{t:momCLT} we can replace $I_n$ with $\mu_n$;
 this gives us Theorem \ref{t:basicunif}. The more quantitative
 versions of the results in Theorem \ref{t:basicunif}, 
 where we keep track of rates of
 convergence, go as follows.
\begin{theorem}[First order results for the uniform case]
	\label{t:PoLimUn}
	Suppose that Assumption \ref{a:WA} applies, and that
	$f \equiv f_0 1_A$ with $f_0=\lambda(A)^{-1}$.
	Let $\xi_n$ denote any of $K_n-1$, $R_n$, $K'_n-1$ or $R'_n$, 
	define $\gamma_n$ as at \eqref{e:supcriupper} 
	and define $\mu_n $ by \eqref{e:defI'}. 

	(a) If $(r_n)_{n \geq 1}$ satisfies
	$|\gamma_n |= O(1)$ as $n \to \infty$,
	then $\dtv(\xi_n,Z_{\mu_n}) = O((\log n)^{-1/2})$
	if $d=2$ and
	 $\dtv(\xi_n,Z_{\mu_n}) =
	 O\Big(\big(\frac{\log \log n}{\log n})^{2}\Big)$
	 if $ d \geq 3$.

	(b)
	If $\gamma_n \to \gamma \in \R$ as $n \to \infty$,
	then as $n \to \infty$,
	$\xi_n \toD Z_{e^{-\gamma}} $ if $d =2$ and
	$\xi_n \toD Z_{c_{d,A} e^{-\gamma/2}} $ if $d \geq 3$,
	where $c_{d,A}$ is defined at \eqref{e:defsigmac}.
%
%	where we set $c_{d,A} := \theta_{d-1}^{-1} 
%	(\theta_d/(2-2/d))^{1-1/d} \sigma_A$
%	and $\sigma_A := |\partial A|/(\lambda(A)^{1-1/d})$.

	(c) 
	Suppose \eqref{e:supcri} and \eqref{e:supcriupper} hold. If $d=2$
	then as $n \to \infty$:
	\begin{align}
		\E[\xi_n] = \mu_n (1+ O((nr_n^2)^{-1/2}));
		 \label{e:EPoUn2}
		 \\
		 \E \Big[ \Big| \frac{\xi_n}{\mu_n} - 1 \Big|  \Big] 
		 = O ( (n r_n^2)^{-1/2} + \mu_n^{-1/2} ),
		 \label{e:L1Un2}
	\end{align}
	while if $d \geq 3$ then as $n \to \infty$:
	\begin{align}
		\E[\xi_n] = \mu_n
		\Big(1+ O\Big(\big( \frac{\log(nr_n^d)}{
			nr_n^d} \big)^{2}\Big) \Big);
	%	~~~~
	%	 \E[\xi'_n] = 
	%	 \mu_n
	%	\Big(1+ O\Big(\big( \frac{\log(nr_n^d)}{
	%		nr_n^d} \big)^{2}\Big) \Big);
			\label{e:EUn3+}
		\\
		 \E \Big[ \Big| \frac{\xi_n}{\mu_n} - 1 \Big|  \Big] 
		  = O \Big( \big( \frac{\log(n r_n^d)}{nr_n^d} \big)^{2} 
		  + \mu_n^{-1/2} \Big).
		 \label{e:L1BiUn3+}
%		  \\
%		 \E \Big[  \Big| \frac{\xi'_n}{\mu_n} - 1 \Big|  \Big] 
%		  = O \Big( \big( \frac{\log(n r_n^d)}{nr_n^d} \big)^{2} 
%		  + \mu_n^{-1/2} \Big);
%		 \label{e:L1PoUn3+}
	\end{align}
\end{theorem}

\begin{theorem}[Variance asymptotics and CLT for the uniform case]
\label{t:momCLTunif} 
	Suppose that Assumption \ref{a:WA} applies, and that
	$f \equiv f_0 1_A$ with $f_0=\lambda(A)^{-1}$.
	Suppose $r_n$ satisfies 
	\eqref{e:supcri} and \eqref{e:supcriupper},
	and define $\mu_n $ by \eqref{e:defI'}. 
	Let $\xi_n$ denote either $K_n$ or $R_n$, and
	let $\xi'_n$ denote either $K'_n$ or $R'_n$.
	If $d=2$ then as $n \to \infty$:
	\begin{align}
		%\E[\xi_n] = \mu_n (1+ O((nr_n^2)^{-1/2})); ~~~~
		% \E[\xi'_n] &= \mu_n (1+ O((nr_n^2)^{-1/2}));
		% \label{e:EPoUn2}
		% \\
		 \Var[\xi'_n] &= \mu_n (1+ O((nr_n^2)^{-1/2}));
		 \label{e:VPoUn2}
		 \\
		 \dk(\mu_n^{-1/2} (\xi'_n - \E[\xi'_n]),
		N(0,1)) & = O((nr_n^2)^{-1/3} + \mu_n^{-1/2}).
		\label{e:CLTPoUn2}
	\end{align}
	If $d \geq 3$ then as $n \to \infty$:
	\begin{align}
		 \Var[\xi'_n] = \mu_n \Big(1+ O\Big(
		 (nr_n^d)^{(1-d)/2}  + \big( \frac{\log (nr_n^d)}{
			 nr_n^d} \big)^2 \Big) \Big);
			 \label{e:VarPoUn3+}
			 \\
		 \Var[\xi_n] = \mu_n \Big(1+ O\Big(
		 (nr_n^d)^{1-d/2}  + \big( \frac{\log (nr_n^d)}{
			 nr_n^d} \big)^2 \Big) \Big);
			 \label{e:VarBiUn3+}
			 \\
		 \dk(\mu_n^{-1/2} (\xi'_n - \E[\xi'_n]),
		N(0,1))  = O\Big( (nr_n^d)^{(1-d)/3} +
		\big( \frac{ \log (nr_n^d)}{nr_n^d} \big)^{4/3}
		+ \mu_n^{-1/2}\Big);
		\label{e:CLTPoUn3+}		 \\
		 \dk(\mu_n^{-1/2} (\xi_n - \E[\xi_n]),
		N(0,1))  = O\Big((nr_n^d)^{(2-d)/3} +
		\big( \frac{ \log (nr_n^d)}{nr_n^d} \big)^{4/3}
		+ \mu_n^{-1/2}\Big).
	%	\sim \E[\xi'_n ] \sim \Var[\xi'_n] \sim I_n,
		%{\rm if~} d\geq 3,
		\label{e:CLTbiUn3+}
	\end{align}
%
%	(a) we have $\E[\xi_n] \sim \E[\xi'_n ] \sim \mu_n$.
%	Also $\Var[\xi'_n ] \sim \mu_n$ and if $d \geq 3$ then
%	 $\Var[\xi_n ] \sim \mu_n$;
%	 %{\bf [Result on $\E[K_n]$ may need $b^+ < b_c$.]}
%
%
%	(b) also $(\xi'_n-\E[\xi'_n])/\sqrt{\mu_n} \toD
%	N(0,1)$, and  if $d \geq 3$ then
%	$(\xi_n-\E[\xi_n])/\sqrt{I_n} \toD N(0,1)$;
%
\end{theorem}

\begin{remark}
	\label{rk:unif1}
	\begin{enumerate}
		\item
		In the uniform case, we have $f_0=f_1$ and
		$b_c =
		(2-2/d)/f_0 $.

\item We can often simplify the expression \eqref{e:defI'} for $\mu_n$
depending on the logarithmic growth rate of $nr_n^d$.  Indeed,  if $d=2$
or $b^+ f_0 < 2/d$ then $\mu_n \sim ne^{-n \theta_d f_0 r_n^d}$,
while if $d \geq 3$ and $b^-f_0  >2/d$ then $\mu_n \sim \theta_{d-1}^{-1}
		|\partial A| r_n^{1-d} e^{-n \theta_d  f_0 r_n^d/2}$.

	\item
		From Theorem %\ref{t:momCLTunif} 
			\ref{t:PoLimUn}
			we see for the 
		uniform case that 
		in the whole of the mildly dense regime both $K_n$ and
		$R_n$ scale like $\mu_n$
		(and if  $d=2$ or $f_0 b^+ < 2/d$,
		like $n\exp(-nf_0 \theta_d r_n^d)$)
		in probability, rather than like a constant times $n$ 
			as given by \eqref{0616a} in the thermodynamic regime. 
\item
In the uniform case, if $nr_n^d \to \infty$ and $\limsup \gamma_n < \infty$
			as $n \to \infty$, then by Propositions 
 \ref{p:average2d} and \ref{p:average3d+}, $I_n \sim \mu_n$ as $n \to \infty$.
Hence by \eqref{e:mulim}, if $\gamma_n \to \gamma \in \R$ as  $n \to \infty$,
then $I_n \to e^{-\gamma}$ if $d=2$ and $I_n \to c_{d,A} e^{-\gamma/2}$
			 if $d\geq 3$.
	%	$$
	%		I_n \to g_d(\gamma) := 
	%		\begin{cases}
	%	e^{-\gamma} ~{\rm if}~ d=2\\
	%	c_{d,A} e^{-\gamma/2} ~{\rm if}~ d\geq 3.
	%\end{cases}
	%		$$
Using this and the fact that $I_n$ is decreasing in $r_n$ while
$\gamma_n$ is increasing in $r_n$ we can deduce that $\gamma_n \to 
- \infty$ if and only if $I_n \to \infty$, as claimed earlier.
	\end{enumerate}
\end{remark}

	\subsection{Overview of proofs}
	\label{ss:overview}

The main insight behind our results is that
the dominant contribution
 both for $K_n-1$ and for $R_n$ 
comes from the  singletons, i.e. the isolated vertices.
Let $S_n$, respectively $ S'_n$, denote the number of singletons
of $G(\cX_n,r)$, resp. $G(\eta_n,r_n)$. 
Our starting point for a proof of Theorem \ref{t:momCLT}
is a similar collection of results for $S_n$ and $S'_n$,
of interest in their own right, which go as follows:
\begin{proposition}[Results on singletons]
	\label{p:Sall}
	Suppose $f$ is continuous on $A$ with $f_0 >0$,
	and that $r_n$ satisfies \eqref{e:supcri} and also $I_n \to \infty$
	as $n \to \infty$. Let $\zeta_n$ be either $S_n$ or $S'_n$.
	Then there exists $ \delta >0$ such that
	%Then $\E[S'_n] = I_n$, and 
	as $n \to \infty$ we have
	$\E[\zeta_n] = I_n (1+ O( e^{-\delta nr_n^d}))$, and
	$\Var[\zeta_n] = I_n(1+ O(e^{-\delta nr_n^d}))$, 
	%and
	%$\Var[S_n] = I_n(1+ e^{-\delta nr^d})$.
	and also 
%	$$\dk(I_n^{-1/2}(\zeta_n-\E[\zeta_n]),N(0,1)) =
%	O(e^{-\delta nr_n^d} + I_n^{-1/2}).
%	$$
%	and
	\begin{align}
		& \dk(I_n^{-1/2}(S'_n - I_n), N(0,1)) = O(e^{-\delta n r_n^d}
		+ I_n^{-1/2});
		\label{e:CLTS'}
		\\
		{\rm if} ~ \partial A \in C^2, ~~~& \dk(\tilde{I}_n^{-1/2}
		(S_n - \E[S_n]), N(0,1)) = O(e^{-\delta nr_n^d} + I_n^{-1/2}).
		\label{e:CLTS}
	\end{align}
\end{proposition}
%Most of the work for proving
Proposition \ref{p:Sall} extends results  
% was already done
in \cite{PY23}, 
%although the results in that paper 
where the same conclusions are derived 
%are stated 
under the extra condition $b^+ < 1/\max(f_0,d(f_0-f_1/2))$
rather than the weaker condition
$I_n \to \infty$ that we consider here.

To get from Proposition \ref{p:Sall} to Theorem \ref{t:momCLT},
let $\xi_n$ be either $K_n-1$ or $R_n$ and $\xi'_n$ be either
$K'_n -1$ or $R'_n$. We show that both the mean and the variance
of both $\xi_n-S_n$ and $\xi'_n -S'_n$, are asymptotically negligible
relative to $I_n$. To do this we deal separately with the contribution
to $\xi_n - S_n$ or $\xi'_n -S'_n$ from components with
Euclidean diameters that are 
categorized as `small', `medium' or `large' compared
to $r_n$, using different arguments for the three different categories.
This requires us to deal with a lot of different cases, as a result
of which Section
%\ref{s:moments},
\ref{s:asympvar},
containing the second moment  estimates, is quite long
(the proofs of the first order results can be read without referring
to that section).
Once we have the moment estimates, we can derive the `quantitative' CLT for
$\xi_n$ or $\xi'_n$  from the one for $S_n$ or $S'_n$ by
using a quantitative version of Slutsky's
theorem. 

Our argument for small components has geometrical ingredients 
(presented in Section \ref{ss:geom})
and takes boundary effects into account.
The argument for large components involves discretization and path-counting
arguments seen in continuum percolation theory. The argument
for medium-sized components involves both geometry and discretization.

To derive our results with more explicit constants in the uniform case
(Theorems \ref{t:PoLimUn} and \ref{t:momCLTunif}) we need to
demonstrate asymptotic equivalence of $I_n$ and $\mu_n$. 
%This can be done by adapting the method of polytopal approximation
%that was recently used for related problems in \cite{HPY24}, \cite{Pen22} and
%\cite{PY21}.
We do this in Section
\ref{s:Iunif}, by approximating the integrand for $I_n$ by a function
of distance to the boundary only, and using a result
from \cite{HPY24} (Lemma \ref{thm:cov} here)
to approximate the integral of such an integrand by a constant times
a one-dimensional integral.

The rest of the paper is organised as follows. 
After some preliminary lemmas in Section \ref{s:prelims},
	in Section \ref{s:sngltn}
we give an asymptotic analysis of $S_n$ and $S'_n$, and of
	of $I_n$,
	in particular proving Propositions 
	\ref{p:bc}
and	\ref{p:Sall}.

In Section \ref{s:moments} we give estimates of $\E[\zeta_n -S_n]$ and
$\E[\zeta'_n - S'_n]$, where $\zeta_n$ is either $K_n-1$ or $R_n$,
and $\zeta'_n$ is either $K'_n-1$ or $R'_n$,
and then conclude the proof of
Theorems \ref{t:momLLN},
\ref{t:exprate} and
	\ref{t:PoLim}.
In Section   \ref{s:asympvar}
we complete the proof of Theorems 
	\ref{t:momCLT} 
and \ref{t:momCLTunif}.

Compared to our  earlier paper \cite{PY23},
our asymptotic analysis of $I_n$ here in the uniform case (in Section 
\ref{s:Iunif}) requires a more careful treatment of boundary effects,  
since here we consider the whole of the intermediate limiting regime
\eqref{e:supcri}, \eqref{e:supcriupper} whereas the results in
\cite{PY23} are derived under an extra condition amounting to
	$b^+ < 1/\max(f_0, d(f_0-f_1/2))$; without this extra condition
	the boundary effects can dominate and take more work to deal with.
For our bounds on $ \xi_n - S_n$ or $\xi'_n - S'_n$, the methods of
\cite{PY23} are not much use because they deal only with clusters
of fixed order, whereas here we need to deal with all orders of cluster 
at once. Some of the ideas in  \cite{PY21} are of use for this,
but analysis of second order asymptotics of these 
quantitites is not required in the limiting regime of
\cite{PY21}, and to  deal with these in our situation
we have developed ways to represent these second moments as multiple integrals. 
These methods may well be useful in other contexts, such as a similar 
analysis of higher order Betti numbers in TDA, or of the number of components
of the vacant region $\R^d \setminus \cup_{x \in \cX_n} B_{r_n}(x)$,
  in the mildly dense regime.

%\xy{(we'll see if we can handle $rn^d = c_1\log n  + c_2 \log_2 n - \infty$ maybe there is boundary vs bulk competition in the regime $c_0 \log n< nr^d c_1\log n  + c_2 \log\log n - \infty$)}

\section{Preliminaries}
\label{s:prelims}

Throughout the rest of this paper we assume 
that $(d,A)$ satisfy Assumption \ref{a:WA}
and $f$ is continuous on $A$ with $f_0 >0$.
%that the density $f$ of the measure $\nu$
%is supported by a compact connected set $A \subset \R^d$ with
%$(d,A)$ satisfying Assumption \ref{a:WA},  that
%$f_0:= \mathrm{inf}_A f >0$, and
%that $f$ is continuous on $A$, so that
% $\fmax := \mathrm{sup}_A f $ is finite.
	%Moreover we assume either 
	%that $d=2$ and $A= [0,1]^2$, or 
	%that $d \geq 2$ and
	%$A \subset \R^d$ is compact and connected with
	%$A = \overline{A^o}$
	%and $\partial A \in C^2$.   
Also we assume $r_n \in (0,\infty) $
 is given for all $n \geq 1$.

Given $D, D' \subset \R^d$, we set $D \oplus D' := \{x+y:x \in D, y \in D'\} $,
the Minkowski sum of $D$ and $D'$. 
Let $o$ denote the origin in $\R^d$.
Let $\|\cdot \|$ denote
the Euclidean norm on $\R^d$.
Given $a>0$ we set $aD := \{ax:x \in D\}$. 
Also we set $D^{(-a)}:= \{x \in D:B(x,a) \subset D\}$.
Given $n \in \N$, we write $[n]$
for the set $\{1,\ldots,n\}$. 

We introduce an ordering $\prec$ on $A$: for $x,y \in A$,
if $\partial A \in C^2$ we
say $x \prec y$ if either $\dist(x,\partial A) < \dist (y,\partial A)$
(using the Euclidean distance) 
or $\dist(x,\partial A) = \dist(y,\partial A)$
and $x$ precedes $y$
(strictly) in the lexicographic ordering. If $A=[0,1]^2$
we say $x \prec y$ if the $\ell_1$ distance from $x$ to the nearest
corner of $A$ is less than that of $y$, or if these two distances 
are equal and $x$ precedes $y$ lexicographically.
In either case, given $x \in A$ we write $A_x := \{y \in A: x \prec y\}$.

For non-empty $U \subset \R^d$ set $\diam(U) := \sup_{x,y \in U}\{ \|x-y\|\}$,
and let
$\#(U)$ denote the number of elements of $U$.

\subsection{Geometrical and combinatorial tools}
\label{ss:geom}

\begin{definition}[Sphere condition]
	\label{d:SC}
	Suppose $\partial A \in C^2$.
        For $z \in \partial A$ let $\hat n_z$ be the unit normal to $\partial A$ at $z$ pointing inside $A$.

        Given $\tau \geq 0$,
        let us say $\tau $ satisfies the {\em sphere condition}
         for $A$ if, for all $x \in \partial A$,
        we have $B(x+ \tau \hat n_x , \tau) \subset A$
        and  $B(x -\tau \hat n_x, \tau) \cap A = \{x\}$.

Let $\tau(A)$ denote the supremum of the set of all $\tau$ satisfying
the sphere condition for $A$.
\end{definition}
\begin{lemma}[Sphere condition lemma]
        Suppose $\partial A \in C^2$. Then $\tau(A) >0$; that is,
        there exists a constant $\tau >0$ such that
        $\tau$ satisfies the sphere condition
         for $A$.
\end{lemma}
\begin{proof}
        See     \cite[Lemma 7]{sphere-condition-lp}.
        %{\bf [more detailed reference?]}
\end{proof}

\begin{lemma}
	\label{l:slice}
	Suppose $\partial A \in C^2$.
	Let $\tau (A) $ be as in Definition \ref{d:SC},
	and suppose  $0 < r < \tau <   \tau(A)$.
	Let $x \in A \setminus A^{(-r)}$.
	Let $\pi(x)$ be the nearest point to $x$ in $\partial A$.
	%Let $e $ be the unit normal vector from $\pi(x)$ that
	%is orthogonal to the tangent plane of $ \partial A$ at $\pi(x)$
	%and points away from $A$.
	Then for any $y \in B_r(x)$: if $(y-\pi(x)) \cdot \hat{n}_{\pi(x)} 
	> r^2/\tau$ then $y \in A$, and  if $(y-\pi(x)) \cdot \hat{n}_{\pi(x)}
	< - r^2/\tau$ then $y \notin A$.
\end{lemma}
\begin{proof}
	Without loss of generality $\pi(x)$
	%the closest point on the boundary to 
	%$x$
	is the origin $o $ and $\hat{n}_x = (0,0,\ldots, 1) =:e_d$, the 
	$d$-th coordinate vector.
	Then $x =  ae_d $ for  some $a  \in  [0, r ).$
	Let $\mathbb H := \{y \in  \mathbb R^d : y \cdot e_d \geq 0\}$,
	the upper half-space.
	%and note that $|B r (x) ∩ H| = ((θ d /2) + h(a/r ))r d , the volume we are using to approximate |B r (x) ∩ A|.
	%Set $\tau = \tau(A)/2$, so that
	%$\tau \in (r , \tau (A))$. 
	Let $S := (B_\tau (\tau e_d ))^o$ and $S' := (B_\tau (-\tau e_d ))^o$. 
	Then $\partial A \cap B_r(x)$ is trapped between the balls $S$ and $S'$,
	and the set $B_r (x) \cap (A \triangle  \mathbb H)$
	is contained in $\mathbb R^d \setminus (S \cup S' )$. Therefore
	by some spherical geometry, it is
contained in a cylinder $C$ centred on $o$ of radius $r$ and height $2s$,
	as illustrated in Figure \ref{f:hourglass},
	%\cite[Fig. 1]{HPY24},
%
  \begin{figure}[!h]

\center
\includegraphics[width=5cm]{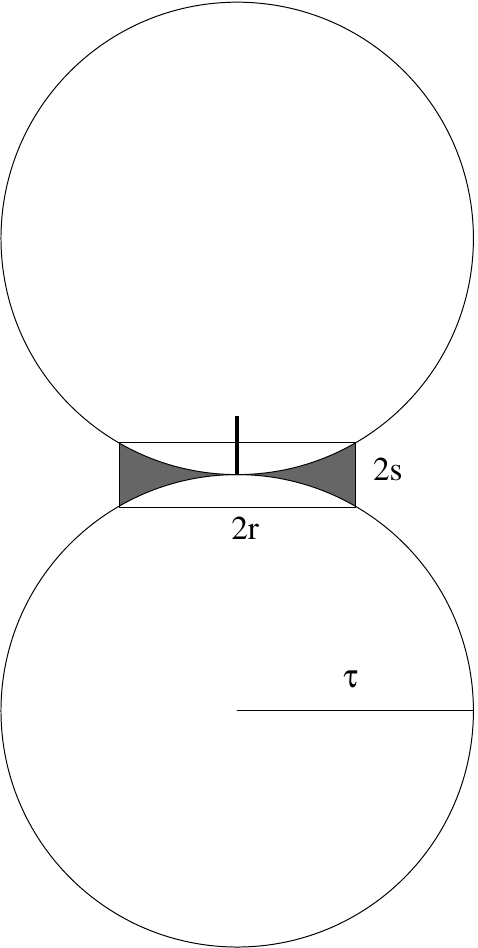}
%\includegraphics[width=1.2\textwidth]{pic1.pdf}
                %\vspace{-2cm}
                \caption{\label{f:hourglass}Illustration for proof of Lemma
		\ref{l:slice}. The circles meet at $o$, and $x$ lies on the vertical line segment (of length $r$).
		The set $B_r(x) \cap (A  \triangle
               % (B_r(x) \cap
		\mathbb{H})$ is contained in the shaded region.
        }
\end{figure}
	with $s$ chosen so $s \leq r$  and $(\tau - s)^2 + r^2 = \tau^2$, so
	$2\tau s = r^2 + s^2 \leq 2r^2$, and hence $s \leq r^2 /\tau $.
	The required conclusion folows from this.
%Thus |C| ≤ 2θ d−1 r d+1 /τ , and Eq. 3.2 follows by letting τ ↑ τ (A).
\end{proof}

For $x\in A$ let $a(x): = \dist (x,\partial A)$,
the Euclidean distance from $x$ to $\partial A$.
For $s \geq 0$ let $g(s) := \lambda(B_1(o) \cap ([0,s] \times \R^{d-1}))$.
For
%small nonnegative $r$, and
$x \in A \setminus  A^{(-s)}$, the next lemma approximates
$| B_s(x) \cap A | $ by $
(\frac{1}{2}\theta_d + g(a(x)/s))s^d$,
which is the volume of the portion of $B_s(x)$
that is not cut off from $x$ by the 
%which lies on one side of the
tangent hyperplane to $\partial A$
at $\pi(x)$.
%the closest point to $x$ on $\partial A$.

\begin{lemma}
        \label{l:bdyvol}
Suppose $d \geq 2$ and $\partial A \in C^2$.
	There is a constant $\tau(A) >0$, such that if
	$0 < s <  \tau(A)$, and $x \in A \setminus A^{(-s)}$, then
        \begin{equation}
                \label{eq:ball-volume-estimate}
                \left|
		\lambda(B_s(x) \cap A) - ((\theta_d/2) + g(a(x)/s))s^d
                \right|
                \leq
                \frac{2 \theta_{d-1} s^{d+1}}{\tau(A)}.
        \end{equation}
\end{lemma}
\begin{proof}
	See \cite[Lemma 3.4]{HPY24}.
\end{proof}

\begin{lemma}
\label{l:A3}
Let $\eps > 0$.
Suppose $d \geq 2$ and $\partial A \in C^2$.
	There exists $ s_0>0$ depending on
	$d$, $A$ and $\eps$ such that if
	$s\in (0,s_0)$ and $y\in A$, $z \in \partial A$,
	then
\begin{align}
	 \lambda(A \cap B_s(y))\ge ((1/2) - \eps)
	\theta_d s^d;
	%~~~ \mbox{\rm and }
	%~~~ 
	%\lambda( A \cap B_s(y)\setminus B_s(x)) \ge \delta 
	%(\kappa/2)\theta_d
	%s^d;
	\label{e:A31}
%\end{align}
%	\begin{align}
	\\
	\lambda(A \cap B_s(z))\le ((1/2) + \eps)
	\theta_d s^d.
		\label{e:A32}
	\end{align}
		If instead $d=2$ and 
		$A= [0,1]^2$ there exists $s_0 > 0 $ such that
		if $y \in A, s \in (0,s_0)$ then $\lambda(A \cap B_s(y))
		\geq (\pi/4)s^2$.
\end{lemma}
\begin{proof}
	The first inequality \eqref{e:A31} is easily deduced from Lemma
	%\ref{l:A3}
        \ref{l:bdyvol}
	since $g(\cdot) \geq 0$.
%	comes
%	from \cite[Lemma 5.6]{Pen03}
%	or \cite[Lemma 3.4]{HPY24}.
	% XXX The latter is given for $C^{1,1}$.
%
	The second inequality \eqref{e:A32} is also deduced from Lemma 
	%\ref{l:A3}
        \ref{l:bdyvol}
	since $g(0)=0$.
	% XXX The lemma used is given for $C^{1,1}$ in HPY24.
	%comes from  \cite[Lemma 5.6]{Pen03};
	%alternatively it can be proved %by repeatedly using 
	%\cite[Lemma 3.2]{PY21}. 

	The third inequality is obvious.
\end{proof}

\begin{lemma}\label{l:A1}
	There exist $\delta_1 \in (0, \theta_d/4)$ and
	$s_0 >0 $  depending on $d$ and $A$ such that if
	$s\in (0,s_0)$ and $x,y\in A$ with $x \prec y$, then
\begin{align}
	\lambda(A \cap B_s(y)\setminus B_s(x)) & \ge 2 \delta_1 s^{d}  
	 & {\rm if } ~ \|y-x\| & \ge s; 
	\label{e:volLB1}
	\\
	\lambda(A \cap B_s(y)\setminus B_s(x)) & \ge 2 \delta_1
	%8^{-d} \theta_{d-1}
	s^{d-1} \|y-x\| 
	 & {\rm if } ~ \|y-x\| &\le 3s,
	\label{e:volLB2}
\end{align} 
	and if  $\partial A \in C^2$ then \eqref{e:volLB1} still
	holds if we drop the condition $x \prec y$.
\end{lemma}
	 Note that when $A =[0,1]^2$ we do require
	 $x \prec y$ for \eqref{e:volLB1}; otherwise $y$ could be `jammed
	 into a corner' of $A$, for example when $x$ is near
	 $(2^{-1/2}s,2^{-1/2}s)$.
\begin{proof}
	Note first that it suffices to prove the second inequality
	\eqref{e:volLB2}
	for $\|y-x\| \leq s$, since it can be proved in the
	case $s \leq \|x-y\| \leq 3s$ by using the first inequality
	\eqref{e:volLB1} and changing $\delta_1$.

	In the case with $\partial A \in C^2$,
	\eqref{e:volLB1}
	comes from
	\cite[Lemma 10]{PY23} (Lemma 5.9 in the Arxiv version),
	which does not require the
	condition $x \prec y$, while 
	\eqref{e:volLB2}
%See \cite[Lemma 3.7]{PY21}. 
%When $\partial A \in C^2$, this is part of
	comes from
	\cite[Lemma 3.6]{HPY24}. 
	%XXX If we relaxed $C^2$ to $C^{1,1}$ would still be ok since
	% \cite[Lemma 5.9]{PY23} is revised to allow this.

	Now suppose $A = [0,1]^2$. Without
	 loss of generality, the nearest corner of $A$ to $x$ is the origin.
	 Writing $x= (x_1,x_2)$ and $y= (y_1,y_2)$, assume
	 also  without loss of generality that $x_2 \leq y_2$.
	 Also assume $y_1 \leq x_1$
	 (otherwise \eqref{e:volLB1}
	 and \eqref{e:volLB2} are
	 easy to see).

	 If $\|y-x\| \geq s$ then $y_2 \geq x_2 + 2^{-1/2} s$
	 (otherwise the condition $x \prec y$ fails).
	 Then the ball of radius $0.05s$ centred on $(y_1+0.05s, y_2 + 0.8s)$
	 is contained in $A \cap B_s(y)\setminus B_s(x)$, and \eqref{e:volLB1}
	 follows for this case. 
	 
	 For \eqref{e:volLB2}, 
	 we assume without loss of generality that $\|y-x\| \leq s$.
	 Consider the segment $S$ of $B(x,s)$ that is
	 cut off from $B(x,s)$ by the line parallel to $[x,y]$ and
	 at a distance  $2^{-1/2} s $ from $x$, away from the origin
	 (here $[x,y]$ denotes the convex hull of $\{x,y\}$).
	 Then as illustrated in Figure \ref{f:cornerdisks},
	 $S \oplus \{y-x\} \subset A$, and by Fubini's theorem
	 there is a constant $\delta_1 >0$ such that
	 $\lambda ((S \oplus \{y-x\}) \setminus S) \geq 2 \delta_1 s \|y-x\|$. 
	  \begin{figure}
\center
\includegraphics[width=0.49\textwidth, trim= 0 10 5 20]{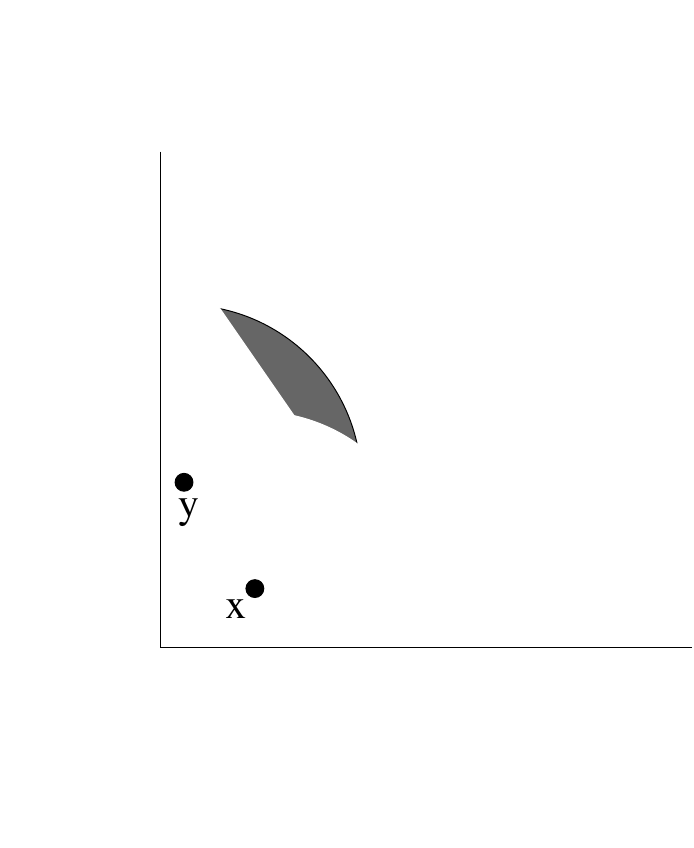}
\vskip -2.0cm
\caption{\label{f:cornerdisks}
		  The shaded region is $(S \oplus \{y-x\}) \setminus S$,
	  as described in the proof of Lemma \ref{l:A1}.  }
\end{figure}
\end{proof}

\begin{lemma}\label{l:A2}
	%Suppose $d \geq 3$.
	Let $0< \eps <K <\infty$. Then there exists
	$\delta_2 = \delta_2(d,A,\eps,K) > 0$ and
	$s_0 = s_0(d,A,\eps,K) > 0$,
	%depending on $\eps,K,d$ and $A$,
	such that for all $s\in (0,s_0)$ and all compact $B\subset A$ with $\diam B \in [\eps s,Ks]$ and
	$x_0 \in B$ with $x_0 \prec y$ for all  $y \in B$,
	we have
\begin{align}
\lambda( (B\oplus B_s(o)) \cap A) \ge \lambda(B)
+ \lambda(B_{s}(x_0) \cap A ) + 2 \delta_2 s^d.
	\label{e:2110}
                %\label{e:MedLB2}
\end{align}	 
\end{lemma}
\begin{proof}
	In the case with $\partial A \in C^2$,
	we can use \cite[Lemma 2.5]{PY21}. %XXX should be OK for $C^{1,1}$.

	If instead $d=2$ and $A= [0,1]^2$, we can argue similarly for $x$ 
	not close to any corner of $A$. In the other case we can use
	\cite[Proposition 5.15]{Pen03}.
\end{proof}

We shall say that a set $\sigma \subset \Z^d$ is {\em $*$-connected}
if the set $\sigma \oplus [-\frac12 \frac12]^d$ is connected.
The following combinatorial result is well-known 
		(e.g. \cite[Lemma 9.3]{Pen03}).
		\begin{lemma} \label{l:Peierls}
			Let $n \in \N$. The number of $*$-connected
			subsets of $\Z^d$ with $n$ elements including $o$ is
			at most $(2^{3^d})^n$.
		\end{lemma}

%\begin{lemma}
%	\cite[Lemma 2.7]{PY21}
%        \label{l:Med2}
%        Suppose $\partial A \in C^{1,1}$.
%        Let $\rho, \eps \in \R$ with $0 < \eps < \rho$.
%        Then there exist
%        $\delta = \delta(d,\rho,\eps) >0$,
%        %and $\delta' = \delta'(d,\rho,\eps) >0$,
%        and $r_0 = r_0 (d, \rho, \eps, A)$,
%        such that for all $r\in (0,r_0)$ and
%        all compact $B \subset A$ with $\eps r \leq \diam B \leq \rho r$
%        we have
%        \begin{align}
%                \label{e:MedLB}
%                | (B \oplus B_r(o)) \cap A | \geq |B| +
%%               todo: could use the notation B^{(r)} but maybe ok
%                ((\theta_d/2) + \delta) r^d,
%        \end{align}
%                and also, letting $x_0$ denote a closest point of $B$
%        to $\partial A$, we have
%         \begin{align}
%                \label{e:MedLB2}
%                | (B \oplus B_r(o)) \cap A | \geq |B| +  |B_r(x_0)\cap A| +
%                2  \delta r^d.
%        \end{align}
%        %}
%\end{lemma}
%\begin{proof} See \cite[Lemma 2.7]{PY21}.

\subsection{Probabilistic tools}
	\label{ss:ProbTools}
%Define $H(a) := 1-a + a \log a$ for $ a >0$. 
%The following lemma will be used repeatedly.

\begin{lemma}[Chernoff bounds]
\label{lemChern}
Suppose  $n \in \N$, $p \in (0,1)$, $t >0$ and $0 < k < n$. 

	(i) If $k \geq e^2 np$ then 
 $\Pr[ \Bin(n,p) \geq k ] \leq \exp\left( - (k/2)
	\log(k/(np))\right) \leq e^{-k}$.

	%(ii) If $k < t$ then $\Pr[Z_t \leq k ] \leq
	%\exp(- t H(k/t))$.

	(ii) For all $t$ large, $\PP[ Z_t \ge t+ t^{3/4} ]\le
	\exp(- \sqrt{t}/9)$ and $\PP[ Z_t \le t - t^{3/4} ]\le  \exp(- \sqrt{t}/9)$.

	(iii) If $k \geq e^2 t$ then $\PP[Z_t \geq k ] \leq e^{-k}$.
\end{lemma}
\begin{proof}
	See e.g. \cite[Lemmas 1.1, 1.2 and  1.4]{Pen03}.
	%for (i), and \cite[Lemma1.3]{RGG} for (ii).
\end{proof}

Let $\bN(\R^d)$ be the space of all finite subsets of $\R^d$, equipped with
the smallest $\sigma$-algebra ${\cal S}(\R^d)$ containing the
sets $\{\cX \in \bN(\R^d): |\cX \cap B|= m\}$ for all Borel $B \subset \R^d$
and all $m \in \N \cup \{0\}$.
Given $F:\bN(\R^d) \to \R$ and $x \in \R^d$, define
the {\em add-one cost}
$D_xF(\X): = F(\X \cup \{x\} ) - F(\X)$
for all $\X \in \bN(\R^d)$.
Also define 
$D_x^+F(\cX):= \max (D_xF(\cX),0)$ and
$D_x^-F(\cX):= \max (-D_xF(\cX),0)$, the positive and negative
parts of $D_xF(\X)$.
\begin{lemma}[Poincar\'e and Efron-Stein inequalities]
	\label{l:Poinc.ES}
	Suppose $F: \bN(\R^d) \to \R$ is measurable and $n >0$. If
	$\E[F(\eta_n)^2] < \infty$ then
\begin{align}
\label{e:poincare}
	\Var[F(\eta_n)] \le n \int_A  \EE[|D_xF(\eta_n)|^2]  \nu(dx).
\end{align}
%for a general Poisson functional $F$ on a metric space with intensity measure $\mu$.
	Also, if $n \in \N$ and $\E[F(\X_n)^2] < \infty$ then
\begin{align}\label{e:p4.2}
	%\Var[F(\X_n)] \le  n \EE[ (F(\X_n) - F(\X_{n-1}) )^2].
	\Var[F(\X_n)] \le  n \int_A \E[|D_x F(\X_{n-1})|^2 ] \nu(dx).
\end{align}
\end{lemma}

\begin{proof}
	The first assertion \eqref{e:poincare} is the Poincar\'e inequality 
	\cite[Theorem 18.7]{LP18}.
	For the second assertion \eqref{e:p4.2},
	we use Efron and Stein's jackknife estimate for
the variance of functions of iid random variables.
Let $\tF_n: \RR^{dn} \to\RR$ be
	given by $\tF_n((x_1,\ldots,x_n)) = F(\{x_1,\ldots,x_n\})$
	for all $x_1,\ldots,x_n \in \R$.
	Then $\tF_n$ is 
	measurable. The Efron-Stein inequality
	(see e.g. \cite{BLB})
	says that
\begin{align}
%\Var[F(\cX_n)] \le \frac{1}{2} \sum_{i=1}^n \EE[( F(\cX_n) - 
	%F(\cX_n^{(i)})^2 ] 
	\Var[\tF_n(\bX_n)] \le \frac{1}{2} \sum_{i=1}^n \EE[( 
	\tF_n(\bX_n) - \tF_n(\bX_{n+1}^{(i)}))^2 ] 
	\label{e:Efron-Stein}
\end{align} 
where $\bX_n:= (X_1,\ldots,X_n)$ and
$\bX_n^{(i)} := (X_1,\ldots,X_{i-1},X_{i+1}, \ldots, X_n)$. 

	We write 
	 $\tF_n(\bX_n)-\tF_n(\bX_{n+1}^{(i)})= 
	 \tF_n(\bX_n) - \tF_{n-1}(\bX_{n}^{(i)}) -( 
	 \tF_n(\bX_{n+1}^{(i)}) - \tF_{n-1}(\bX_{n}^{(i)}) ) $. By the bound 
	$(a+b)^2\le 2a^2+2b^2$ (which comes from Jensen's inequality), 
	and (\ref{e:Efron-Stein}), and the exchangeability
	of $(X_1,\ldots,X_{n+1})$,
\begin{align*}
	%\label{e:p4.2}
	\Var(F(\X_n)) =	\Var[\tF_n(\bX_n)] \le  n \EE[ (\tF_n(\bX_n) - \tF_{n-1}
	(\bX_{n-1}) )^2],
\end{align*}
	and \eqref{e:p4.2} follows.
\end{proof}

\begin{lemma}[Quantitative version of Slutsky's theorem]
	\label{l:QSlut}
	Suppose $X$ and $Y$ are random variables on the
	same probability space with $\E[Y]=0$ and $\Var [Y] < \infty$.
	Then $\dk(X+Y, N(0,1)) \leq 3 (\dk(X,N(0,1)) + (\Var[Y])^{1/3})$.
\end{lemma}
\begin{proof}
	Let $t \in \R$ and set $a := (\Var[Y])^{1/3}$. Then
	by the union bound and Chebyshev's inequality
	\begin{align*}
		|\Pr[X+Y \leq t] - \Pr[X \leq t]|
		& \leq \Pr [\{X+Y \leq t\} \triangle \{X \leq t\}] 
		\\
		& \leq  \Pr[ t-a < X \leq t+a] + 
		\Pr[|Y| \geq a]
		\\
		& \leq \Pr[ t-a < N(0,1) \leq t+a] + 2 \dk(X,N(0,1))
		+ a^{-2}\Var[Y]
		\\
		& \leq 3a + 2 \dk(X,N(0,1)),
	\end{align*}
	and the result follows.
\end{proof}

%Also here: {\bf state Mecke?}.
%2nd order Poincare?  2-scale stabilization? Azuma?

%\subsection{Poisson approximation for $S'_{n,k}$ }
%\label{s:PoSprime}

%Let $\bN(\R^d)$ be the space of all finite subsets of $\R^d$, equipped with
%the smallest $\sigma$-algebra ${\cal S}(\R^d)$ containing the
%sets $\{\cX \in \bN(\R^d): |\cX \cap B|= m\}$ for all Borel $B \subset \R^d$
%and all $m \in \N \cup \{0\}$. 
%Given $m \in \N$, let $\bN_m(\R^d):= \{\X \in
%\bN(\R^d):|\X|=m\}$.

To  prove   Poisson approximation for the number of singletons,
%result (\ref{e:k_pois}) in Theorem \ref{t:int_k} is
we shall use the following coupling bound from
\cite{Pen18} adapted to our situation (i.e. without marking). 
%{\bf (New!)} 
For any event $X$ and any event $E$ with non-zero probability
on the same probability space,
let $\mathscr L(X)$ and $\mathscr L(X|E)$ denote
the distribution (law) of $X$, and the conditional
distribution of $X$ given $E$ occurs, respectively.
%The function  $g$ in the next result has
%nothing to do with the $g$ in Section \ref{s:Boolean}.

\begin{lemma}[{\cite[Theorem 3.1]{Pen18}}]
\label{t:JAP}
        Let $g: \R^d % \bN_1(\R^d)
        \times \mathbf{N}(\RR^d)\to \{0,1\}$  be measurable.
        Define
\begin{align*}
W:=F(\P_n) := \sum_{x \in  \P_n }
	g(x, \P_n\setminus \{x\}).
\end{align*}
Let $n >0$.
        For $x \in \RR^d$,
        set $p(x):=
        \EE[g(x,\P_n)]$ and set $\mu = n \nu$.
        Assume that for $\mu$-a.e. $x$ with $p(x)>0$, we can find coupled random variables $U_x, V_x$ such that
\begin{itemize}
\item $\mathscr L(U_x) = \mathscr L(W)$;
\item $\mathscr L(1+ V_x) = \mathscr L(F(\P_n\cup\{x\})| g
        (x,\P_n)=1)$.
\item
        $\E[|U_{x}- V_{x}|] \le w(x)$, where
                $w: \R^d \to [0,\infty)$ is measurable.

		\end{itemize}
Then
\begin{align}
        \dtv(W, Z_{\EE[W]}
        %\mathrm{Po}(\EE[W])
        )
	%\le \frac{\min(1, \EE[W]^{-1})}{k!}
	\le \min(1, \EE[W]^{-1})
	\int
        %\EE[|U_\bx-V_\bx|]
        w(x)
        p(x)\mu(dx).
        \label{0730b}
\end{align}
\end{lemma}

For Poisson approximation in the binomial setting, % i.e. for $S_{n,k}$,
we use the following result from
\cite[Theorem II.24.3]{Lin92}
or
\cite[Theorem 1.B]{BHJ92}.
\begin{lemma}
        %[{\rm \cite{Lin92}, Theorem II.24.3}]
        \label{l:Lindvall}
        Let $n \in \N$.
        Suppose $Y_1,\ldots,Y_n$ are  Bernoulli random variables
        on a common probability space.
        Set $W:= \sum_{i=1}^n Y_i$.
        %and  $p_i := \E[Y_i]$ for $i \in [ n]$.
        %and and $\la = \EE[W]$.
        Suppose  for each $i \in [n]$ that
        there exist coupled random variables $U_i,V_i$ such that
         $\LL(U_i)= \LL(W)$
         and $\LL(1+V_i) = \LL(W|Y_i=1)$.
         Then
         $$
         \dtv(W, Z_{\EE[W]} ) \leq  (\min(1,1/\EE[W]))\sum_{i=1}^n
	 \EE[ Y_i]
	 %p_i
         \EE[ |U_i- V_i|].
         $$
\end{lemma}

\subsection{Percolation type estimates}
\label{ss:perc}
%Throughout this subsection we assume either 
%	that $d=2$ and $A= [0,1]^2$, or 
%	that $d \geq 2$ and
%	$A \subset \R^d$ is compact and connected, with $A= \overline{A^o}$
%	and $\partial A \in C^2$.   
%
For finite $\X \subset \R^d$, and $x \in \X$ and $ s>0$, let
$\cC_s(x,\X)$ denote the vertex set of the component of
$G(\X,s) $ containing $x$, so  $\#(\cC_s(x,\X))$ is
the order of this component.

To prove our theorems, we shall need to establish uniqueness
of the giant component in $G(\X_n,r)$ or $G(\eta_n,r)$ (with $r=r(n)$).
The next two lemmas help do this, and are proved using 
discretization and path-counting (Peierls) arguments of the sort used
in the theory of continuum percolation.

The first lemma says that if
 $n r_n^d \to \infty$ as $n \to \infty$,
the existence of two components of diameter much larger than $r_n$
is extremely unlikely for $n$ large. 
Throughout, the diameter of a component means the
  Euclidean (rather than graph-theoretic) diameter of its set of vertices.
  
Bounds of this sort also arise in the study of connectivity thresholds 
(which concerns the regime with  $I_n \to c\in (0,\infty)$);
see for instance \cite[Proposition 3.2]{Pen99}.
In the proof, we shall invoke a topological lemma from  \cite{Pen99}.
\begin{lemma}[Uniqueness of the large component]
\label{l:uniqueness} 
 Suppose %we are given
	$(r_n)_{n \geq 1}$ satisfies 
	$n r_n^d \to \infty$ as $n \to \infty$.
	Let $\phi_n$ be given with  $\phi_n \geq \log n $
	for all $n \geq 1$ and assume $\phi_n r_n  \to 0$ as $n \to \infty$.
	Let $\scr U_n$, respectively $\tilde{\scr U}_n$,
	denote the event that there exists at most one component of 
	$G(\eta_n, r_n)$ (respectively $G(\X_{n},r_n)$)
	with diameter larger than $\phi_n r_n$. Then
	for all $n$ large enough,
\begin{align}
\PP[\scr U_n^c]\le \exp(- \delta_3 \phi_n nr_n^d ); ~~~~~~~~ 
	\PP[\tilde{\scr U}_n^c]\le \exp(-\delta_3 \phi_n nr_n^d ), 
	\label{e:0524a}
\end{align}
where 	$\delta_3 >0$ is a constant  
	depending only on $d, A $ and $f$.
\end{lemma}  
	\begin{proof}
		First assume that $\partial A \in C^2$.
	Let $\ep= 1/(99\sqrt{d})$. 
		%\xy{notation clash $A_x$}
	 Given $n$, partition $\RR^d$ into cubes $(Q_{n,i})$ of side length
	 $\ep r_n$
	indexed by $i \in \ZZ^d$. To be definite, for $i= (i_1,\ldots,i_d) \in
	\Z^d$, set $Q_{n,i} := \prod_{k=1}^d ((i_{k}-1)\ep r_n, i_k\ep r_n]$.
	%We say a set $\sigma\subset \ZZ^d$ is  $*$-connected if the closure of the union of cubes $Q_{n,i}, i\in\sigma$  is connected. 
		Recall the definition of $*$-connectedness just before
		Lemma \ref{l:Peierls}.
By the deterministic topological lemma \cite[Lemma 3.5]{Pen99},
	there exist $\al, \al'>0, n_1\in\NN$ such that for all $n\ge n_1$ and  
	for any finite $\cX\subset A$, if
	 $U$ and $V$ are the vertex sets of two components of $G(\X,r_n)$, 
	then there exists a $*$-connected set $\sigma\subset \ZZ^d$ enjoying
		the following properties:
\begin{itemize}
	\item[i)] $\cX \cap (\cup_{i\in\sigma} Q_{n,i} )= \emptyset$;
	\item[ii)] $ \#(\{i\in \sigma: Q_{n,i}\subset A\}) \ge \al \, \#(\sigma)$;
	\item[iii)] $\eps r_n \#(\sigma) \ge \min(d^{-1/2}\diam(U), 
		d^{-1/2}\diam(V), \al')$.
\end{itemize}
		In   (iii), the factor of $d^{-1/2}$ arises because
		if $\diam_\infty$ denotes diameter in the
		$\ell_\infty$ sense
		(as used in \cite{Pen99}) then $\diam_\infty (\cdot)
		\geq d^{-1/2}
		\diam (\cdot)$.

	We shall apply this lemma to $\cC_1$ and $\cC_2$ which we define to be
	 the vertex sets of the largest and second-largest component
	 (in terms of Euclidean diameter) of $G(\eta_n,r_n)$,
		with diameter $\ell_1, \ell_2$ respectively
	 (so $\ell_1 \geq \ell_2$). 

		For $n >0, k \in \N$, define
		\begin{align}
			\cK_{n,k,\alpha} :=
			\{ \sigma \subset \Z^d: \#(\sigma)=k, \sigma ~\mbox{is
			$*$-connected}, \#\{i \in \sigma:Q_{n,i} \subset A\}
			\geq \alpha k\},
			\label{e:Knka}
		\end{align}
		and define the events 
		\begin{align}
		\scr G_{n,k,\alpha} := \cup_{\sigma \in \cK_{n,k,\alpha}}
		\{ \eta_n \cap (\cup_{i \in \sigma} Q_{n,i} ) = \emptyset\};
		~~~~
		\tilde{\scr G}_{n,k,\alpha}
			:= \cup_{\sigma \in \cK_{n,k,\alpha}}
		\{ \X_{n} \cap (\cup_{i \in \sigma} Q_{n,i} ) = \emptyset\}.
			\label{e:GGdef}
			\end{align}

		If event $\scr U_n^c$ occurs,
	then $\ell_2 \geq r_n \phi_n$, so by the lemma, 
	$\scr G_{n,k,\alpha}$ occurs for some $k \geq \eps^{-1}
		d^{-1/2} \phi_n  \geq \phi_n$.
		By Lemma \ref{l:Peierls},
		%a Peierls argument (e.g. \cite[Lemma 9.3]{Pen03})
	there exists $c=c(d,A) > 0 $ such that the family of
	$*$-connected sets $\sigma \subset \ZZ^d$ with
		$\#(\sigma)=k$ and with $Q_{n,i} \cap A \neq \emptyset$
	for some $i \in \sigma$ has cardinality at most $c r_n^{-d} e^{ck}$,
	which is  at most $ne^{ck}$,
	provided $n$ is large enough, by the condition $nr_n^d \to \infty$.
	Thus by the union bound, for $n $ large enough we have
\begin{align}
	\PP[\scr G_{n,k,\alpha}] \le n  \exp( ck -  k \al  nf_0 (\ep r_n)^d)
	\leq n \exp(-(\alpha f_0 \eps^d/2) k n r_n^d),
	\label{0623a}
\end{align}
where we used that $nr_n^d \to \infty$
again for the last inequality.
The same bound holds for $\tilde{\scr G}_{n,k,\alpha}$, since 
the probability of a binomial random quantity taking the value zero
is bounded above by the corresponding probability for a Poisson random
quantity  with the same mean.

By (\ref{0623a}),
for $n$ large enough
\begin{align*}
	\PP[\scr U_n^c]
	 \le \sum_{k\ge  \phi_n} \PP[\scr G_{n,k,\alpha}] & \le 2n
	\exp( - (\al f_0 \ep^d/2) n  r_n^d \phi_n ) 
	\\
	& \leq
	 \exp( - (\al f_0 \ep^d/4) n  r_n^d \phi_n ), 
\end{align*}
where we used the conditions $n r_n^d \to \infty$ 
and $\phi_n \geq \log n$, for the last inequality.
This gives us the first assertion in (\ref{e:0524a}), and 
the second assertion is obtained similarly using $\tilde{\scr G}_{n,k}$.

%{\bf This para. was the repeated verbatim in proof of Lemma \ref{l:E9} but
%was thenexpanded there.}
	In the case where $A=[0,1]^2$, we can argue similarly (see
	\cite[Lemma 13.5]{Pen03}).
	We should now take $\eps $ so that the cubes $Q_{n,i}$ fit
	exactly in the unit cube, which means  $\eps$ needs to vary with $n$
	but we can take $\eps(n)$ to satisfy this condition as well as 
	$\eps \in [1/(99 \sqrt{d}),  1/(98\sqrt{d}))$ for all large enough $n$,
	and the preceding argument still works.

We can prove the results for the other choices of $\xi_n$
in the statement of the lemma, by  similar arguments.
	%and the preceding argument still works.
	\end{proof}

We next provide a bound on the probability
of  existence of a moderately large component of $G(\eta_n,r_n)$ 
near a given location in $A$, again measuring
`size' of a component $\cC$ by the  Euclidean diameter of 
its vertex set $V(\cC)$. 
%We introduce some notation that will be used again later on.
For $x,y,z \in \R^d$ and $\cX$
a finite set of points in $\R^d$, we use the notation
\begin{align}
	\cX^x:= \cX \cup \{x\}; ~~~
\cX^{x,y}:=  \cX \cup\{x,y\}; ~~~  
\cX^{x,y,z}:= \cX \cup\{x,y,z\}. 
	\label{e:addpoint}
\end{align}
Suppose $0\leq \eps < K \leq \infty$. 
Given  $(r_n)_{n \geq 1}$ we define events
\begin{align}
	\scr M_{n,\eps,K}(x,\X)
	& := \{ \eps r_n< \diam(\cC_{r_n}(x,\X^x  ))
	\le K r_n\};  \label{e:medevent1}
	\\
~~~
	\scr M^*_{n,\eps,K}(x,\X) & := \cup_{y \in \cX \cap B_{r_n}(x)}
	\scr M_{n,\eps,K}(y,\X). 
	\label{e:medevent2}
\end{align}

\begin{lemma}[Non-existence of moderately large components near
	a fixed site]
	\label{l:E9}
	Suppose %$r_n $ satisfies
	$nr_n^d \to \infty$ and
	$n^{2/3}r_n^d \to 0$ as $n \to \infty$. 
	Then there exists $n_1 \in (0,\infty)$ 
	such that for all $n \geq n_1$ and all
	$x,y  \in A$, all $\rho \in [1,n^{1/(3d)}]$,
	with $\xi_n$ representing any of $\eta_n$, $\eta_n \cup \{y\}$,
	$\cX_{n-1}$,  $\cX_{n-2} \cup \{y\}$ or
	$\cX_{n-3} \cup \{y\}$,
	we have
\begin{align}
	\PP[ \scr M_{n,\rho,n^{1/(3d)}} (x, \xi_{n}) ]
	\le  \exp(-\delta_4 \rho n r_n^d);
	\label{e:E2}
	\\
	\PP[ \scr M^*_{n,\rho,n^{1/(3d)}} (x, \xi_{n}) ]
	\le  \exp(-\delta_4 \rho n r_n^d),
	\label{e:E9}
\end{align}
where $\delta_4 >0$ is a constant depending only on $d$ and $f_0$.
\end{lemma}
	\begin{proof}
		%[Proof of Lemma \ref{l:E9}]
	Suppose $\xi_n = \eta_n$.
	Assume for now that $ \partial A \in C^2$.
	As in the previous proof, given $n$ we
	partition $\R^d$ into cubes $Q_{n,i}, i \in \Z^d$ of side
	$\eps r_n$ with $\eps = 1/(9\sqrt{d})$.
	For $n >0, k \in \N$, $\alpha > 0$, with
	$\cK_{n,k,\alpha} $ defined at \eqref{e:Knka} define
	$$
	\cK_{n,k,\alpha,x} =
	\{ \sigma \in \cK_{n,k,\alpha} : 
		%\eps r_n \sigma \cap B_{r_n}(x) \neq \emptyset \mbox{ or}
	%\R^d \setminus
	\cup_{i \in \sigma} \overline{ Q_{n,i}} 
	\mbox{ surrounds }
	x 
	%\mbox { lies in a bounded component of }
		\},
	$$
	where we say a set $D \subset \R^d$ {\em surrounds} $x$
	if $x $ lies in a bounded component of $\R^d \setminus D$.
	%and $a \sigma := \{az:z \in \sigma\}$ for $a >0$,
	%and $\overline{D}$ denotes the closure of a set $D$.
	Define the event
		\begin{align}
	\scr G_{n,k,\alpha,x} : = \cup_{\sigma \in \cK_{n,k,\alpha,x}}
	\{ \eta_n \cap (\cup_{i\in\sigma} Q_{n,i} )= \emptyset \}.
			\label{e:G4def}
		\end{align}

	Let $\rho >0$.  Suppose now
	that $\scr M_{n,\rho, n^{1/(3d)}}(x,\eta_n)$
	occurs, and let $\cC:= \cC_{r_n}(x,\eta_n^x)$.
%In this case we shall use ideas from
% the proof of \cite[Lemma 13.9]{Pen03} (or that of 
% \cite[Lemma 3.5]{Pen99}, which is the same).

Set $\cC'= \cC \oplus B_{r_n/2}(o)$;
 then $\cC'$ is a connected compact set. 
Let ${\cal D}$ be the closure of the
unbounded component of $\R^d \setminus \cC'$, and
 let
 $\partial \cC' := \cC' \cap {\cal D} $,
 which is the external boundary of $\cC'$.
		Note that every $y \in \partial \cC'$ satisfies
	%	is distant exactly $r_n$ from the nearest point in 
		$\dist(y,\cC) = r_n/2$.
 %i.e.
 %the intersection of $\cC'$ with the closure of the unbounded component
%	of $\R^d \setminus \cC'$.
 Let $\Sigma$ denote the collection of $i \in \Z^d$ such that
 $Q_{n,i} \cap \partial \cC' \neq \emptyset$.

 Then $\partial \cC'$ is connected by the unicoherence of $\R^d$
 (see e.g. \cite{Pen03}),
 so $\Sigma$ is $*$-connected.
 Also $\eta_n \cap Q_{n,i} = \emptyset$ for all $i \in \Sigma$.
 Moreover, since $\diam(\cC) \leq  n^{1/(3d)} r_n$ and
 %$\cC \cap B_{r_n}(x) \neq \emptyset$,
 $x \in \cC $,
 we have that $\cup_{i \in \Sigma} Q_{n,i} \subset B_{2  n^{1/(3d)} 
 r_n}(x)$.

 We claim that 
 $ \cup_{i \in \Sigma} \overline{ Q_{n,i}}$ surrounds $x$.
Indeed, $x \notin \cup_{i \in \Sigma} \overline{ Q_{n,i}}$ since
	 $\dist (x, \partial \cC') \geq {r_n}/2$,
	 whereas for all $u \in \cup_{i \in \Sigma} \overline{Q_{n,i}}$
	 we have
	 $\dist(u,\partial \cC') \leq \sqrt{d} \eps r_n \leq r_n/9$.
Since  
$x \in \cC \subset \cC'$, 
  any path from $x$  to a point in ${\cal D}$ must pass
 through a point in $\partial \cC'$,
 and the claim follows.

 Note that $\#(\Sigma) \geq \eps^{-1} d^{-1/2} \rho \geq \rho$.
Taking $ \alpha = (1+ \theta_d (4/\varepsilon)^d)^{-1}$,
we claim  that 
(provided $n$ is large enough) we have
 $\Sigma \in {\cal K}_{n,k,\alpha,x}$ for some $k$.

By the assumption $n^{2/3} r_n^d \to 0$,
we have that 
$ n^{1/(3d)} {r_n} \to 0$ as $n \to \infty$.
If $\dist(x,\partial A) > 4 n^{1/(3d)} r_n$ then
%
%Since $ n^{1/(2d)} {r_n} \to 0$ as $n \to \infty$,
%if $x \in \Omega_3$ then 
we have
$B_{2  n^{1/(3d)} r_n}(x) \subset A$ and hence
(provided $n$ is large enough) 
 $\cup_{i \in \Sigma} Q_{n,i} \subset A$,  so the claim
 is valid in this case.

Now suppose instead that
 $\dist(x,\partial A) \leq 4 n^{1/(3d)} r_n$.
We shall now justify the preceding claim
in this case too,
 which takes more work.
% verifiying the claim takes
% more work.
 %If $x \notin \Omega_3$ then (provided $n$ is large enough)
 %we have $\cC \cap \Omega_2 = \emptyset$, 
 %and by following the second case considered in the proof of
 %\cite[Lemma 13.9]{Pen03}, we see that
 %$\Sigma \in {\cal K}_{n,k,\alpha,x}$ for some $k$ (proving 
%this will take some work).
%
% $\dist(x,\partial A) \leq 4 n^{1/(3d)} r_n$,
 %so we now  assume this. 

Without loss of generality we can and do assume the closest point of
$\partial A$ to $x$ is at the origin $o$. 
Let  $\hat{n}_o$ be 
%With $\hat{n}_o$ defined in Definition \ref{d:SC},
%let $e= -\hat{n}_o$, 
%i.e.
the inward unit vecter orthogonal  to the tangent plane at $o$,
as in Definition \ref{d:SC}.
%be the unit  vector perpendicular to the tangent plane of
%$\partial A$ at $o$  pointing away from  $A$ from $o$ (so that $x= -  \|x\| e$).

Given $i \in \Sigma$ with $Q_{n,i} \setminus A \neq \emptyset$,
we define $\phi(i) \in \Sigma $ 
%with $Q_{n,\phi(i)} \subset A$
as follows. 
Take 
$X = X(i) \in {\cal C}$ such that there exists
$y   \in \partial {\mathcal C}' \cap Q_{n,i}$
%\overline{Q_{n,i}}$ and
 with $\|X - y\| = r_n/2$, choosing the first such $X$
%in both cases making the choice  using
in the lexicographic ordering if there is a choice.
Set
$$
\lambda(i) = \max \{\lambda \in [0,\infty): X + \lambda \hat{n}_o
\in {\mathcal C}'\}.
$$
Set $w(i) := X + \lambda(i) \hat{n}_o,$ and
define $ \phi(i)$ to be the $z \in \Z^d$
such that $w(i) \in Q_{n,z}$.  

Let $\mathbb H := \{y \in \mathbb R^d: y \cdot \hat{n}_o \geq 0\}$, and 
$\mathbb L := \{y \in \mathbb R^d: y \cdot \hat{n}_o = 0\}$. 
Let $\tau = \tau(A)/2$.
Set  $b_n = (2n^{1/(3d)})^2r_n/\tau$, and note that 
$b_n \to 0$ as $n \to \infty$  by our assumption on $r_n$.
By Lemma \ref{l:slice},
%the proof of \cite[Lemma 3.4]{HPY24},
all points of the set $ B_{2n^{1/(3d)}r_n}(x)
\cap (A \triangle  \mathbb H) $
lie
%is contained in the  slice
within distance 
$(2n^{1/(3d)}r_n)^2/\tau = b_n r_n$
of the hyperplane $ \mathbb L$.
%Note 
%=: b_n r_n$
%=: b_n r_n$
%lies within distance 
%is contained in a cylinder centred on $o$ with radius 
%$2 n^{1/(3d)} r_n $ and height 
%$s$ with $s \leq
%, where
%$b_n \to 0$ as $n \to \infty$.

%$(n^{1/(3d)} r_n)^2 = % (n^{2/(3d)} r_n)r_n =
% o(r_n)$ as $n \to \infty$.  In particular 

Next we show $Q_{n,\phi(i)} \subset A$.
Since $X \in A \cap B_{2 n^{1/(3d)}}(x)$, we have
$x \cdot \hat{n}_o \geq -b_n r_n$, and since $\lambda(i) \geq r_n/2$,
we have that $w(i) \cdot \hat{n}_o \geq (\frac12 - b_n ) r_n$ and thus
for all $u \in Q_{n,\phi(i)}$ we have (provided $n$ is large enough)
that
$u \cdot \hat{n}_o \geq ( \frac12 - b_n -  \eps \sqrt{d} ) r_n 
\geq  b_n r_n $, and therefore $u \in A$, confirming that
$Q_{n,\phi(i)} \subset A$.

%Next we show that $X+ r_n e \notin A$. To see this
% note that since $X \in A \cap B_{2 n^{1/(3d)}r_n}(x)$ 

Let $\psi(\cdot)$ denote orthogonal projection onto the hyperlane
$\mathbb L$. Then we have:
\begin{align*}
\| \psi( \varepsilon r_n \phi(i) )
	- \psi(X) \| = \|\psi( \varepsilon r_n \phi(i)) - \psi (w(i))\|
\leq \|\varepsilon r_n \phi(i) - w(i)\| \leq \sqrt{d} \eps r_n.
\end{align*}

Choose
$y \in \partial \mathcal C' \cap Q_{n,i}$
with $\| X-y \| = r_n/2$.
Since $Q_{n,i} \setminus A \neq \emptyset$
and $Q_{n,i} \subset B_{2 n^{1/(3d)} r_n}(x)$, we have
$y \cdot \hat{n}_o \leq ( b_n + \sqrt{d} \eps) r_n$. 
Then  we have
$X \cdot \hat{n}_o \leq ( b_n + \sqrt{d} \eps + \frac12 ) r_n$.
% and therefore
%$(X + r_n e)) \cdot e \geq (- b_n - \sqrt{d} \eps + \frac12 ) r_n \geq
% b_n r_n$. 
Also since $X \in A \cap B_{2 n^{1/(3d)}}(x)$ we have
$X \cdot \hat{n}_o \geq -b_n r_n$.
Therefore 
$$
\|X- \psi(X) \| = | X \cdot \hat{n}_o| \leq r_n,
$$ 
and by the triangle inequality
$$
\| X- i \eps r_n \| \leq \|X-y\| + \|y - i \eps r_n\|
\leq  r_n.
$$
Combining the last three displays and using the triangle inequality
again  we have
$$
\|
\psi (\eps r_n \phi(i))
- 
i \eps r_n 
\| \leq 3 r_n.
$$
Therefore given $z \in \Z^d$, the number of $i \in \Sigma$
which satisfy $\phi(i) = z$
is bounded by the number of points of $\eps r_n \Z^d$ lying in
the ball $B( \eps r_n z, 3 r_n)$, which is bounded by
$ 4^d \theta_d / \eps^d$.
From this we can deduce as required that
$\Sigma \in {\cal K}_{n,k, \alpha, x}$, taking
$\alpha = (1+ \theta_d (4/\varepsilon)^d)^{-1},$ as claimed.

 Thus if 
 %in either case if %$\scr E_9(x,\rho,n) $
 $\scr M_{n,\rho,n^{1/(3d)}}(x,\eta_n) $
 occurs, then event $\scr G_{n,k,\alpha,x}$ (defined at \eqref{e:G4def}
 with the above choice of $\alpha$)
 occurs for some $k \geq \rho $.

 By Lemma \ref{l:Peierls}
 there are constants $c,c'$ such that for all $n,k$
 the family of $*$-connected sets $\sigma \subset \Z^d$
 with $\#(\sigma) =k$ and
 with
$ \cup_{i \in \sigma}
  \overline{  Q_{n,i}}$
 surrounding $x$  
 %or $\eps r_n \sigma \cap B_{r_n}(x) \neq \emptyset$
 has cardinality at most $c' e^{ck}$. Hence 
	for $n $ large enough we have
\begin{align*}
	\PP[\scr G_{n,k,\alpha,x}] \le c'  \exp( ck -  k \al  nf_0 (\ep r_n)^d)
	\leq c'
	\exp(-(\alpha f_0 \eps^d/2) k n r_n^d).
\end{align*}
Summing over $k \geq \rho$, using the geometric series formula,  yields
  for $n$ large enough that
\begin{align*}
	\PP[\scr M_{n,\rho,n^{1/(3d)}}(x,\eta_n) ] 
	\leq 2 c' \exp(-(\alpha f_0 \eps^d/2)  \rho n r_n^d).
\end{align*}
Taking $\delta_4 = \alpha f_0 \eps^d/4$, we obtain (\ref{e:E2}). 
Then using Markov's inequality,
 the Mecke formula (see e.g. \cite{LP18}) and \eqref{e:E2} we can deduce that
\begin{align*}
	\PP[\scr M^*_{n,\rho,n^{1/(3d)}}(x,\eta_n) ] 
	& \leq n \int_{B_{r_n}(x)} 
	\PP[\scr M_{n,\rho,n^{1/(3d)}}(y,\eta_n) ] 
	\nu(dy)
	\\
	& = O( nr_n^d \exp(- \delta_4 \rho nr_n^d)),
\end{align*}
and on taking a smaller value of $\delta_4 $ we obtain \eqref{e:E9}.

	In the case where $A=[0,1]^2$, we adapt the preceding argument as 
	follows.
	%(see \cite[Lemma 13.5]{Pen03}).
	We should now take $\eps $ so that the cubes $Q_{n,i}$ fit
	exactly in the unit cube, which means  $\eps$ needs to vary with $n$
	but we can take $\eps(n)$ to satisfy this condition as well as 
	$\eps \in [1/(9 \sqrt{d}),  1/(8\sqrt{d}))$
	for all large enough $n$.
	Also, in this case we define $\cC'$ to be the set
	$\cC \oplus B_{r_n/2}(o) \cap A$, and note that 
	$A \setminus (\{x\} \oplus [-2 n^{1/(3d)}r_n, 2 n^{1/(3d)} r_n ]^d )$
	is connected and disjoint from $\cC'$.
	Let ${\cal D}$ be the closure of the
	component of $A \setminus \cC'$ that contains 
	 this set. Let $\partial \cC' := \cC' \cap {\cal D}$, and
	now set $$
	\Sigma := \{i \in \Z^d: Q_{n,i} \cap \partial \cC' \neq \emptyset, 
	Q_{n,i} \subset A\}.
	$$
Then $\Sigma$ is connected, and surrounds $x$ in the sense that any
path in $A$ from $x$ to $\cal D$ must pass through 
$\cup_{i \in \Sigma} \overline{Q_{n,i}}$. There exist
positive finite constants $\gamma$ and $c$ such that
the number of such $\Sigma$ of length
$n$ is bounded by $c \gamma^n$.
We can then follow the same argument as in the case $\partial A \in C^2$.
\end{proof}

 We shall use crossing estimates from the theory of continuum percolation.
% For $s >0 $ let $C_s:= [0,s]^d$.
 Given $s >0$, and given a point set 
 $\mathcal X \subset [0,s]^d$, %\mathbb R^d$,
 we say that the graph $G(\X,r)$ {\em crosses}
 the cube $[0,s]^d$ %$C_s$
 in the first coordinate
 if there exists a component of $G(\X,r)$ such that its vertex set
 $\cC$ satisfies $(\cC\oplus B_{r/2} (o))\cap (\{0\}\times [0,s]^{d-1})\neq
 \emptyset$ and $(\cC\oplus B_{r/2}(o))\cap (\{s\}\times [0,s]^{d-1})\neq 
 \emptyset$, namely, we can find a path contained in $\cC\oplus B_{r/2}(o)$
 which connects two opposite faces of $[0,s]^d$
 %$C_s$
 along the first coordinate. 
 For each $k\in\{2,\ldots,d\}$,
 we define the event that 
 the graph $G(\X,r)$  crosses
 the cube $[0,s]^d$ % $C_s$ 
 in the $k$th coordinate in an analogous manner.

 %Define similarly $\cross_k(s)$ 
 %for each $k\in\{2,\ldots,d\}$.
 %We say $\cross(s)$ occurs if $\cross_k(s)$ occurs for all $k\in[d]$.   

 Now consider a homogeneous Poisson process $\cH_\al$ in $\RR^d$ with intensity  $\al$. For each $s>0$, let $\cH_{\al,s}=\cH_\al\cap 
 %C_s$ where $C_s:=
 [0,s]^d$.
 For $k \in [d] := \{1,\ldots,d\}$
 we define  $\cross_k(s,\alpha)$ to be the event
  that the graph $G(\cH_{\al,s},1)$
 crosses the cube $[0,s]^d$ in the $k$th coordinate.
 %if there exists a component of $G(\cH_{\al,s},1)$ such that its vertex set
 %$\cC$ satisfies $(\cC\oplus B_{1/2} (o))\cap (\{0\}\times [0,s]^{d-1})\neq \emptyset$ and $(\cC\oplus B_{1/2}(o))\cap (\{s\}\times [0,s]^{d-1})\neq 
 %\emptyset$, namely, we can find a path contained in $\cC\oplus B_{1/2}(o)$
 %which connects two opposite faces of $C_s$ along the first coordinate. 
 %Define similarly $\cross_k(s)$ for each $k\in\{2,\ldots,d\}$.
 We say $\cross(s,\alpha)$ occurs if $\cross_k(s,\alpha)$ occurs
 for all $k\in[d]$.   
 Observe that the crossing event defined above is slightly different from
 the one in Meester and Roy \cite{MR96} where a crossing  in the first
 coordinate is said to occur if there is a path in $(\cH_\al\oplus
 B_{1/2}(o))\cap [0,s]^d$ connecting two opposite faces of $[0,s]^d$ along the first coordinate.  In other words, in \cite{MR96}, one is allowed to use all the Poisson points to construct a crossing path in $[0,s]^d$, while in our setting, one is restricted to the Poisson points in $[0,s]^d$.

 A fundamental fact about continuum percolation is the existence of 
 $\al_c\in (0,\infty)$ such that,  as $s\to\infty$, $\PP[\cross(s,\alpha)]
 \to 1$ for $\al>\al_c$ and $\PP[\cross(s,\alpha)]\to 0$ for $\al<\al_c$. 
 %Here $\PP_\al$ denotes the law of $\cH_\al$.
 For our purpose, we are concerned with the super-critical phase $\al>\al_c$.  The following  estimate taken from \cite{Pen03} quantifies the convergence of the crossing probabilities. 
 
\begin{lemma}[{\cite[Lemma 10.5 and Proposition 10.6]{Pen03}}] 
\label{l:cross}
 Let $d\ge 2$ and $\al> \al_c$. Then there exists a finite
	constant $\delta_5(d, \alpha) >0$ such that for all $s \geq 1$, 
\begin{align*}
1- \PP[\cross_1(s,\alpha)]\le e^{-\delta_5 s}.
\end{align*}
\end{lemma}

From this we derive a bound for the probability of having a small 
giant component. Again in the next result, $\diam$ refers to the
Euclidean metric diameter of the vertex set of a component.

Given finite nonempty $\cX \subset \R^d$ and $n \geq 1$, let $\cL_n(\cX)$
and $\cL_{n,2}(\cX)$
denote the vertex set of the component of $G(\cX,r_n)$ with 
with largest order and second largest order, respectively
(setting $\cL_{n,2}(\cX)$ to be empty if the graph is connected).
Choose the left-most one if there is a tie. 

\begin{lemma}
\label{l:smallgiant}
	Suppose $nr_n^d \to \infty$ and $nr_n^d = O(\log n)$ as $n \to \infty$.
	Then there
	exist  constants $\delta_6, n_1 \in (0,\infty)$ such that 
	for all $n \geq n_1$,
	with $\xi_n$ denoting either $\eta_n$ or $\cX_n$,
\bea
		\PP[\diam( \cL_n(\xi_n)	)<
	(\log n)^2 r_n]
	\le \exp(-   \delta_6 (n/\log n)^{1/d}). 
	\label{0621d}
	\eea
\end{lemma}
\begin{proof}
	First we show there exist constants $\delta, c' \in (0,\infty)$
	such that for all large enough $n$,
	\bea
	\PP[ \# (\cL_n(\xi_n)) < \delta r_n^{-1}]
	\le \exp(-  c'(n/\log n)^{1/d}). 
	\label{0621c}
	\eea
	Without loss of generality we can and do choose 
	$\delta >0$ such that $C_{2\delta } :=[0,2\delta ]^d\subset A$.
	Define the event 
	$$
	\scr Y_n :=\{ G(\eta_{e^{-2}n} \cap C_{2\delta },r_n)~ {\rm crosses}~
	C_{2 \delta } ~{\rm in~ the~ first~ coordinate} \}.
	$$
	Since $\eta_{e^{-2}n} \subset \eta_n$, we have
	that $\scr Y_n  \subset \{\# (\cL_n(\eta_n))
	\geq \delta /r_n\}$ for $n$ large.

	Clearly the graph
	$G(\eta_{e^{-2}n} \cap C_{2 \delta },r_n)$ is isomorphic to
	$G(r_n^{-1}(\eta_{e^{-2}n} \cap C_{2 \delta }), 1)$.
	%where for $a >0$ and $\cX \subset \R^d$ we write
	%$a\X$ for $\{ax: x \in \X\}$.
	Also $e^{-2}nf_0r_n^d>\al_c+1$ for all $n$ large by \eqref{e:supcri}. 
	%where $r^{-1}\eta_n$ denotes the image of $\eta_n$ by the map $x\mapsto r^{-1}x$.
	We claim that for such $n$, we have $\PP[\scr Y_n ] \geq
	\PP_{\al_c+1}[\cross_1(2\delta/r_n)]$. 
	Indeed, by the mapping theorem 
	\cite[Theorem 5.1]{LP18}, $r_n^{-1}(\eta_{e^{-2} n} \cap 
	C_{2 \delta})$
	is a Poisson process in $C_{2\delta/r_n}$  with intensity
	measure having a density bounded below by
	$e^{-2}nf_0 r_n^d $, and hence by $\al_c +1$. 
	By the thinning theorem \cite[Corollary 5.9]{LP18}, one can couple 
	$\cH_{\al_c+1}$ and $\eta_{e^{-2}n}$ 
	in such a way that
	$\cH_{\al_c+1} \cap C_{2 \delta/r_n} \subset r_n^{-1} ( \eta_{e^{-2}n}
	\cap C_{2 \delta})$.
	Since
	the crossing event is increasing in the sense that adding more points to the Poisson process increases the chance of its occurrence,  this coupling
	justifies the claim. Thus
	by Lemma \ref{l:cross},
\begin{align*}
	\Pr[ \#(\cL_n(\eta_n))
	\geq \delta/r_n] \geq \PP[\scr Y_n ]
	\geq
	\PP[\cross_1(2\delta/r_n,\al_c+1)] 
	\ge 1- e^{-2 \delta_5 \delta/r_n},
\end{align*}
	For the case of binomial input, note that if $Z_{e^{-2} n} \leq n$ then
	$\eta_{e^{-2} n} \subset \X_n$ and hence if also $\scr Y_n $ occurs then
	$ \#(\cL_n(\X_n)) \geq \delta/r_n$ for $n$ large. Therefore
	using Lemma \ref{lemChern}(iii) we have
	$$
	\Pr[ \#(\cL_n(\X_n)) < \delta/r_n] \leq \Pr[\scr Y_n ^c]
	+ \Pr[Z_{e^{-2} n} >n] \leq e^{-2 \delta_5 \delta/r_n} + e^{-n}.
	$$
	Thus using the assumption $nr_n^d = O(\log n)$,
	%\eqref{e:supcri}
	we have (\ref{0621c}) for both choices of $\xi_n$.

	%For (\ref{0621d}), 
	Now for $n \geq 1$ let $\rho := \rho(n) := \max ((\log n)^2,1)$, and
	partition $\RR^d$ into cubes of side length $r_n$. 
	Necessarily $\cL_n(\eta_n)$
	intersects one of the cubes with non-empty intersection with $A$, called $Q$, and if 
	$\diam(\cL_n(\eta_n)) <\rho r_n$,
	then
	$\cL_n(\eta_n)
	\subset Q\oplus B_{\rho r_n}$. If also
	$\#(\cL_n(\eta_n)) \geq \delta/r_n$,
	then
	$\eta_n(Q\oplus B_{\rho r_n})\ge \delta/r_n$.  Since $\rho \geq 1$, we have
	$\lambda (Q \oplus B_{\rho r_n}) \leq (3 \rho r_n)^d$.
	By the union bound, we have
	for some constant $c$ that
\begin{align*}
	\PP[\{\diam(\cL_n(\eta_n)) < \rho r_n\}\cap \{ \#(\cL_n(\eta_n))
	\geq \delta/r_n \}]
	\le c r_n^{-d} \PP[Z_{3^d\rho^d n\fmax r_n^d} \ge \delta/r_n]. 
\end{align*}
We can then apply Lemma 
	\ref{lemChern}(iii) provided $\delta /r_n \geq e^2 (3^d \rho^d n\fmax r_n^d)$, or
	in other words $\rho^d \leq  (c' n r_n^{d+1})^{-1}$ for some constant $c'$.

	By assumption $nr_n^d = O(\log n)$ so
	$\rho^d nr_n^{d+1}=O((\log n)^{2d+(d+1)/d} n^{-1/d})$. Hence
	we can apply 
	Lemma \ref{lemChern}(iii) to deduce that for $n $ large
	$$
	\PP[\{\diam(\cL_n(\eta_n))<\rho r_n\}\cap \{ \#(\cL_n(\eta_n)) \geq \delta/r_n \}]
	\le c r_n^{-d} \exp(- \delta/r_n)
	\leq \exp(-\delta/(2r_n)).
	$$
	%We take $\rho = \rho(n) = (\log n)^2$. 
	Since $r_n^d = O((\log n)/n)$,
	we have $r_n^{-1} = \Omega \big( \big(\frac{n}{\log n} \big)^{1/d} 
	\big)$
so using (\ref{0621c}) and the union bound we can deduce
	 (\ref{0621d}) for $\xi_n = \eta_n$.
	 We can obtain 
	 \eqref{0621d} for $\xi_n = \cX_n$ by a similar argument,
	 using Lemma \ref{lemChern}(i) instead of Lemma \ref{lemChern}(iii).
\end{proof}

\section{The number of isolated vertices}
\label{s:sngltn}
 In this section we prove Propositions 
	\ref{p:bc}
and	\ref{p:Sall}.
In the uniform case we also demonstrate the asymptotic equivalence of
$I_n $ and $\mu_n$, defined at \eqref{e:def:I_n} and
\eqref{e:defI'} respectively.

We continue to make the assumptions on $\nu$ and $A$
that we set out at the start of Section \ref{s:prelims}.
%
%Throughout the rest of this paper, we always assume that 
%$d\geq 2$, that
%the density $f$ of the measure $\nu$
%is supported by a compact set $A$ with
%$f_0:= \mathrm{inf}_A f >0$, and
%that $f$ is continuous on $A$, so that
% $\fmax := \mathrm{sup}_A f $ is finite.
% We also assume either that $A$ is connected with smooth boundary
% and $A= \overline{A^o}$,
% or that $d=2$ and $A= [0,1]^2$, and that
Also we assume $r_n \in (0,\infty) $
 is given for all $n \geq 1$.
Recalling from (\ref{e:def:I_n}) that 
$I_n:= n \int \exp(-n \nu(B_{r_n}(x)))\nu(dx)$,
we
assume throughout this section  that $r_n$ satisfies 
\begin{align}
\label{c:lowerbound}
\lim_{n \to \infty} nr_n^d = \infty; 
\\
%\end{align} and 
%\begin{align}
	\liminf_{n \to \infty} I_n > 0.
	\label{c:upperbound}
\end{align}
Recall that for $s >0$ we write $A^{(-s)} := \{x \in A: B_s(x) \subset A\}$.

% Sometimes we make the extra assumption that $f$ is constant
% on $A$: we refer to this as the {\em uniform case}.

 \subsection{Mean and variance of the  number of isolated vertices}
 \label{ss:IsoMom}

Let $S_n$ (respectively $S'_n$) denote the number of singletons
(i.e. isolated vertices) of $G(\X_n,r_n)$ (resp., of $G(\eta_n,r_n)$).
That is, set
\begin{align}\label{e:def_Sn}
	%S'_n = \sum_{x\in \eta_n} \1\{ (\eta_n\setminus x)(B_r(x))=0\};
	S'_n = \sum_{x\in \eta_n} \1\{ \eta_n \cap B_{r_n}(x)= \{x\}\};
~~~~~~~
	%S_n = \sum_{x\in \X_n} \1\{ (\X_n\setminus x)(B_r(x))=0\}.
	S_n = \sum_{x\in \X_n} \1\{ \X_n \cap B_{r_n}(x)= \{x\}\}.
\end{align}
By the Mecke formula $\E[S'_n] = I_n$.
Also  define
\begin{align}
	\tilde I_n := \E[S_n]=  n \int (1-\nu(B_{r_n}(x)))^{n-1} \nu(dx).
	\label{e:tIdef}
\end{align}

\begin{lemma}[Lower bounds on $I_n$]
	\label{l:Ilower}
	Let $f_0^+$, $f_1^+$ be constants with $f_0^+ >f_0$ and
	$f_1^+ >f_1$. Then as $n \to \infty$,
	\bea
	n \exp(-n \theta_d f_0^+ r_n^d) = o( I_n);
	\label{e:Ilower}
	\\
	n^{1-1/d} \exp(-n \theta_d f_1^+ r_n^d/2) = o(I_n).
	\label{e:Ilower2}
	\eea
\end{lemma}
\begin{proof}
	Assume for now that $\partial A \in C^2$.
	See \cite[Lemma 2]{PY23} (Lemma 3.1 in the Arxiv version)
	for a proof of \eqref{e:Ilower}.
	For \eqref{e:Ilower2}, 
	choose $x_0 \in \partial A$ with $f(x_0) < f_1^+$.
	Using the assumed continuity of $f$,
	choose $s_0 >0$ and $\delta  >0$ such that
	$$
	(1+\delta)^{d+1}
	\sup \{ f(y): y \in A \cap B_{2s_0}(x_0) \} 
	\leq f_1^+.
	$$
	By
	Lemma \ref{l:A3},
	there is a constant $r_1>0$
	(independent of $z$) such that 
	$\lambda(B_s(z) \cap A) \leq (1+\delta) \theta_d s^d/2$
	for all $z \in \partial A, s \in (0, r_1)$. 

	Let $y \in B_{s_0}(x_0) \cap A \setminus A^{(-\delta r_n)}$
	and let $z$ be the closest point of $\partial A$ to $y$. 
	Then $\|y-z\| \leq \delta r_n$, so
	provided $n$ is large enough
	\begin{align*}
		\nu(B_{r_n}(y)) \leq \nu(B_{(1 + \delta){r_n}}(z)) 
		& \leq  (1+ \delta)^{d+1}
	\sup \{ f(y): y \in A \cap B_{2s_0}(x_0) \} 
	%	f_1 
		\theta_d r_n^d/2
		\\ &	\leq f_1^+ \theta_d r_n^d/2.
	\end{align*}
	Therefore since $\lambda(B_{s_0}(x_0) \cap A \setminus A^{(-s)})
	= \Omega (s)$ as $s \downarrow 0$,
	\begin{align*}
		I_n \geq n \int_{B_{s_0}(x_0) \cap A \setminus
		A^{(-\delta {r_n})} } \exp(-n \nu(B_{r_n}(y)) ) \nu(dy)
		= \Omega ( n r_n e^{-nf_1^+ \theta_d r_n^d/2}),
	\end{align*}
	and then using \eqref{c:lowerbound} we obtain \eqref{e:Ilower2}.

	In the case where
	$A= [0,1]^2$ instead of $\partial A \in C^2$, the preceding
	proof still works, since we can take $x_0$ to not be a corner of
	$A$.
\end{proof}
\begin{lemma}[Upper bound on $I_n$]
	\label{l:IUB}
	Suppose $b^+ \leq (6/5) b_c$ with $b_c$ given at \eqref{e:bcdef}. Then given $\eps >0$, 
	\begin{align}
	I_n = O(ne^{-n \theta_d f_0 r_n^d} + n^{1-1/d}  e^{-n r_n^d \theta_d f_1 
	(\frac12-\eps)}).
		\label{e:InUB}
	\end{align}
\end{lemma}
\begin{proof}
	 Note first that for all $x \in A^{(-r_n)}$ we have
	 $\nu(B_{r_n}(x)) \geq \theta_d r_n^d f_0$, so that 
	 writing $I_n(S) $ for $n \int_S\exp(-n \nu(B_{r_n}(x))) \nu(dx)$,
	 we have
	 $$
	 I_n(A^{(-r_n)}) = n \int_{A^{(-r_n)}}  e^{- n \nu(B_{r_n}(x))}
	 \nu(dx)
	 \leq n e^{-n \theta_d r_n^d f_0}.
	 $$
	 Let $\eps  \in (0,\frac12)$. 
	 Suppose $\partial A \in C^2$. Then by Lemma \ref{l:A3} and
	 the continuity of $f$, for all large enough $n$ and all 
	 $x \in A \setminus A^{(-r_n)} $ we have
	 $\nu(B_{r_n}(x)) \geq (1 - \eps) f_1 \theta_d r_n^d/2$;
	 hence
	 %also since we assume $b^+ < \infty$ we have
	 %$r = O(((\log n)/n)^{1/d})$ so that
	 $$
	 I_n (A \setminus A^{(-r_n)} ) = 
	 O(n r_n e^{-n \theta_d r_n^d f_1  (1 - \eps)/2}) =
	 O(n^{1-1/d} e^{-n \theta_d r_n^d f_1  (\frac12 - \eps)}).
	 $$
	 Now suppose  instead that $d=2$ and
	 $A = [0,1]^2$. Let $\mathsf{Cor}_n $
	 denote the set of $x \in A$ lying at an $\ell_\infty$
	 distance at most $r_n$ from one of the corners of $A$.
 By the same argument as above
	 $$
	 I_n(A \setminus A^{(-r_n)} \setminus \mathsf{Cor}_n)
	 =
	 O(n^{1/2}  e^{-n \theta_d r_n^2 f_1  (\frac12 - \eps)}).
	 $$
	 Also $I_n(\mathsf{Cor}_n) \leq 4 \fmax nr_n^2 
	 \exp(-n\pi r_n^2 f_0/4)$
	 and using the 
	  assumption $b^+ < (6/5) b_c = 6/(5f_0)$,
	% $n f_0 \pi r^d < (1+ \eps) \log n$,
	 we obtain
	 for large $n$ that $n f_0 \pi  r_n^2 \leq (5/4) \log n$ so that
	 $$
	 \frac{I_n(\mathsf{Cor}_n)}{
		 n e^{- n \pi r_n^2 f_0 } } 
	 \leq 4 \fmax  r_n^2 \exp( 3 n \pi r_n^2 f_0/4 )
	  = O(r_n^2 n^{15/16 }) = o(1).
	 $$
	 Combining all of the preceding estimates we obtain
	 for both cases ($\partial A \in C^2$ or $A= [0,1]^2$) that
	 \eqref{e:InUB} holds.
\end{proof}

\begin{proof}[Proof of Proposition \ref{p:bc}]
	If $b^+ < b_c = \max( \frac{1}{f_0},
	\frac{2-2/d}{f_1})$ then
	we claim that 
	%there exists $\delta >0$ such that
	%$I_n > n^\delta$  for all $n$ large enough, so in particular
	$I_n \to \infty$. Indeed, if $b^+ <1/f_0$,
	 choose $f_0^+ > f_0$, $\delta >0$,
	such that $f_0^+(b^+ + \delta) <1$. Then 
	for $n$ large $ne^{-n \theta_d f_0^+ r_n^d} >n
	e^{-f_0^+ (b^+ + \delta) \log n}$
	so $I_n \to \infty$ by Lemma \ref{l:Ilower}.
	 If $b^+ < (2-2/d)/f_1$ then  choose
	 $f_1^+ > f_1$ and $\delta >0$ with $f_1^+(b^+ + \delta) < 2-2/d$. 
	 Then for $n$ large,
	 $n^{1-1/d} e^{-n \theta_d f_1^+ r_n^d/2}
	 > n^{1-1/d} e^{- \frac12 f_1^+ (b^+ + \delta)\log n}$
	so again $I_n \to \infty$ by Lemma \ref{l:Ilower}, and the
	claim follows.

	Now suppose $b^- > b_c$. We need to show $I_n \to 0$ as $n \to \infty$.
	Since %the number of isolated vertices of $G(\X,s)$ 
	$n \int_A e^{-n \nu(B_s(x))} \nu(dx)$ is nonincreasing
	in $s$, it suffices to prove this under the extra assumption
	$b^+ \leq 6b_c/5$, which makes Lemma \ref{l:IUB} applicable.
	Since $b^- > b_c$ there exists $\eps >0$ such that
	for $n$ large enough $\theta_d f_0 n r_n^d >
	(1+ \eps) \log n$
	and $n r_n^d \theta_d f_1 (\frac12 - \eps) > (1- 1/d) (1+ \eps) \log n$,
	and then we see $I_n \to 0$ by \eqref{e:InUB}.

	Finally if $b^+ > b_c$ then by the preceding argument
	$I_n \to 0$ as $n \to \infty$ along some subsequence,
	so we must have $\liminf_{n \to \infty} I_n = 0$.
\end{proof}

\begin{lemma}[Asymptotic equivalence of $I_n$ and $\tilde{I}_n$] 
\label{l:diff_I_Itilde}
	There exists $\delta_7 >0$ such that as $n\to\infty$ we have 
	$|I_n-\tilde I_n| = O(e^{- \delta_7 nr_n^{d}}I_n)$.
\end{lemma}
\begin{proof}
	For $x \in A$, given $n$ write $p_n(x) := \nu(B_{r_n}(x))$.
	By the bounds  $1-p_n(x) \leq e^{-p_n(x)}$,
	and $1-p_n(x) \geq 1- \fmax \theta_d r_n^d$, 
	and the condition \eqref{c:lowerbound},
	\begin{align*}
		%\E[S_n ] & = n \int_A (1- p_n(x))^{n-1} \nu(dx) 
		%\\
		\tilde{I}_n &
		\leq (1-  \fmax \theta_d r_n^d)^{-1} n \int_A e^{-n p_n(x)} \nu(dx)
		\\
		& = (1+ O(r_n^d))I_n.
	\end{align*}
	Also by Taylor's theorem 
	$\log ( 1-p) \geq -p - p^2 $ for $p >0$ close to 0, so
	\begin{align*}
		\tilde{I}_n & \geq
		%n \int_A (1-p_n(x))^n dx =
		n \int_A \exp(n\log (1-p_n(x))) dx
		\\
		& 	\geq n \int_A \exp(n (-p_n(x) -p_n(x)^2)) dx
		\\
		& 	\geq e^{-n (\fmax \theta_d r_n^d)^2} I_n. 
		%\\
		%& = (1+ o(1))I_n,
	\end{align*}
	Combining these two estimates and using \eqref{c:lowerbound}
	yields $|\tilde{I}_n -I_n| = 
	O(nr_n^{2d}I_n)$. Therefore it suffices to show $nr_n^{2d} 
	e^{\delta nr_n^d} = O(1)$ for some $\delta >0$. By \eqref{c:upperbound}
	and Proposition \ref{p:bc},
	$nr_n^d = O(\log n)$ so for $\delta $ small enough $e^{\delta nr_n^d}
	= O(n^{1/2}) $, while $nr_n^{2d} = O((\log n)^2 n^{-1})$
	so $nr_n^{2d} e^{\delta nr_n^d} = O((\log n)^2 n^{-1/2}) = o(1)$.
\end{proof}

\begin{proposition}[Variance asymptotics of the number of singletons: Poisson input]
\label{p:Sn}
%	Suppose that \eqref{c:lowerbound} holds, $r=r(n)\to 0$ as $n \to \infty$
%	and $A$ is strongly smooth. Then 
	Let $\delta_1$ be as in Lemma \ref{l:A1}. Then
	%There exists $c >0$ such that
	as $n \to \infty$,
	%for  $n$ large, 
\begin{align*}
\Var[S'_n]= I_n (1 + O(e^{- \delta_1 f_0 nr_n^d})). 
\end{align*}
\end{proposition}
\begin{proof}
	Since $S'_{n}(S'_{n}-1)$ is the number of ordered
        pairs of distinct isolated vertices,
	%in $G(\P_n,r)$,
        %with $h(\cdot):= h_r(\cdot)$ as given at (\ref{e:Brhrdef}),
        %we can write $S'_{n}(S'_{n}-1)$ as follows
\begin{align*}
        S'_{n}(S'_{n}-1) = 
        \sum_{x, y \in  \P_n}
	\1\{ (\P_n\setminus \{x,y\})
	\cap B_{r_n} (x,y)=\emptyset,  \|y- x\| > r_n \},
        %\label{e:Snk^2}
\end{align*}
        where $B_r(x,y):= B_r(x) \cup B_r(y)$.
%Indeed, for distinct  $x, y \in \P_n$, their distance must be
%	larger than $r_n$ in order that they are both connected components of order $k$, for otherwise $\psi$ and $\ph$ are connected with $|\psi\cup\ph|>k$.
        Thus by the multivariate Mecke equation
\begin{align*}
        \EE[(S'_{n})^2]-\EE[S'_{n}] =
        %\frac{n^{2k}}{k!k!}
	n^2
        \int_{A} \int_{A}
       % h(\bx)h(\by) 
	\exp(-n \nu(B_{r_n}(x,y)))
        \1\{\|x - y\| > r_n\} \nu(dy) \nu( dx).
\end{align*}
%and for $\bx = (x_1,\ldots,x_k), \by = (y_1,\ldots,y_k) \in (\R^d)^k$
%we set $\dist(\bx,\by) := \min_{i,j \in [n]} \|x_i -y_j\|$.
        We compare this integral with
\begin{align*}
        \EE[S'_{n}]^2 = I_n^2 =  n^2 % \frac{n^{2k}}{k!k!}
        \int_{A} \int_{A}
       % h(\bx)h(\by) 
	\exp(-n[ \nu(B_{r_n}(x)) +  \nu (B_{r_n}(y))])
        \nu(dy) \nu(dx).
\end{align*}
  %which comes from (\ref{e:exp_Snk}).
	Observe that  $\nu(B_{r_n}(x,y)) = \nu(B_{r_n}(x)) + \nu(B_{r_n}(y))$
	whenever $\| x - y \| > 2{r_n}$.
        Therefore $|\Var[S'_{n}]-\EE[S'_{n}]|\le J_{1,n}+J_{2,n}$,
        where
\begin{align}
        J_{1,n}&:= n^2  %\frac{n^{2k}}{k!k!}
	\int_{A} \int_{A}
	\exp(-n \nu(B_{r_n}(x,y)))
        \1\{ r_n < \| x -y \|  \leq 2r_n %(r,2r]
	\}
        \nu(dy) \nu(dx);
        \label{e:J1def}
        \\
        J_{2,n}&:= %\frac{n^{2k}}{k!k!}
	n^2 \int_{A^{2}}
       % h(\bx)h(\by)
	\exp(-n[\nu(B_{r_n}(x))+\nu(B_{r_n}(y))])
        \1\{\|x - y\| \leq 2r_n \} \nu^{2}(d(x, y)).
        \label{e:J2def}
\end{align}
We estimate $J_{1,n}$ in Lemma \ref{l:J1} below.
For $J_{2,n}$, note that by Lemma 
	\ref{l:A3}
%	\ref{l:bdy_est}
	there exists $n_0
        \in (0,\infty)$
such that for all $n \geq n_0$ and all
$y \in A$
we have $\nu(B_{r_n}(y)) \geq (\theta_d/4) f_0 r_n^d$.
%Hence for $\bx \in (A^{(9kr)})^k$ we have
Hence
%for $x \in A$ we have
$$
%\frac{n^k}{k!} 
\sup_{x \in A} \left( n \int_{A} % h(\by)
\exp(-n \nu(B_{r_n}(y)))
\1\{\|x - y\| \leq 2r_n\} \nu(d y)
\right)
=O( nr_n^d \exp(-n (\theta_d/4) f_0  r_n^d)),
$$
and hence
$J_{2,n}
= O( I_{n} nr_n^d \exp(-n f_0 (\theta_d/4) r_n^d))$, which
is $O(I_{n} \exp(-n f_0 (\theta_d/4)r^d))$.
Combined with Lemma \ref{l:J1}, this completes the proof.
\end{proof}

%
%	This is proved in \cite[Proposition 4.1, case $k=1$]{PY23}, 
%	 for $\partial A \in C^2$, under the extra assumption
%	\cite[equation (3.1)]{PY23} which amounts to
%	$b^+ < 1/\max(f_0, d(f_0-f_1/2))$ (or $b < 2/(d f_0)$ in the
%	uniform case).  We need to  relax this condition to
%	$\liminf_{n \to \infty} I_n  > 0$ as well as $nr_n^d \to \infty$.
%	Also we need to consider the case with $d=2, A =[0,1]^2$.
%	We find that the proof of \cite[Proposition 4.1]{PY23} still works,
%	except for where it uses \cite[Lemma 4.2]{PY23}.  The latter
%	 proof does invoke the extra assumption on $b^+$, so we
%	 provide a different proof of \cite[Lemma 4.2, case $k=1$]{PY23}
%	 not requiring the extra assumption on $b^+ $,
%	 that works for $A=[0,1]^2$ as well as for $\partial A \in C^2$.
%
%
%
%	 The quantity denoted 
%	 $J_{2,n}$ in \cite{PY23} is $O(I_n e^{-cnr_n^d})$ as
%	 in \cite{PY23}. The quantity denoted 
%%	 $J_{1,n}$ in \cite{PY23} 
%	 (when $k=1$)
%	 is given by
	\begin{lemma}
        \label{l:J1}
        Let $J_{1,n}$ be given by \eqref{e:J1def}.
		Then
        $J_{1,n} = O(e^{- \delta_1 f_0 nr_n^d}I_{n})$ as $n \to \infty$, where
		$\delta_1$ is as in Lemma \ref{l:A1}.
\end{lemma}
\begin{proof} Since the integrand in \eqref{e:J1def}
	is symmetric in $x$ and $y$, we have:
	 \begin{align*}
		 J_{1,n} 
		 %& =
		 %n^2 \int_A \int_A e^{-n \nu(B_{r_n}(x) \cup B_{r_n}(y))}
		 %\1\{{r_n} < \|y-x\| \leq 2{r_n}\} \nu (dy) \nu (dx)
		 %\\
		 %&
		 \leq 2 n^2 \int_A \int_A e^{-n \nu(B_{r_n}(x) \cup B_{r_n}(y))}
		 \1\{{r_n} < \|y-x\| \leq 2{r_n}, x \prec y\} \nu (dy) \nu (dx).
	 \end{align*}
	By \eqref{e:volLB1} from Lemma \ref{l:A1},
	 there exists $\delta_1 >0$ such
	 that for $n$ large and $x,y \in A$ with $\|x-y \| \geq {r_n}$
	 and $x \prec y$ 
	 we have $\nu(B_{r_n}(y) \setminus B_{r_n}(x)) \geq 2 \delta_1 f_0 r_n^d$.
	 Hence
	 \begin{align*}
		 J_{1,n} & \leq 2 n^2 \fmax \theta_d (2r_n)^d
		 \int_A e^{-n \nu(B_{r_n}(x)) - 2 n \delta_1 f_0 r_n^d} \nu(dx)
		 \\
		 &  \leq % \fmax \theta_d 2^{d+1} (nr_n^d)
		 e^{-n\delta_1 f_0 r_n^d} I_n,
	 \end{align*}
	 which gives us
	 %\cite[Lemma 4.2, case $k=1$]{PY23} and hence
	 the result. 
	 % XXX {\bf [Alternatively just write out the whole proof!?]}
	 % Also, could extract as lemma the estimate on
	 %$J_{1,n}$ since we use several times
\end{proof}

\begin{proposition}[Variance asymptotics of the number of singletons: binomial input]
\label{p:var_iso_bin}
	There exists $\delta_8 >0$ such that as $n \to \infty$,
	$$\Var[S_n]= I_n(1+O(e^{-\delta_8 nr_n^d})).$$
\end{proposition}

\begin{proof}
	See the case $k=1$ of
	\cite[Proposition 3]{PY23} (Proposition 4.3 in the Arxiv version)
	and Lemma \ref{l:diff_I_Itilde}
	of the present paper.
	In the proof of \cite[Proposition 3]{PY23}
	it is assumed that \cite[equation (3.1)]{PY23} holds
	(i.e. $b^+ < 1/\max(f_0, d(f_0-f_1/2))$ in our notation here),
	but the proof 
	carries through to the general case with $nr_n^d \to \infty$
	and $I_n \to \infty$. Instead of \cite[Lemma 5]{PY23}
	(Lemma 4.2 of the Arxiv version)
	we can use Lemma 
	%\ref{}
        \ref{l:J1}
	%the estimate on $J_{1,n}$ in the proof of Proposition \ref{p:Sn}
	of the present paper.
\end{proof}

\subsection{Asymptotic distribution of the singleton count}
For both the normal and Poisson convergence results, we
use the following. Again $\delta_1$ is as in Lemma \ref{l:A1}
\begin{lemma}[Poisson  approximation  for $S'_n$]
	\label{l:dTV}
	%{\bf [Anything like this in connectivity threshold paper??
	%Apparently there is a weaker version of this there]}
	%
	%There  exists $\delta >0$  such that 
	As  $n \to \infty$,
	\begin{align}
	& \dtv(S'_n,Z_{I_n}) 
	 = O( e^{- \delta_1 f_0 n r_n^d});
	\label{e:dtvS'Z}
\end{align}
\end{lemma}
\begin{proof}
	%{\bf (New).}
%
%	\begin{proposition}
%\label{p:poi}
%        Let $k \in \N$. Then
%there exist finite constants $c, n_5 >0$ such that
%        %for all $n \geq n_0$ we have
%$$
%        %- \log (       \dtv(S'_{n,k}, Z_{I_{n,k}}) ) = \Omega(nr^d)
%                \dtv(S'_{n,k}, Z_{I_{n,k}})  \leq e^{-c n r^d},
%        ~~~~~ \forall ~ n \geq n_5.
%        %~~~~~{\rm as} ~ n \to \infty.
%        %\le e^{- c n r^d}.
%        $$
%\end{proposition}
%
%\begin{proof}
	We apply Lemma \ref{t:JAP} with
        $
        g(x,\psi):=
	\1\{\psi \cap B_{r_n}(x)= \emptyset\}.
        $
        %For $\bx = (x_1,\ldots,x_k) \in (\R^d)^k$,
        For $x  \in \R^d$,
        we construct  coupled random variables $(U_x,V_x)$ as follows.
        Define $U_x := \sum_{y \in \P_n} 
	g(y, \P_n \setminus \{y\})$, and
\begin{align*}
V_x := \sum_{y \in \P_n\setminus %(\bx\oplus B_r)
	B_{r_n}(x)} g(y, (\P_n\setminus B_{r_n}(x))
        %(\bx\oplus B_r)
	\setminus \{y\}).
\end{align*}
This coupling satisfies the distributional
	requirement in Lemma \ref{t:JAP} because the conditional distribution of
        $\P_n$ given the event $\{ g(x,\P_n)=1\}$ is the same as the distribution of
	$\P_n\setminus B_{r_n}(x)$.

	There are two sources of contribution to the change 
	$V_{x}-U_{x}$ in the singleton count upon removal of
	all the Poisson points in $B_{r_n}(x)$.
First, after removal, all singletons  of $G(\P_n,r_n)$
	that were inside  $B_{r_n}(x)$ are destroyed,
	thereby  reducing the singleton  count.  Second, every  % $k$-set
	$y \in  \P_n\setminus B_{r_n}(x)$ satisfying the two properties
\begin{itemize}
	\item[(a)] $g(y, ( \P_n\setminus
                %(\bx\oplus B_r)
		B_{r_n}(x) )
                \setminus y)=1$;
        \item[(b)] $\P_n \cap 
                %(\bx\oplus B_r)
		B_{r_n}(x)
		\cap  B_{r_n}(y )\neq \emptyset$;
\end{itemize}
becomes an isolated vertex only after removing all the Poisson points
	in $ B_{r_n}(x)$, thereby increasing the number of singletons.
                Let $\xi_1(x)$ denote the number of singletons
                of $G(\P_n,r_n)$ that lie in
		$B_{r_n}(x)$ and let $\xi_2(x)$ denote the
                number of singletonss $y$ of
		$\P_n\setminus  B_{r_n}(x)$ satisfying property (a)
		and property (c) $B_{r_n}(y)\cap B_{r_n}(x)
                \neq\emptyset$. It is clear that (b) implies (c) and
\begin{align}
|U_x-V_x|\le \xi_1(x)+\xi_2(x).
        \label{e:UVX1X2}
\end{align}

We estimate $\EE[\xi_1(x)]$ and $\EE[\xi_2(x)]$ separately. Since
$$
	\xi_1(x) = \sum_{y \in \P_n \cap B_{r_n}(x) }
	\1\{ (\P_n \setminus \{y\})  \cap B_{r_n}(y) = \emptyset \},
$$
applying the Mecke equation leads to
$$
\E [ \xi_1(x) ] = \int_{B_{r_n}(x)}
e^{-n \nu(B_{r_n}(y))} n \nu(dy).
$$
%If $y \in A  B_{r_n}(x)$ %and $h_r(\by)=1$
%then
%$y \subset B_{kr}(\bx)^k$. Moreover 
By Lemma \ref{l:A3},
if $n$ is large enough then $\nu(B_{r_n}(y)) \geq f_0 (\theta_d/4) r_n^d$
for any $y \in A$. Hence for all $n$ large enough and all $x$,
$$
\E[ \xi_1(x)] \leq   n  \theta_d r_n^d \fmax e^{- f_0 (\theta_d/4) n r_n^d}.
$$
Therefore setting $p(x)= \EE[g(x,\P_n)]$, and using 
%(\ref{e:exp_Snk}),
		\eqref{e:def:I_n}, we have
\begin{align}
        \int_{A} \EE[ \xi_1(x)] p(x) n \nu(dx)
%& \leq \frac{ (\fmax \theta_d k^{d+1})^k}{(k-1)!} (nr^d)^k e^{-c_1f_0nr^d}
        & \leq
	%\fmax \theta_d nr_n^d e^{-(\theta_d/4)f_0nr_n^d}
       %% (\frac{n^k}{k!}) 
%	n \int_{A} p(x) \nu(dx)
%\nonumber \\
%        & =
         \fmax \theta_d  nr_n^d e^{-(\theta_d/4)f_0nr_n^d}
        I_{n}.
        \label{0730c}
\end{align}

 Turning to $\xi_2(x)$, set $\gamma_n(x,y)= \1\{ r_n < \dist(x,y) \leq 2r_n\}$.
 By the Mecke equation
\begin{align*}
\EE[\xi_2(x)] &= n
        \int_{A}
        %h_r(y)
	\gamma_n(x,y)
	e^{-n \nu(B_{r_n}(y) \setminus B_{r_n}(x))}
        \nu (  dy),
\end{align*}
and therefore writing $B_r(x,y)$ for $B_r(x) \cup B_r(y)$,
we have that
\begin{align*}
        \int_{A} \EE[\xi_2(x)] p(x)n \nu(dx)
        = n^{2} \int_{A} \int_{A}
        \gamma_n(x,y)
	e^{-n \nu(B_{r_n}(x,y))}
 \nu       (d y) \nu(dx).
\end{align*}
By (\ref{e:J1def}) this expression is equal to $ J_{1,n}$, and therefore
by Lemma \ref{l:J1}
it is $O(e^{- \delta_1 f_0 nr_n^d}I_{n})$.

Combining this with (\ref{0730c}), and using (\ref{e:UVX1X2})
and the fact that we took $\delta_1 < \theta_d/4$ in Lemma \ref{l:A1},
we obtain that
$$
\int_{A} \E [|U_{x} - V_{x}| ] p(x) n \nu(dx)
= O( e^{- \delta_1 f_0  n r_n^d}  I_{n}).
        $$
Applying Lemma \ref{t:JAP} with the present choice of $g$
(so that the $W$ of that result is $S'_{n}$)
gives the desired bound in Poisson approximation, completing the proof
of
%\end{proof}
%%
%%	\begin{proof}
%	The result \eqref{e:dtvS'Z}
%	%for $S'_n$ in  \eqref{e:dtvSZ}
%%	is proved in 
%	\cite[Proposition 5.4]{PY23} (taking $k=1$ there),
%	under the extra assumptions that
%	$\partial A \in C^2$ and 
%	that \cite[equation (3.1)]{PY23} holds,
%	i.e. $b^+ \leq 1/\max(f_0, d(f_0-f_1/2))$
%	in our notation.
%	 We wish to drop this assumption on $b^+$,
%	and assume only that $n r_n^d \to \infty$ and $\liminf_{n \to \infty}
%	(I_n) > 0.$
%	Also we wish to include the case where $A =[0,1]^2$.
%
%	Going through the proof of \cite[Proposition 5.4]{PY23}
%	we find that it still works under our weaker assumptions,
%	but that it uses \cite[Lemma 4.2]{PY23}. In the 
%	proof of Proposition \ref{p:Sn} we have
%	already provided a the proof of \cite[Lemma 4.2]{PY23}
%	under our weaker assumptions, justifying 
%%	Thus we have justified the result for $S'_n$ in
	of \eqref{e:dtvS'Z}.
\end{proof}

	%{\bf New proof of \eqref{e:dtvSZ}.}
	
\begin{lemma}[Poisson  approximation  for  $S_n$]
	\label{l:dTV2}
	%There  exists $\delta >0$  such that 
	As  $n \to \infty$,
	\begin{align}
	{\rm ~if~} \partial A \in C^2, ~~~ &
	\dtv(S_n,Z_{\tilde{I}_n}) = O( e^{- \delta_1 f_0 n r_n^d}).
	\label{e:dtvSZ}
	\end{align}
\end{lemma}
\begin{proof}
	We shall % prove \eqref{e:dtvSZ} 
	use  Lemma \ref{l:Lindvall}. 
	 Let $Y_i$ be the indicator of the
event that $\cC_{r_n}(X_i,\cX_n) =\{X_i\}$. 
%{\em and} $X_i$ is the left-most
%point of $\cC_r(X_i,\cX_n)$ (we called this $\xi_i$ at (\ref{e:xidef})).
Then $S_{n}= \sum_{i=1}^n Y_i $.
%
%To use Lemma \ref{l:Lindvall}. 
	  We
 need to define $U_i$,  $V_i$ for each $i \in [n]$ so that
$\LL(U_i) = \LL(S_{n})$, and $\LL(1+V_i)= \LL(S_{n}|Y_i=1)$,
and so that we can find a good bound for $\EE[|U_i-V_i|]$.
We do this for $i=1$ as follows. % First define the event
%$$
%\cE:= \{Y_1=1\} \cap \{\cC_{r_n}(X_1,\X_n) = \{X_1,\ldots,X_k\}\}.
%$$
Let $\tX_1 $ be a random vector in $\R^d$
with $\LL(\tX_1) = \LL(X_1|Y_1=1)$.
Also let $(X_{i,j}, i \in [n], j \in \N)$ be 
an array of independent $\nu$-distributed random variables,
independent of $\tX_1$.
Set $\X_{n,1} := \{X_{1,1},\ldots,X_{n,1}\}$.

For $2 \leq i \leq n$ set $J_i := \min \{j: \|X_{i,j} - \tX_1\| > r_n\}$
 %\notin \cup_{\ell=1}^k B_r(\tX_\ell) \}$ 
and set $\tX_{i} := X_{i,J_i}$.
Then set $\X_{n,2}:= \{\tX_1,\ldots,\tX_n\}$.

In other words, we sample the random vector $\tX_1$
from the conditional distribution of $X_1$ given
that $Y_1 = 1$, independently of $\X_{n,1}$.
Given the outcome of $\tX_1$,
	for $i \in \{2,\ldots, n\}$, if $| X_{i,1} - \tX_1 | >
r_n$
we take $\tX_i := X_{i,1}$. Otherwise
 we re-sample a random  vector with distribution $\nu$ repeatedly
until we get a value that is not in $ B_{r_n}(\tX_1)$,
and call this $\tX_i$.
Thus, given the value of $\tX_1$,
the distribution  of $\tX_i$ is given by the measure
$\nu $ restricted to $A \setminus B_{r_n}(\tX_1)$,
normalized to a probability measure. 

	For $x \in \R^d$, $\cX \subset \R^d$,
	we use the notation $h_n(x,\X) := {\bf 1} 
	\{
	\X \cap B_{r_n}(x) 
	\setminus \{x\} = \emptyset \}$.
	Let 
	$$
	U_1:= \sum_{i=1}^n h_n(X_{1,i},\X_{n,1}); ~~~~~
%	, which is
%	 the number of isolated vertices of $G(\cX_{n,1},r_n)$,
%	 and let
%	$
	V_1 := \sum_{i=2}^n h_n(\tilde{X}_{i}, \cX_{n,2}). 
	$$
%	which
	%is  the number of isolated vertices of 
%Set $U_1 := \scr S_{k,r}(\cX_{n,1})$, and
%$V_1 := K_{k,r}(\{\tX_1,\ldots,\tX_n\}) -1$.
%$G(\X_{n,2},r_n)$, minus 1.
%Note that in this coupling  we make no attempt to make  a `good'
%coupling of $(\tX_1,\ldots,\tX_k)$
%to $(X_{1,1},\ldots,X_{k,1})$. However, for
%$k+1 \leq i \leq n$ the variables $\tX_{i}$,
%are the same as $X_{i,1}$ except in
%those rare cases where $X_{i,1} $ lies in
%$B_r((\tX_1,\ldots,\tX_k))$. This is what makes
%the coupling effective.
Clearly $\LL(U_1)= \LL(S_{n})$. Also we claim that
	\begin{align}
	\LL(1+V_1)= \LL(S_{n}|Y_1=1). 
		\label{e:V1dist}
	\end{align}
	Indeed, the conditional distribution of $X_2,\ldots,X_n$, given
	that $X_1$ is a singleton and given also the location of
	$X_1$, is given by independent $d$-vectors each with
	distribution given by the restriction of $\nu$ to the
	complement of $B_{r_n}(X_1)$, normalized to a probability measure,
	which implies \eqref{e:V1dist}. For a more detailed proof of
	\eqref{e:V1dist}, see \cite[Lemma 8]{PY23} (Lemma 5.7 of
	the Arxiv version).
%\end{lemma}
%\begin{proof}
%Refer to other paper.
%\end{proof}

%\begin{proof}[Proof of (4.11)] 
%Proposition \ref{p:PoApproxBi}]
	We need to find a useful bound on $\E[|U_1-V_1|]$.
	Note that the `new' point process 
$\mathcal X_{n,2}=\{\tX_1,\ldots\tX_n\}$
	is obtained from the `old' point process
	$\X_{n,1}$ by moving the point
	$X_{1,1}$ to $\tX_1$, 
	and also, for those $i \in \{2,\ldots,n\}$ such
	that $J_i >1$, moving  the point $X_{i,1} $ to $\tX_i$,
	leaving the other points unchanged.

%	We refer to vertices and components (here also called {\em clusters})
%	of $G(\X_{n,1},r_n)$ as being {\em old} while vertices  and clusters of
%	$G(\cX_{n,2},r_n)$ are {\em new}.

	We claim $|U_1-V_1| \leq \sum_{i=1}^6 N_i,$ where we set:
	%is bounded by the sum of the following variables
%$N_i, 1 \leq i \leq 6$:
	\begin{align}
	N_1 & := h_n(X_{1,1}, \X_{n,1});
	\\
		N_2 & := \sum_{i=2}^n h_n(X_{i,1}, \cX_{n,1} \setminus \{X_{1,1}\})
		{\bf 1}\{ \|X_{i,1} - X_{1,1}\| \leq r_n\}
	\\
	N_3 & := \sum_{i=2}^n \sum_{j=2}^n
		h_n(X_{i,1}, \cX_{n,1} \setminus \{X_{1,1}\})
		{\bf 1} \{ \| X_{i,1} - \tilde{X}_j \| \leq r_n,
		J_j > 1, i \neq j \}; 
	\\
	N_4  & := \sum_{i=2}^n
		h_n( \tilde{X}_i, \X_{n,1} \setminus \{X_{1,1}\})
		{\bf 1}\{ J_i > 1, \| \tX_i - \tX_1 \| > 2 r_n \};
	\\
		N_5 & := \sum_{i=2}^n h_n(X_i, \X_{n,1} \setminus \{X_{1,1}\})
		{\bf 1} \{ \|X_i - \tilde{X}_1 \| \leq r_n\};
	\\
	 N_6 & := \sum_{i=2}^n h_n ( \tilde{X}_i ,\X_{n,2}) {\bf 1} \{ r_n <
		\|\tilde{X}_i - \tilde{X}_1 \| \leq 2r_n\}.
	\end{align}
	Indeed, given $i \in \{2,\ldots,n\}$, for the $i$th vertex to
	contribute to $U_1$ but not to $V_1$ 
	we must have $X_{i,1}$ connected to  $\tilde{X}_j$ for some $j \in
	\{2,\ldots,n\} \setminus \{i\}$ with $J_j > 1$, but otherwise isolated,
	and hence
	counting towards $N_3$; or $X_{i,1}$ connected to $\tX_1$ but
	otherwise isolated, and hence counting towards $N_5$.
	For the $i$th vertex to contribute to $V_1$ but not to $U_1$,
	we must have
	  either $X_{i,1}$  connected to $X_{1,1}$ but otherwise isolated
	(hence counting towards $N_2$); or the $i$th vertex moved to
	an isolated location distant more than  
	%$J_i > 1$ and $\tX_i$ isolated
	$2r_n$ from $\tX_1$ (hence counting towards $N_4$);
	or $\tX_i$ isolated and distant at most $2r_n$ from
	$\tX_1$ (hence counting towards $N_6$).

%	$N_{1}$ is the number of old $1$-clusters involving 
%	$X_{1,1}$,
%	and $N_2$ is the number of new $1$-clusters within distance $r_n$
%of $X_{1,1}$ but not using any of the
%new 	vertices  $\tX_i$ with $i>k, J_i>1$, and with no
%	vertex in $B_{2r_n}(\tX_1)$.
%These will be affected by removing $X_{1,1}$.
%
%
%$N_3$ is the number of old $1$-clusters intersecting $B_{r_n}(\tX_i)$
%	for some $i  \in \{2,\ldots, n\}$ with $J_i > 1$, and not using
%	$X_{i,1}$ (if such a cluster does use $X_{i,1}$ it is included in
%	$N_5$ below).
%	$N_4$ is the number of new $1$-clusters involving $\tX_i$
%	for some $i >1$ with $J_i >1$, and having no vertex in 
%	$B_{2r_n}(\tX_1)$.
%	These are affected by the creation of new vertices at $\tX_i$
%	with $i >1, J_i>1$.
%
%
%	$N_5$ is the number of old $k$-clusters having at least one 
%	vertex in $B_{r_n}(\tcC)$.  These clusters are affected by the
%	removal of old vertices in $B_{r_n}(\tcC)$.
%
%
%$N_6$ is the number of new singletons lying in
%	$B_{2r_n}(\tX_{1})$,
%	other than $\tX_1$ 
%itself. These could be created from previously larger clusters
%	due to the removal of old vertices in 
%	$B_{r_n}(\tX_1)$.

	We estimate $\EE[N_i]$ for $i=1,2,\ldots,6$,
	%each $i \in [6]$,
	repeatedly using the fact that $\nu(B_s(x)) \geq (\theta_d/4)f_0s^d$
	for all small enough $s>0$ and 
	all $x \in A$
	by Lemma %\ref{l:bdy_est}.
	\ref{l:A3}.
	We have for large enough $n$ that
	\begin{align*}
		\EE[N_1] & \leq  
		\int_{A}  (1- \nu(B_{r_n}(x)))^{n-1} \nu(d x)
		\\
		& \leq 2   %(\fmax \theta_d ((k-1)r_n)^d)^{k-1}
		\exp(-f_0 (\theta_d/4) n r_n^d),
	\end{align*}
	and
	%writing $h_r(x,\bx)$ for $h_r((x,\bx))$, we have
	\begin{align*}
		\EE[N_2] & \leq (n-1)
		%k \binom{n-k}{k}
		\int_A \int_{B_{r_n}(x)} % h_r(x,\bx)
		(1- \nu(B_{r_n}(y)))^{n-2} 
		\nu(d y) \nu(dx)
		\\
		& \leq 2  n \fmax \theta_d r_n^d
		\exp(-f_0 (\theta_d/4) n r_n^d).
	\end{align*}
	Using the fact that $\Pr[J_i>1] \leq  \fmax \theta_d r_n^d$
	for $i >1$, and the point process
	$\X_{n,1} \setminus \{X_{i,1}\}$ is independent of 
	the event $\{J_i >1\}$ and the random vector $\tX_i$, we obtain that
	\begin{align*}
		\EE[N_3] & \leq (n-1)^2  \fmax \theta_d r_n^d
		%\binom{n-1}{k}
		\sup_{x \in A}
		\int_{B_{r_n}(x)} (1- \nu(B_{r_n}(y)))^{n-2}
		\nu(d y)
		\\
		& \leq c (nr_n^d)^{2} \exp(- f_0 (\theta_d/4) n r_n^d).
	\end{align*}
Using the same bound on $\Pr[J_i >1]$ and the fact that for $i \in
\{2,\ldots,n\}$ 
the  distribution
of $\X_{n,2} \setminus \{\tX_1,\tX_i\},$ 
given $(\tX_1,\tX_i)$ is that of
a sample of size $n-2$ from the restriction of $\nu$ to $A \setminus
B_{r_n}(\tX_1)$ normalized to be a probability measure,
we have that
	\begin{align*}
		\EE[N_4] & \leq (n-2)  \fmax \theta_d r_n^d
		%\binom{n-k-1}{k-1}
		\sup_{x \in A} 
		%\int_{A^k}
		\left\{
		%	(2 \fmax \theta_d (kr)^d)^{k-1}
		(1- \nu(B_{r_n}(x)))^{n-2} \right\}
		\\
		& \leq c n r_n^d \exp(-f_0 (\theta_d/4) n r_n^d).
	\end{align*}
	Next, by conditioning on $\tX_1$, we have
	\begin{align*}
		\EE[N_5] & \leq n  \fmax \theta_d r_n^d
		%\binom{n-1}{k-1} 
		\sup_{x \in A}
		%\int_{A^{k-1}} h_r(x,\bx)
		(1- \nu(B_{r_n}(x)))^{n-1} %\nu^{k-1}(d\bx)
		 \\ & \leq c nr_n^d \exp(- f_0 (\theta_d/4) nr_n^d).
	\end{align*}
	%For $N_6$, note that for each $i >k$ we have $\Pr[J_i >1] =O(nr^d)$,
	%and if $J_i >1$ then $\tX_i$ has probability $O(r^d)$ of being near
	%$\tX_1$. Hence 
	%$$
	%\E[N_6] = O(nr^{2d}) = O(\exp(-c nr^d)),
	%$$
	%for some $c>0$.

%	In our case 
	$N_6$ is the number of singletons
	of $G(\X_{n,2},r_n)$ within distance $2r_n$ of $\tilde{X}_1$. Each
	vertex has probability $O(r_n^d)$ of lying in
	$B_{2r_n}(\tilde{X}_1) \setminus B_{r_n}(\tilde{X}_1)$,
	and given its location and that of $\tilde{X}_1$, using 
	\eqref{e:volLB1} from Lemma \ref{l:A1} without
	assuming $x \prec y$ there (and therefore requiring
	$\partial A \in C^2$), it has
	probability at most $e^{-2 \delta_1 f_0 nr_n^d}$ of being isolated,
	for some $\delta >0$.
	Thus we can obtain that $\E[N_6] = O(nr_n^d \exp(- 2 \delta_1 f_0
	nr_n^d))$.
	%which is what we need. 

	Combining  these  estimates for
	$N_1,\ldots,N_6$, we obtain that 
	%there exists $\eps >0$ that 
	$$
	\E [ |U_1-V_1|] = O(\exp(-  \delta_1 f_0 n r_n^d)).
	$$
By the exchangeability of $X_1,\ldots,X_n$, we can construct
$U_i, V_i$ similarly for each $i \in [n]$, with the same bound
for $\E[|U_i-V_i|]$. %Also using \eqref{c:lowerbound}, \eqref{c:upperbound},
%\eqref{e:tIdef} and 
%Lemma \ref{l:diff_I_Itilde}, we can and do choose $\delta_9 >0$ such that
%$\E[S_n ] \geq \delta_9$ for all large enough $n$.
%
%we know
%$\E[S_{n}] \to \infty$ as $n \to \infty$ by Proposition 
%\ref{p:Snmeanbin} and Lemma  \ref{l:Ilower}.
Then by Lemma \ref{l:Lindvall} (with the $W$ of that result
equal to our $S_{n}$) we obtain for  $n$ large enough  that
\begin{align*}
	\dtv(S_{n},Z_{\E[S_{n}]} ) & \leq
	\min \Big( 1, \frac{1}{\E[S_n]} \Big)
	%(\delta_9^{-1}/\E[S_{n}])
	%^{-1}
	\sum_{i=1}^n \E[Y_i]
	\times O(\exp(- \delta_1 f_0 n r_n^d))
%	\\
	 = O(e^{- \delta_1 f_0 n  r_n^d}),
\end{align*}
as required.
\end{proof}

%	Similarly \eqref{e:dtvSZ} is proved in \cite[Proposition 5.6]{PY23},
%taking $k=1$ there, under the extra assumption that
%	%which says that
%	%\bea
%	%\dtv(S_n,Z_{\tilde{I}_n}) \leq e^{-cn r^d}.
%	%\label{e:dtvSbZ}
%	%\eea
%	%Again, it is assumed in \cite{PY23} that 
%	$b^+ \leq 1/\max(f_0, d(f_0-f_1/2))$,
%	which we  wish to drop.
%	%and we wish to drop this assumption.
%	Going through the proof of \cite[Proposition 5.6]{PY23},
%	we find that it still works without the extra assumption
%	on $b^+$. Since  $k=1$ for us, we can give a simpler proof of 
%	\cite[Lemma 5.8]{PY23} 
%	than the one given there
%	for general  $k$, as follows.
%	
%\end{proof}

%The following CLT
%for $S_n$ and $S'_n$, defined at \eqref{e:def_Sn},
	%is a corollary of Lemma \ref{l:dTV}.
%\begin{proposition}
%	\label{l:CLTS}
%	Suppose $nr_n^d \to \infty$ and $I_n \to \infty$ as $n \to \infty$. Then
%	there exists $\delta >0$ such that as $n \to  \infty$,
%\end{proposition}
\begin{proof}[Proof of Proposition \ref{p:Sall}]

	The assertion 
	$\E[\zeta_n]= I_n (1+ O(e^{-\delta nr_n^d}))$ 
	follows from the Mecke formula (in the case $\zeta_n= S'_n$)
	and from Lemma \ref{l:diff_I_Itilde}
	(in the case $\zeta_n = S_n$).
	The assertion 
	$\Var[\zeta_n]= I_n (1+ O(e^{-\delta nr_n^d}))$ follows
	from Proposition \ref{p:Sn} (if  $\zeta_n = S'_n$)
	and from Proposition \ref{p:var_iso_bin}
	(if  $\zeta_n = S_n$).

	  By the Berry-Esseen theorem $\dk(t^{-1/2}(Z_t-t),N(0,1))=
        O(t^{-1/2})$ as $t \to \infty$. Hence, by
                the triangle inequality for $\dk$,
                %and the fact that $\E[Z_n ]= I_n$,
		we have
\begin{align}
	\dk( I_n^{-1/2}(S'_n- I_n),N(0,1))\le
	\dk(S'_n,Z_{I_n}) + O(I_n^{-1/2}).
        \label{e:PotoNbin}
\end{align}
	Then using
	\eqref{e:dtvS'Z}, and
	the obvious inequality $\dk \leq \dtv$,
	we obtain \eqref{e:CLTS'}.
	%By \cite[Proposition 5.6]{PY23},
        %the first term in the right hand side of (\ref{e:PotoNbin})
        %is $O(e^{-cnr^d})$ for some $c>0$.
	Similarly using \eqref{e:dtvSZ} and Lemma \ref{l:diff_I_Itilde}
	we obtain
%Once we have \eqref{e:dtvSbZ} we can finish the proof of
	\eqref{e:CLTS}.
	%the same way as for \eqref{e:CLTS'}.
\end{proof}

	\subsection{Asymptotics for $I_n$ in the uniform case}
	\label{s:Iunif}

Throughout this subsection we make the additional  assumption
that $f \equiv f_0 \1_A$, with
$f_0 = 1/\lambda(A)$,
and assume as $n \to \infty$ that
$r_n$ satisfy  
\begin{align}
nr_n^d \to \infty; ~~~~~~~~~~
	\limsup (nf_0 \theta_d r_n^d - (2-2/d) (\log n -
	\1\{d \geq 3\} \log \log n) ) < \infty.
	%\to -\infty.
	\label{e:supcri5}
\end{align}
We shall demonstrate  the asymptotic equivalence  of
$I_n$ and $\mu_n$, defined at 
\eqref{e:def:I_n} and \eqref{e:defI'} respectively.
	Given $n$, for $x \in A$ set $p_n(x) := \exp(-n \nu(B_{r_n}(x)))$.
	Given Borel $S \subset A$, let
\begin{align}\label{e:E[F_nS]}
	I_n(S) : = nf_0 \int_S  p_n(x) dx.
\end{align}

\begin{proposition}[The case $d=2$]
 \label{p:average2d}
	Suppose $d=2 $ and either $\partial A \in C^2$,
	or $A= [0,1]^2$. Suppose \eqref{e:supcri5} holds.
	%i.e.  $r=r(n)$ satisfies
%\begin{align}
%	nr^2 \to \infty ~~~~~~{\rm  and} ~~~~~~
%	n f_0  \pi r^2 - \log n \to - \infty  
%	~~~~~~~~~~~ {\rm as}~ n \to \infty.
%	\label{e:supcri4}
%\end{align}
	Then we have as $n \to \infty$ that 
	%{\bf (Add error bound later?)}
	\begin{align}
		I_n =  n \exp(-n  f_0 \pi r_n^2) (1+ O(nr_n^2)^{-1/2}).
		\label{e:In2d}
	\end{align}
\end{proposition}

\begin{proof}
	{\bf Case 1: $\partial A \in C^2$.}
	In this case the result follows from Proposition \ref{p:average3d+}
	below,
	since
	the ratio between the two terms in the right hand side of 
	\eqref{e:EFconv} is given by
	\begin{align*}
		\frac{
	 e^{-n f_0 \theta_d r_n^d/2} \theta_{d-1}^{-1} r_n^{1-d}
		|\partial A| }{ 
		n  e^{-n f_0 \theta_d r_n^d}}
		= O\Big( \exp \big( \frac{n f_0 \pi r_n^2}{2} - \frac{\log n}{2}
		- \frac{\log(nr_n^2)}{2} \big) \Big),  
	\end{align*}
	which is $O((nr_n^2)^{-1/2})$
%	tends to zero 
	by
	\eqref{e:supcri5}.
%	our conditions on $r$.
%	Therefore the
%	first term in the right hand side of \eqref{e:EFconv} dominates,
%	and \eqref{e:In2d} follows in this case.

{\bf Case 2: $A= [0,1]^2$.}
In this case $f_0=1$.
	Define 
	%the `interior region' contribution
	%$I_n^{{\rm int}} := n \int_{A^{(-r)}} e^{-n \nu(B_r(x))}
	% dx$, and
	the `moat'
$\mathsf{Mo}_n := A \setminus A^{(-r_n)}$.

For $1 \leq i \leq 4$
let $\mathsf{Cor}_{n,i}$  be the region of $A$ within $\ell_\infty$-distance 
	$r_n$ of the $i$th corner of $A$ (a square of side $r_n$).
	Then
%\begin{align*}
$	I_n  ({\mathsf{Cor}_{n,i}})
	%n  \int_{\mathsf{Cor}_{n,i}} e^{-n \nu(B_r(x))} dx
	 \leq  n r_n^2 e^{- n \pi r_n^2/4}. 
	%\le c nr^2 e^{-\kappa f_0 \pi nr^2},
	$
%\end{align*}
%leading to a total contribution 
%from these corner regions 
% $O((n r^2)e^{-\kappa f_0 n \pi r^2 })$.

	 The set $\mathsf{Mo}_n \setminus \cup_{i=1}^4
	\mathsf{Cor}_{n,i}$ is a union of 
	rectangular regions $\mathsf{Rec}_{n,i}$, $1 \leq i \leq 4$.
	For all $n$ large enough, $i \leq 4$
and $x \in \mathsf{Rec}_{n,i}$
	we have $\nu(B_{r_n}(x)) \geq  r_n^2(\frac{\pi}{2} + a(x)/{r_n})$,
	where $a(x)$ denotes the distance from $x$ to $\partial A$. 
	%$\delta $ is a positive constant.
Hence
\begin{align*}
	I_n( \mathsf{Rec}_{n,i}) 
	\le n {r_n} e^{-n  \pi r_n^2/2}   
	\int_0^1 e^{-n r_n^2  a} da
	= O(n {r_n} e^{- \pi nr_n^2/2} (nr_n^2)^{-1} ).
\end{align*}
Combined with the corner region estimate, and the bound
$\pi nr_n^2 \leq \log n + c$ from \eqref{e:supcri5}, this yields
\begin{align*}
	I_n(\mathsf{Mo}_n )/I_n(A^{(-{r_n})})
	& =
	O(n^{-1} r_n^{-1} e^{ n \pi r_n^2/2} ) 
	+ O( r_n^{2} e^{ (3/4)  n \pi r_n^2} ) 
\\
	& = O( n^{-1/2} r_n^{-1}) + O(((\log n)/n) n^{3/4} ).
\end{align*}
Also $I_n(A^{(-{r_n})}) = ne^{-n \pi r_n^2}(1+ O({r_n}))$
and $r_n= o((nr_n^2)^{-1/2})$.  Putting together these estimates
yields (\ref{e:In2d}) in Case 2.
\end{proof}

\begin{proposition}[The case $\partial A \in C^2$] 
	\label{p:average3d+}
	Suppose %$f \equiv f_0\1_A$, with 
	$d \geq 2$ and $\partial A \in C^2$.
	Suppose $(r_n)_{n \geq 1}$ satisfy \eqref{e:supcri5}.
	Then as $n \to \infty$,
\begin{align}
	I_n = n  e^{-n f_0 \theta_d r_n^d}
	+
	 e^{-n f_0 \theta_d r_n^d/2} \theta_{d-1}^{-1}
	 |\partial A|
	 r_n^{1-d}
	 \Big( 1+ O \Big( \big(\frac{\log (nr_n^d)}{nr_n^d} \big)^2
	 \Big) \Big). 
	\label{e:EFconv}
\end{align}
\end{proposition}
As in Section \ref{ss:geom}, for $x\in A$ let $a(x): = \dist (x,\partial A)$,
the Euclidean distance from $x$ to $\partial A$, and
for $s \geq 0$ let $g(s) := \lambda(B_1(o) \cap ([0,s] \times \R^{d-1}))$.
To prove Proposition \ref{p:average3d+} we shall use the following result
from \cite{HPY24}.

\begin{lemma}
	%[Reparameterization]
        \label{thm:cov}
	If $d \geq 2$ and $\partial A \in C^2$,
        there are positive finite  constants $c = c(A), r_0 = r_0(A)$, such that
%For $r \in (0,\frac{1}{2}\tau(\partial A))$,
for all $r \in (0,r_0)$, and all bounded measurable
         %$h : \partial A^{(r)} \to [0,\infty)$,
	 $\Psi: [0,r) \to [0,\infty)$,
%\begin{equation}
%        \label{eq:change-of-variables}
%        \Big |\int_{\partial A^{(r)}} \psi(y) \, d y
%        - \int_{0}^{r} \int_{\partial A} \psi(z + s \hat n_z) \, d z \,
%        d s \Big|
%        \leq c r
%         \int_{0}^{r} \int_{\partial A} \psi(z + s \hat n_z) \, d z \, d s,
%\end{equation}
%where the inner integral is a surface integral.
%        If $\psi(y)$ depends only on $\dist(y, \partial A)$,
%        i.e.\  there exists $\Psi :[0,r_0) \to \R$ such that
%        $\psi(z + s \hat n_z) = \Psi (s)$ for all $(z,s) \in \partial A \times
%        (0,r_0]$,
%then
        %\eqref{eq:change-of-variables} becomes
\begin{equation}
	\Big| \int_{A \setminus A^{(-r)}} \Psi(a(y)) \, d y
        - |\partial A| \int_0^{r} \Psi (s) \,d s \Big|
        \leq c  r |\partial A| %|\partial M|
        \int_0^{r} \Psi (s) \,d s.
        \label{e:CoVcoro}
\end{equation}
\end{lemma}
\begin{proof}
	See \cite[Proposition 3.8]{HPY24}.
\end{proof}

	\begin{proof}[Proof of Proposition \ref{p:average3d+}]
		We refer to $A^{(-r_n)}$ as the {\em bulk}.
To deal with this region, note that for each $x\in A^{(-r_n)}$, we have
$p_n(x)= e^{-nf_0 \theta_d r_n^d}$ so that by (\ref{e:E[F_nS]}),
\begin{align}
	I_n(A^{(-r_n)})=
	nf_0 \lambda(A^{(-r_n)}) e^{-nf_0 \theta_d r_n^d}
	= (1+O(r_n))
	\label{e:Inblk}
	n e^{-nf_0 \theta_d r_n^d}.
\end{align}

It remains to deal with the region
		$ \mathsf{Mo}_n := A \setminus A^{(-r_n)}$ (which we call 
		the {\em moat}).
This is the region within distance $r_n$ of $\partial A$.
%As above, for $a >0$ let $g(a) := \lambda(B_1(o) \cap ([0,a] \times \R^{d-1}))$,
%and for $x \in A$ let $a(x) $ be the distance from $x$ to
%$\partial A$.
For each 
$x\in \mathsf{Mo}_{n}$, 
		%using (\ref{0131a})
		we have
\begin{align*}
|p_n(x) - e^{-nf_0  r_n^d(\frac{\theta_d}{2}+g(\frac{a(x)}{r_n}))}| 
	 & \le  e^{-nf_0  r_n^d(\frac{\theta_d}{2}+g(\frac{a(x)}{r_n}))}
	\\
	 & \times \Big|  \exp\Big(nf_0 \Big( r_n^d \big( 
	 \frac{\theta_d}{2} + g(\frac{a(x)}{r_n}) \big) - 
	 \lambda(B_{r_n}(x) \cap A) \Big) \Big)  -1 \Big|.
\end{align*}
%with $\de=2Kr_n^{3/2}$.
		Using the inequality $|e^{s} -1| \leq 2|s|$ for $s \in [-1,1]$,
and (\ref{e:supcri5}), and 
		%the fact that $g'(a) \leq  \theta_{d-1}$ for all $a$,
		Lemma \ref{l:bdyvol}
		we obtain that there exists a constant $c$ such that
		for all $x \in \mathsf{Mo}_n$,
\begin{align*}
|p_n(x) - e^{-nf_0 r_n^d(\frac{\theta_d}{2}+g(\frac{a(x)}{r_n}))}| &
\le c %(\log n)^{1-\frac{1}{d}} n^{-2\al + \frac{1}{d}}
	n r_n^{d+1}
	e^{-nf_0 r_n^d(\frac{\theta_d}{2}+g(\frac{a(x)}{r_n}))}.
\end{align*}
		Integrating over $\mathsf{Mo}_n$
and using (\ref{e:E[F_nS]}), we obtain that
\begin{align*}
	I_n(\mathsf{Mo}_n) =
	n f_0 \int_{\mathsf{Mo}_{n}} p_n(x) dx = 
	( 1 + O( n r_n^{d+1}))
	nf_0 
	\int_{\mathsf{Mo}_n} 
	e^{-nf_0 r_n^d(\frac{\theta_d}{2}+g(\frac{a(x)}{r_n}))}da .
	%|\mathsf{BPri}_{n,i}| r_n \int_0^1 e^{-nf_0 r_n^d(\frac{\theta_d}{2}+g(a))} da. 
\end{align*}
Hence, using Lemma \ref{thm:cov} and the fact that $r_n = o(nr_n^{d+1})$ by
\eqref{e:supcri5}, we obtain that
\begin{align}
	I_n(\mathsf{Mo}_n) =
	( 1 + O( n r_n^{d+1}))
	n f_0 |\partial A| \int_0^1 r_n e^{-n f_0 r_n^d ( \frac{\theta_d}{2} + g(a))} da.
	\label{0131b}
\end{align}

Next we claim that as $n \to \infty$,
\begin{align}\label{e:error_const}
\int_0^1 nf_0 r_n e^{-nf_0  r_n^d(\frac{\theta_d}{2}+g(a))} da 
	& = e^{-n f_0 \theta_d r_n^d/2} \theta_{d-1}^{-1} r_n^{1-d}
	\Big(   1+ O \big( \frac{\log (nr_n^d)}{nr_n^d} \big)^2  \Big) ,
\end{align}
To prove this,
we first notice that 
$g(a)= \theta_{d-1} \int_0^a (1-t^2)^{(d-1)/2}dt $, $0 \leq a \leq 1$.
Therefore, we have (i) $g(0)=0$, $g$ is increasing and 
$ g(a)\le \theta_{d-1}a$, (ii) for any $\ep \in (0,1/2)$ there exists
$\delta \in (0,1)$ such that for all $a\in (0,\delta )$,
$g(a)\ge (1-\ep)\theta_{d-1}a$, and we claim that (iii)
upon choosing a smaller $\delta$ in (ii), we also have 
$g(a)\ge \theta_{d-1}(a - d a^3)$
for $a\in (0,\delta )$. To justify  (iii), 
we use the Taylor expansion $(1-t^2)^{(d-1)/2} =
1- % \left(  
\big( \frac{d-1}{2} \big)
%\right)
t^2 + O(t^4)$ as $t \downarrow 0$,
so that $\theta_{d-1}^{-1} g(a) = a - \big( \frac{d-1}{6} \big)
a^3 + O(a^5)$ as $ a \downarrow 0$, and the fact that $(d-1)/6 < d$.

By item (i), we have
\begin{align}
	nf_0 \theta_{d-1} r_n^d \int_0^1 e^{-nf_0  r_n^d g(a)} da & \ge
	n f_0 \theta_{d-1} r_n^d \int_0^1 e^{-n r_n^d f_0\theta_{d-1} a} da
	\nonumber
\\
	& = 1- e^{-nf_0 \theta_{d-1} r_n^d}.
	\label{e:0117a}
\end{align}
Let $\eps_n\in (0,\delta)$. By item (iii), we have
\begin{align}
nf_0 \theta_{d-1} r_n^d \int_0^{\eps_n} e^{-nf_0  r_n^d g(a)} da &\le nf_0 \theta_{d-1} r_n^d \int_0^{\eps_n} e^{-nf_0 \theta_{d-1}r_n^d a(1- d \eps_n^2)} da \le
	1+c\eps_n^2.
	\label{e:0117b}
\end{align}
By item (ii), we have
\begin{align}
nf_0 \theta_{d-1} r_n^d \int_{\eps_n}^\delta e^{-nf_0  r_n^d g(a)} da
	&\le nf_0 \theta_{d-1} r_n^d \int_{\eps_n}^{\delta}
	e^{-nf_0r_n^d(1-\ep) \theta_{d-1}a} da
	\nonumber \\
	&\le 2 \exp(-nf_0r_n^d \theta_{d-1}(1-\ep)\eps_n).
	\label{e:0117c}
\end{align}
Moreover, using \eqref{e:supcri5}
it is easy to see that for $n $ large
\begin{align}
nf_0 \theta_{d-1} r_n^d \int_\delta^1 e^{-nf_0  r_n^d g(a)} da\le 
	\exp(-nr_n^d f_0  g(\eps_n)/2  ). 
	\label{e:0117d}
%
%	c (\log n)^{1+(2-2/d)g(\de)} n^{-(2-2/d)g(\de)}.
\end{align}

Set $u_n := nr_n^d$, and note
$u_n \to \infty$ as $n \to \infty$ by 
\eqref{e:supcri5}.  The right hand side of \eqref{e:0117a}
is $1- O(u_n^{-3})$.
 We take  $\eps_n = c' (\log u_n )/ u_n$ with a big constant $c'$;
 then the right hand side of \eqref{e:0117b} is 
 $1 + O \big( \big( \frac{\log u_n}{u_n} \big)^2 \big)$. 
 Provided $c'$ is big enough, the right hand side of
 \eqref{e:0117c} is  $O(u_n^{-3})$, as is the right hand
 side of \eqref{e:0117d}. Thus
combining these four estimates
yields
\begin{align}\label{e:conv1}
 nf_0 \theta_{d-1} r_n^d \int_0^1 e^{-nf_0  r_n^d g(a)} da  
	= 1 + O \Big( \big(  \frac{\log (nr_n^d)}{nr_n^d} \big)^2 \Big).
\end{align}
Since the left side of (\ref{e:conv1}), multiplied by
	$\theta_{d-1}^{-1} r_n^{1-d} e^{-nf_0\theta_d r_n^d/2}$,
comes to the left side of (\ref{e:error_const}), \eqref{e:conv1} yields \eqref{e:error_const}.

By \eqref{e:supcri5},
	$\big(\frac{\log(nr_n^d)}{nr_n^d} \big)^2 = \Omega( (\log n)^{-2})$ and
$nr_n^{d+1} = % O(r_n^{1/4}) =
o( (\log n)^{-2})
= o \Big( \big( \frac{\log nr_n^d}{(nr_n^d)} \big)^2 \Big) $.
Therefore combining \eqref{0131b} and \eqref{e:error_const} 
leads to
\begin{align}
		I_n(\mathsf{Mo}_n) =
	e^{-n f_0 \theta_d r_n^d/2} \theta_{d-1}^{-1} |\partial A| r_n^{1-d}  
	\Big(1 + O \Big( \big( \frac{\log (nr_n^d)}{nr_n^d} \big)^2 \Big)
	\Big).
	\label{0131c}
\end{align}

 The error term in  the right hand side of \eqref{e:Inblk}, divided by the leading-order  term in
 the right hand side of \eqref{0131c},
 %the preceding expression for
%		$I_n(\mathsf{Mo}_n) $,
satisfies
 \begin{align*}
	 \frac{n r_ne^{-n \theta_d f_0 r_n^d}}{e^{-n \theta_d f_0 r_n^d/2} r_n^{1-d}}
	 = O( nr_n^d e^{-n \theta_d f_0 r_n^d/2} ) =
	 O \Big( \big( \frac{\log (nr_n^d)}{nr_n^d} \big)^2 \Big),
 \end{align*}
 Thus combining  \eqref{0131c} with
 \eqref{e:Inblk} shows that \eqref{e:EFconv} holds.
\end{proof}

\section{Proof of first-order asymptotics}
\label{s:moments}

Throughout this section we make the same assumptions
on $\nu$ and $A$ that we set out at the start of Section
\ref{s:prelims}.
%assume that $d \geq 2$ and
%$A \subset \R^d$ is compact and connected with
%$\partial A \in C^2$ and $A= \overline{A^o}$, or $d=2$ and $A =[0,1]^2$,
%and that $f$ is continuous on $A$ with $f_0 >0$.
We also assume that
\eqref{c:lowerbound} and \eqref{c:upperbound} hold,
i.e. that $nr_n^d \to \infty$ and $\liminf(I_n) >0$
as $n \to \infty$. 
We note for later use that the latter assumption,
together with Proposition
		\ref{p:bc}, implies
		\begin{align}
			nr_n^d = O(\log n) ~~~~ {\rm as~} n \to \infty.
			\label{e:RLB}
		\end{align}

We shall prove that if $\xi_n$ denotes any of $K_n-1, K'_n-1, R_n$ or
$R'_n$, then both $\E[\xi_n]$ and $\Var [\xi_n]$ are asymptotic
to $I_n$
(which was  defined at \eqref{e:def:I_n}) as $n \to \infty$;
we will then be able to prove the first-order convergence
results from Section \ref{s:stateresults}, i.e.
Theorems \ref{t:momLLN}, \ref{t:exprate}, \ref{t:PoLim} and
	\ref{t:PoLimUn}.

To achieve this goal,
we shall consider separately
the contributions to $\xi_n$ from non-singleton components
of $G(\X_n,r_n)$ or $G(\eta_n,r_n)$ that are
 small, medium or large. Here, given fixed $\rho > \eps > 0$,
 we say a component is {\em small} (respectively {\em medium},
 {\em large}) if
 its Euclidean diameter is less than $\eps r_n$ 
 (resp.,  between $\eps r_n$ and $\rho r_n$, greater than $\rho r_n$).
 We shall make appropriate choices of the constants $\eps, \rho$
 as we go along.
 %with $r=r(n)$ as usual.
 %for some fixed small n $\eps r$ and $\rho r$, greater than $\rho r$.
 %$\eps >0$, and large if its diameter exceeds $\rho r$ for some
 %large fixed $\rho < \infty$. 

%\label{s:EK}
%Under the conditions given at the start of Section \ref{s:sngltn},
%we now determine the
% asymptotic behaviour of $\E[K_n]$,  $\E[K'_n]$,
%  $\E[R_n]$ and  $\E[R'_n]$. 

  %\subsection{Asymptotics for $\E[K_n]$ and $\E[K'_n]$}

%Clearly $I_n \leq \E[K'_n]$. We shall 
 %show (in Proposition \ref{p:J} below) that as $n \to \infty$
%this lower bound for $\E[K'_n]$ is sharp, i.e. 
 %$\EE[K'_n]-I_n = o(I_n)$, so that $\EE[K'_n] \sim I_n $, 
 %and likewise that $\EE[K_n] \sim I_n $.  In Proposition \ref{p:exp_R_bin}
 %we shall show that $\E[R'_n] \sim I_n$ and $\E[R_n]  \sim I_n$.
 %%as $n \to \infty$.

 %To achieve our aim,
 %we need to estimate the expected number of vertices in
 %non-singleton components
 %of $G(\X_n,r)$ or $G(\eta_n,r)$; we shall consider separately 
 %components of various sizes (measured by Euclidean diameter). 
%We shall argue that a crowded component confined in a small space is unlikely. On the other hand, if a component occupies a large space, then it necessarily has  a large empty space surrounding it, which is unlikely because the average number of points $nr^d$ per cube of side length $r$ is high.
%%In between the two extremes, we argue differently. {\bf [Edit paragraph]} 

For finite $\X \subset \R^d$, $x \in \X$, and $n \geq 1$ we let
$\scr F_n(x,\X) $ denote the event that $x$ is the first
element of $\cC_{r_n}(x,\X)$ (defined in  Section
\ref{ss:perc}) in the $\prec$ ordering (defined in Section
\ref{ss:geom}), i.e.
\begin{align}
	\scr F_n(x,\X) :=  \{ 
	x \prec y ~~\forall ~~ y \in \cC_{r_n}(x,\X) \setminus \{x\} \}.
	\label{e:defZ}
\end{align}
Given $n$ and $(r_n)_{n \geq 1}$,
for $0 \leq \eps  < \rho \leq \infty$ we define
$K_{n,\eps ,\rho}(\cX)$
%(respectively $K'_{n,\eps ,\rho}$)
to be the number of components  of $G(\cX,r_n)$
%$G(\X_n,r)$ (respectively $G(\eta_n,r)$) 
that have Euclidean diameter in the range $(\eps r_n,\rho r_n]$,
and $R_{n,\eps,\rho}$ to be the number of vertices in such components,
that is,  with event $\scr M_{n,\eps,\rho}(x,\X)$ defined at
\eqref{e:medevent1}, 
\begin{align}
	K_{n,\eps ,\rho}(\cX)  := \sum_{x \in \X} \1_{
		\scr M_{n,\eps,\rho}(\X)
		\cap
		\scr F_n(x,\X)
		}
	;
	\label{e:Knrdef}
	%\\
	~~~~~
	 R_{n,\eps,\rho}(\cX)  := 
	 \sum_{x \in \X} {\bf 1}_{ 
	 \scr M_{n,\eps,\rho}(x,\cX)} .
\end{align}
We then define the random variables
$K_{n,\eps,\rho} := K_{n,\eps,\rho}(\cX_n)$
and $K'_{n,\eps,\rho} := K_{n,\eps,\rho}(\eta_n)$.
 %With notation $\cC_s(x,\X)$ defined at the start of Section \ref{ss:perc}, 
 %for 
%	$0 \leq \eps < \rho \leq \infty$, $n >0$ and finite $\cX \subset \R^d$,
%	 Then define the random variables 
Also we set
	  $R_{n,\eps,\rho} := R_{n,\eps,\rho}(\cX_n)$
	  and $R'_{n,\eps,\rho} := R_{n,\eps,\rho}(\eta_n)$.

	  \subsection{Asymptotics of means}
	  \label{ss:means}

We shall bound the expected number of `small' non-singleton components,
$\E[K_{n,0,\eps }]$,
using the following lemma.

\begin{lemma}
\label{l:smallK}
	There exist $\delta_9  \in (0,1)$ and $c, n_0 <\infty$ such that 
	for all $n\ge n_0$ and any $x\in A$, 
\begin{align}
	\PP[\scr F_n(x,\eta^x_n)  \cap \{0<\diam(\cC_{r_n}(x,\eta_n^x)) 
	\leq \delta_9  {r_n}\}] 
	\le c(nr_n^d)^{1-d} e^{-n\nu(B_{r_n}(x))}; 
	\label{e:Zub}
	\\ 
	\PP[\scr F_n(x,\X^x_{n-1})  \cap \{0<\diam(\cC_{r_n}(x,\X_{n-1}^x)) 
	\leq \delta_9  {r_n}\}] 
	\le c(nr_n^d)^{1-d} e^{-n\nu(B_{r_n}(x))}.
	\label{e:Zubbin}
\end{align}
\end{lemma}
\begin{proof}
	See \cite[Lemma 4.2(i)]{PY21},
	taking $k=1$ there. Note that
	a $0$-separating set in $\cX$ (as it is called in
	\cite{PY21})
	is simply a component of $G(\cX,{r_n})$.
	Note also that if $A= [0,1]^2$ the proof of
	\cite[Lemma 4.2(i)]{PY21} remains applicable, using Lemma \ref{l:A1}
	of the present paper.
%	Note  {\bf There we also consider $A$ polygonal
%	in $d=2$.}
\end{proof}

We shall bound the expected number of `medium-sized' non-singleton components,
$\E[K_{n,\eps ,\rho}]$,
using the following two lemmas (here we use notation such as $\cX^x$ from \eqref{e:addpoint} and $\scr M_{n,\eps,K}(x,\cX)$ from \eqref{e:medevent1}). 

%write $\cX^x$, for $\cX \cup \{x\}$,
%$\cX^{x,y}$ for $\cX \cup\{x,y\}$ 
%and
%$\cX^{x,y,z}$ for $\cX \cup\{x,y,z\}$. 

\begin{lemma}
\label{l:medK}
Let $\eps, \rho \in (0,\infty)$ with $\eps < \rho$. Then
	there exists $\delta_{10} = \delta_{10}
	(d,A,\eps,\rho) >0$ such that for all $n$ large and  all 
	distinct $x,y,z \in A$, we have:
\begin{align}
	\PP[\scr F_n(x,\eta_n^x) \cap
	%\{\eps {r_n} < \diam(\cC_{r_n}(x,\eta_n^x)) \le \rho {r_n}\}
	\scr M_{n,\eps, \rho}(x, \eta_n)
	] \le e^{-n\nu(B_{r_n}(x)) - \delta_{10}  n r_n^d}; 
	\label{e:0825a}
	\\
	\PP[\scr F_n(x,\eta_n^{x,y}) \cap 
	%\{\eps r_n < \diam(\cC_{r_n}(x,\eta_n^{x,y})) \le \rho r_n\}
	\scr M_{n,\eps,\rho}(x,\eta_n^y)
	] \le e^{-n\nu(B_{r_n}(x)) - \delta_{10}  n r_n^d}; 
	\label{e:0825b}
	\\
	\PP[\scr F_n(x,\eta_n^{x,y,z}) \cap 
	%\{\eps r_n < \diam(\cC_{r_n}(x,\eta_n^{x,y,z})) \le \rho {r_n}\}
	\scr M_{n,\eps,\rho}(x,\eta_n^{y,z})
	] \le e^{-n\nu(B_{r_n}(x)) - \delta_{10}  n r_n^d}; 
	\label{e:0825e}
	\\
	\PP[\scr F_n(x,\X_{n-1}^x) \cap 
	%\{\eps r_n < \diam(\cC_{r_n}(x,\X_{n-1}^x)) \le \rho r_n\}
	\scr M_{n,\eps, \rho}(x,\X_{n-1})
	] \le e^{-n\nu(B_{r_n}(x)) - \delta_{10}  n r_n^d}; 
	\label{e:0825c}
	\\
	\PP[\scr F_n(x,\X_{n-2}^{x,y}) \cap
	%\{\eps r_n < \diam(\cC_{r_n}(x,\X_{n-2}^{x,y})) \le \rho r_n\}
	\scr M_{n,\eps,\rho}(x,\X_{n-2}^y)
	] \le e^{-n\nu(B_{r_n}(x)) - \delta_{10}  n r_n^d}. 
	\label{e:0825d}
\end{align}
%	(b) If $A = [0,1]^2$ then there exists $K>0$ such that
%	for all distinct $x,y \in A$ with $\dist x, \partial_0 A 
%	\geq K r$, we have \eqref{e:0825a}, \eqref{e:0825b},
%	\eqref{e:0825c} and \eqref{e:0825d}.
\end{lemma}
\begin{proof}
	 %See \cite[Lemma 4.3]{PY21}, taking the parameter $k$
%	there to be 1, for a proof of \eqref{e:0825a}
%	and \eqref{e:0825c}. The results
%	\eqref{e:0825b}, \eqref{e:0825e} and
%	\eqref{e:0825d} are proved similarly.
%	In the case where $A=[0,1]^2$, \cite[Lemma 4.3]{PY21}
%	applies only when $x$ is not too close to any of the corners
%	of $A$, but the remaining cases can be proved by
%	a similar argument using 
%	\cite[Proposition 5.15]{Pen03}. \\
%
%
	%{\bf NEW.}
	Later in the proof we shall use the fact that since
	we assume $A$ is compact and $f$ is continuous on $A$ with $f_0>0$,
	\bea
	\lim_{s \downarrow 0} \big(
	\sup\{f(y)/f(x): x,y \in A, \|y-x\|\leq s\} \big)
	=1.
	\label{0703c}
	\eea
	
	%(i) Suppose $\partial A \in C^2$.
	We shall first show \eqref{e:0825b}.  Without loss of generality, 
	we can and do assume $\eps <1$.  
	Let $\delta := \delta_2(d,A,\rho,\eps )$ be as in Lemma
	\ref{l:A2}.
	Choose 
	$\delta'  \in (0, 1/(99 \sqrt{d}))$
	such that
	\begin{align}
		\lambda( B_1(o)\setminus B_{1-\sqrt{d}\delta'  }(o)
		) \le \delta.
		%~~~~~~~~~~~~~~
		%|B_1(o)\setminus B_1(2\sqrt{d}\delta_2   e_1)|\le \delta_1/16.
		\label{e:del2}
	\end{align}
	Partition $\RR^d$ into cubes of side length $\delta'   r_n$.
	Given finite $\cY \subset \R^d$, denote by $\A_{\delta'}
	(\cY)$ the closure of the union of all the cubes in
	the partition that intersect $\cY$. Here
	$\A$ stands for ``animal''.
	%and is unrelated to our underlying domain $A$.
	If $x \in \cY$ and
	$\diam \cY \in (\eps  r_n, \rho r_n]$, then
	$\A_{\delta'}(\cY)  \subset
	%B(x, \rho r + \delta_2 d^{1/2} r)$ 
	B_{\rho r_n + \delta' d^{1/2} r_n}(x)$ 
	and $\A_{\delta'}(\cY) $
	can take at most
	$c := 2^{{(2\lceil (\rho /{\delta'})  +\sqrt{d}\rceil
	)^d}}$ different possible shapes. 

	%If the event  $\medevent(\Po_n)$ occurs there is at least one
	%set $\cY \in \cC_{r}(x,\Po_n)$ with $\eps  r < \diam \cY
	%\le \rho r$. If there are several such sets $\cY$, choose one
	%of these according to some deterministic rule, and
	%denote it by $\cY^*(\Po_n)$.
	
	Fix $x$ and $y$.
	For finite $\cX \subset \R^d$, 
	write $\cY^*(\cX)$ for $\cC_{r_n}(x,\cX^{x,y})$.
	Fix a possible shape $\sigma$ that might arise
	as $\A_{\delta'}%(\cY)$ for some $\cY \in
	%(\cC_{r_n}(x,\Po_n))
	(\cY^*(\Po_n))$
	when $\scr F_n(x,\eta_n^{x,y}) \cap 
	%\{\eps r_n < \diam(\cC_{r_n}(x,\eta_n^{x,y})) \le \rho r_n\}
	\scr M_{n,\eps,\rho}(x,\eta_n^y)$ occurs,
	%with $\diam \cY \in (\eps r,\rho r]$,
	and suppose
	 event
	$\scr F_n(x,\eta_n^{x,y}) \cap 
	%\{\eps r_n < \diam(\cC_{r_n}(x,\eta_n^{x,y})) \le \rho r_n\}
	\scr M_{n,\eps,\rho}(x,\eta_n^y)
	\cap \{\A_{\delta'}
	%(\cC_{r_n}(x,\Po_n))
	(\cY^*(\Po_n))
	= \sigma\}$ occurs. 

	Let $\sigma^*:= \{z\in \sigma: x\prec z \} \cup \{x\}$. 
	Set $H:=H(\sigma)= (\sigma^* \oplus 
	B_{(1-\sqrt{d}{\delta'}) r_n}(o)) \setminus \sigma^*$. By the
	triangle inequality, $H\subset \cY^*(\Po_n)
	\oplus B_{r_n}(o)$. We claim that $\eta_n \cap H = \emptyset$. 
	Indeed, if there exists  $ u \in \eta_n\cap H$, then by
	definition of $\cY^*(\Po_n) $ we have 
	$u \in \cY^*(\Po_n) $.
	Hence  $u \in \eta_n\cap H\cap \cY^*(\Po_n)$, implying 
	$u \in \sigma$ and therefore $u \in \sigma \setminus \sigma^*$ 
	(since $u \in H$), but this would contradict the assumption that 
	$\scr F_n(x, \Po_n^{x,y})$ occurs.

	Now we estimate from below the volume of $H \cap A$. 
	By Lemma \ref{l:A2} and our choice of 
	$\delta$ and $\delta'$,
	\begin{align*}
		\lambda ( H\cap A )
		\ge \lambda( B_{{r_n}(1-\sqrt{d}{\delta'})}(x)\cap A) + 
		2 \delta r_n^d. 
	\end{align*}
	By \eqref{e:del2}, $\lambda((B_{r_n}(x) \setminus 
	B_{{r_n}(1- \sqrt{d}\delta')}(x))
	\cap A ) \leq \delta r_n^d$ and hence
	$$
	\lambda ( B_{{r_n}(1-\sqrt{d}{\delta'})}(x)\cap A ) 
	\ge \lambda( B_{r_n}(x)\cap A)- \delta r_n^d.
	$$
	Let $\delta'' \in (0,1/2)$ be such  that
	$ \delta''':= (1-2\delta'')(1 + \delta/(\fmax \theta_d)) - 1 >0 $.
	By the preceding estimates,
	and \eqref{0703c}, provided $n$ is large enough
	we have that
	\begin{align*}
		\nu(H) & \geq (1-\delta'') f(x) \left(
		\lambda(B_{r_n}(x) \cap A) +  \delta r_n^d \right)
		\\
		& \geq (1-2 \delta'') \nu(B_{r_n}(x)) 
		\left( 1 + \frac{  \delta r_n^d}{ \fmax
			\theta_d r_n^d
		} \right) = (1+ \delta''') \nu(B_{r_n}(x)).
	\end{align*}
	By Lemma \ref{l:A3}, 
	for all small enough $ r >0$ and all $y \in A$ we have
	%Choose $\delta_0$ so that
	$\nu(B_r(y)) \geq f_0 (\theta_d/4) r^d$.
	%(it can be
	%seen that this
	%is possible, e.g. using Lemma \ref{l:A1}).
	Let $\delta^* = (\theta_d f_0/4) \delta'''$.
	%with $\delta_0$ given at \eqref{e:ballbds}.
	Then
	\begin{align}
		\label{e:0811-3}
		\nu(H) \ge
		\nu(B_{r_n}(x)) + \delta^*  r_n^d. 
	\end{align}
	%Also, because of the upper bound on diameters,
	%there is a constant $c_1 \in [1,\infty)$ such that
	%$\nu(H) \leq c_1 \nu(B_r(x))$ uniformly over all possible
	%$x$,  all small $r$, and all possible $\sigma$.
%
%	Using these upper and lower bounds on $\nu(H)$, 
Then we can deduce that
	\begin{align*}
		\PP[
	\scr F_n(x,\eta_n^{x,y}) \cap 
	%\{\eps r_n < \diam(\cC_{r_n}(x,\eta_n^{x,y})) \le \rho r_n\}
	\scr M_{n,\eps,\rho}(x,\eta_n^y)
		%	\medevent(\Po_n)
			\cap \{\A_{\delta_2}
		(\cY^*(\Po_n)) = \sigma\}]
		& \leq \PP[\Po_n \cap H = \emptyset ] \\
		& \leq 
		e^{- n \nu(B_{r_n}(x)) - n \delta^* r_n^d}.
	\end{align*}
	This, together with the union bound over the choice of possible shapes $\sigma$, gives us \eqref{e:0825b}, and
	%The proofs of  
	\eqref{e:0825a} and \eqref{e:0825e} are proved similarly.
	
	%We prove the result \eqref{e:0825d} for 
	Now consider the binomial case.  Using
	%the volume estimates
	\eqref{e:0811-3} again,
	%and $\nu(H) \leq c_1 \nu(B_r(x))$ once more.
	for $n$ large, we have
	\begin{align*}
		\PP[
			\scr F_n(x,\X_{n-2}^{x,y}) \cap \scr 
			M_{n,\eps,\rho}(x,\X_{n-2}^y)
		%	\medevent(\cX_{n-1})
			\cap \{ \A_{\delta_2}
		(\cY^*(\cX_{n-2})) 
		= \sigma\}] & \leq 
		\PP[\cX_{n-2} \cap H = \emptyset ]
		\\
		& = %\sum_{j=0}^{k-1} \binom{n-1}{j} \nu(H)^j
		(1-\nu(H))^{n-2}
		\\
		& \leq 2 % c_1^{k-1} \sum_{j=0}^{k-1} (n^j /j!) \nu(B_r(x))^j
		\exp(- n \nu(H)),
	\end{align*}
	and hence \eqref{e:0825d}; \eqref{e:0825c} is proved similarly.
%
%	  (ii) {\bf [Maybe do with (i)??]}
%	 Suppose $d=2$ and $A= [0,1]^2$. Let $0 < \eps  < \rho < \infty$.
%        Choose $K$ such that for all $r$ and all $x \in A $
%	distant more than $Kr$ from any corner of $A$,
%	the ball $B_{(\rho +9)r}(x)$ intersects at most one edge of $A$.
%	%We can choose such a $K$ by a similar argument to the proof of Lemma \ref{l:poly_diff_ball}.
%        Then for $x \in A $ distant more than $Kr$ from any corner of $A$
%	%$\setminus \Cor^{(Kr)}$
%        we can deduce the result
%%	\eqref{e:Zmed} and \eqref{e:Zbimed}
%        in the same manner as in the proof of part (i).
%
%	For $x$ in the corner we can  prove  the result by
%	a similar argument using \cite[Proposition 5.15]{Pen03}. 
\end{proof}

	 \begin{lemma}[Bound on means for moderately large components]
		 \label{l:ERmod}
		 There exists $\rho_1 \in (1,\infty)$ such that
	 $\EE [R_{n,\rho_1,(\log n)^2} ] = 
		 %O(e^{-\theta_d nf_0 r_n^d/2} I_n)$
		 O(e^{- n r_n^d} I_n)$
	 and $\EE [R'_{n,\rho_1,(\log n)^2} ] =
		 %O(e^{-\theta_d nf_0 r_n^d/2} I_n)$.
		 O(e^{- n r_n^d} I_n)$
		 as $n \to \infty$, where 
		 $I_n$ is defined at \eqref{e:def:I_n}.
	 \end{lemma}
	 \begin{proof}
		 Let $\rho > 4$.
		 Given $i \in [n] := \{1,\ldots,n\}$,
		 if $\rho r_n< \diam(\cC_{r_n}(X_i,\X_n)) \leq  (\log n)^2r_n $,
		 then there is at least one component of
		 $G(\X_{n-1},r_n)$ with at least one vertex
		 in $B_{r_n}(X_i)$ and with diameter in the range
		 $((\rho-4)r_n/2, (\log n)^2 r_n]$. Hence by the definition
	at \eqref{e:medevent2},
	 $$
	 \EE [R_{n,\rho,(\log n)^2} ] \leq  n \int_A
		 \Pr[ \scr M^*_{n,(\rho-4)/2,(\log n)^2} (x,\cX_{n-1})
		 ]
	 \nu(dx).
	 $$
Hence by Lemma \ref{l:E9}
		 (which applies since the
		 the condition $n^{2/3}r_n^d \to 0$
		  holds by \eqref{e:RLB}),
		 we can choose $\rho_1$ large enough
		 that for $n$ large
		 \begin{align*}
	 \EE [R_{n,\rho_1,(\log n)^2} ] \leq  n
		 \exp (- (\theta_d  f_0 +2) n r_n^d).
		%	 \label{e:ER}
			 \end{align*}
		Hence by  Lemma \ref{l:Ilower},
		  %if it failed 
		 %then we would have $\liminf_{n \to \infty} n^{1/2} r_n^d >0$
		 %and then by Lemma \ref{l:IUB}
		 %the condition
		 % \eqref{c:lowerbound} would be violated.
%
%		 Using \eqref{e:ER} and Lemma \ref{l:Ilower} 
		 we have
		 the result claimed for
		 $\E[R_{n,\rho_1,(\log n)^2}]$.
		 The result for
		 $\E[R'_{n,\rho_1,(\log n)^2}]$,
		 possibly after taking $\rho_1$ even larger, is proved
		 similarly, using the Mecke formula.
	 \end{proof}

 We shall approximate $R_n$ with
	$S_n + R_{n,0,(\log n)^2}$ and $K_n$ with $
	S_{n} + K_{n,0,(\log n)^2} +1$. 

	\begin{lemma}
		\label{l:RvsR0an}
		Let $K >0$. Then all of
		$ \Pr[ R_n \neq S_n + R_{n,0,(\log n)^2} ],$
		$\Pr[ R'_n \neq S'_n + R'_{n,0,(\log n)^2} ]$,
		$  \Pr[ K_n \neq S_n + K_{n,0,(\log n)^2} +1 ]$ and
		$\Pr[ K'_n \neq S'_n + K'_{n,0,(\log n)^2} +1 ]$
		are  $O(n^{-K}I_n e^{-n r_n^d}) $
		as $n \to \infty$.
	\end{lemma}
	\begin{proof} 
		By \eqref{e:RLB} and
		the assumption $\liminf(I_n) >0$,
		 there exists $\alpha >0 $ such
		that for $n$ large we have $nr_n^d < (\alpha/2) \log n$
		and $I_n > n^{-\alpha/2}$ and hence $I_n e^{-nr_n^d}
		> n^{-\alpha/2} e^{-(\alpha/2) \log n}  = n^{-\alpha}$. 
	%	
	%	Proposition
	%	\ref{p:bc} we have
	%$nr_n^d = O(\log n)$ as $n \to \infty$.
	% Hence %by Lemma \ref{l:Ilower}, 
	%for some $\alpha >0$,
	%	we have $I_n e^{-nr_n^d} > n^{-\alpha}$
	%	for all large $n$.  
		Therefore it suffices to prove
		that for any $K>0$,  the probabilities under
		consideration are
		$O(n^{-K})$ as $n \to \infty$.

		Define event $\tilde{\scr U}_n$ as
		in Lemma \ref{l:uniqueness}, taking $\phi_n= (\log n)^2$.
		Then 
		recalling the definition of $\cL_n(\cX)$ %$\tilde{\cG}$
		just
		before Lemma \ref{l:smallgiant},
		we have the event inclusion
$$
		\{ R_n \neq S_n + R_{n,0,(\log n)^2} \} 
		\cup
		\{ K_n \neq S_n + K_{n,0,(\log n)^2} +1 \} 
		\subset \tilde{\scr U}^c_n \cup \{ \diam ( 
		\cL_n(\cX_n)) \leq (\log n)^2 r_n \}.
		$$
	By Lemma \ref{l:uniqueness}, there is a constant $c$ such that
		$\Pr[\tilde{\scr U}^c_n ] \leq \exp(-c (\log n)^2 nr_n^d)$
		for $n$ large.  Combining this
		with \eqref{0621d} from  Lemma \ref{l:smallgiant}
		(which is applicable 
		%because 
		by \eqref{e:RLB})
		%and
		% our condition $\liminf I_n >0$ implies 
		%$nr_n^d= O(\log n)$ as mentioned already)
		%by Proposition \ref{p:bc})
		gives us the results for $R_n$ and $K_n$,
		and the results
		for $R'_n$ and $K'_n$ are proved similarly.
	\end{proof}

\begin{proposition}[Approximation of $K_n$ by $S_n+1$, $K'_n$ by $S'_n+1$]
\label{p:J}
\label{p:exp_K_bin}
	As $n \to \infty$ we have 
\begin{align}
	\max(\E[|K'_n-S'_n -1|], \E[|K_n-S_n -1|])=
	O( (nr_n^d)^{1-d} I_n).
	\label{e:Jlim1}
	%\\
	%\max(|\E[K'_n]-I_n -1|, |\E[K_n]-\tilde{I}_n -1|)=
	%O( (nr_n^d)^{1-d} I_n).
	%\label{e:Jlim2}
\end{align}
%If also $I_n\to\infty$ 
%	then $\max(|\E[K'_n] -I_n|, |\E[K_n]- \tilde{I}_n|) =o(I_n)$
%	as $n \to \infty$.
\end{proposition}

\begin{proof}
	Take $\delta_9 $ as in Lemma \ref{l:smallK} and $\rho_1$ as in
	Lemma  \ref{l:ERmod}.
	Then $K'_n - S'_n = K'_{n,0,\delta_9 } + K'_{n,\delta_9 ,\rho_1} +
	K'_{n,\rho_1,(\log n)^2}
	+ K'_{n,(\log n)^2,\infty}
	$. Taking expectations and using
	 the Mecke formula, we obtain that
	\begin{align}
		\E[|K'_n - S'_n -1 |]     
		\leq \: & n \int_A \Pr[\scr F_n(x,\eta_n^x) \cap
		\{0 < \diam (\cC_{r_n}(x,\eta_n^x)) \leq \delta_9 r_n\}] \nu(dx)
	\nonumber \\
		& +  \int_A \Pr[\scr F_n(x,\eta_n^x) \cap
		\{\delta_9  r_n < \diam (\cC_{r_n}(x,\eta_n^x)) \leq \rho_1 r_n\}] n \nu(dx)
	\nonumber \\
		& +  \E[K'_{n,\rho_1,(\log n)^2}]
		 +  \E[|K'_{n,(\log n)^2,\infty} -1|].
		\label{e:Jn3}
	\end{align}
By Lemma \ref{l:smallK}, the first term 
	in the right hand side of (\ref{e:Jn3})
	is $O((nr_n^d)^{1-d} I_n)$. 
	By Lemma \ref{l:medK} there exists $\delta >0$ such that
	the second term
	in the right hand side of (\ref{e:Jn3})
	is at most $e^{-\delta n r_n^d} I_n$ for all large enough $n$.
	By Lemma \ref{l:ERmod},
	the third term in the right hand side is $O(e^{-nr_n^d}I_n)$.
	For the fourth term,
	recalling $\#(\eta_n) = Z_n$ is Poisson with mean $n$,
	note that $|K_{n, (\log n)^2, \infty}-1| \leq (Z_n +1)
	{\bf 1}\{ K'_{n, (\log n)^2,\infty} \neq 1\}$. Using
	 the Cauchy-Schwarz inequality, and then 
	Lemma \ref{l:RvsR0an} taking $K=2$,
	and the assumption $\liminf(I_n) >0$,
	we deduce that
	\begin{align*}
		\E[|K'_{n,(\log n)^2,\infty} - 1|]
		& 	\leq (\E[(Z_n+1)^2 ])^{1/2} (\Pr[ K'_{n,(\log n)^2,\infty} 
		\neq 1])^{1/2}
		\\
		& = O(n \times n^{-1} I_n^{1/2} e^{-nr_n^d/2})
		\\
		& = O(e^{-nr_n^d/2} I_n). 
	%	= o(I_n).
		%\leq 2n \exp (-c (\log n)^2 n r_n^d). 
	\end{align*}

Combining these estimates shows that
$\E[|K'_n - S'_n -1 |] = O((nr_n^d)^{1-d}I_n)  $.
The proof that $\E[|K_n - S_n -1 |] = O((nr_n^d)^{1-d}I_n)  $
	is similar; in that case the we have an analogous bound to 
	\eqref{e:Jn3} but
	with $\cX_{n-1}^x$ instead of $\eta_n^x$. Thus we have \eqref{e:Jlim1}.
\end{proof}

\begin{lemma}
  \label{l:V0eps}
	Suppose $\delta_1, \delta_9$ are as in Lemma \ref{l:A1}, Lemma 
	\ref{l:smallK} respectively, and $0 < \rho < \min(\frac12,\delta_9,(\delta_1 f_0
	/(\fmax \theta_d))^{1/(d-1)} ).$ Then
	  as  $n \to \infty,$ we have
	\begin{align}
	  \max( \E[R_{n,0,\rho}], \E[R'_{n,0,\rho}]) 
			  = O((nr_n^d)^{1-d}I_n).
			  \label{e:V0eps}
	\end{align}
 \end{lemma}
  \begin{proof}
  %Let $\rho  > 0$.  
	  If the interpoint distances of $\cX_n$ are all distinct
and non-zero (an event of probability 1), then $R_{n,0,\rho } = K_{n,0,\rho} +
		  N_{n,1} + N_{n,2}$, where we set
	  \begin{align*}
		  N_{n,1} :=
		  \sum_{(i,j) \in [n] \times [n]: i \neq j}
		  {\bf 1} ( \scr F_n(X_i,\cX_n)
		  \cap
		  \{0< \diam(\cC_{r_n}(X_i,\X_n)) \leq \rho  r_n  \}
		  \cap 
		  \{X_j \in \cC_{r_n}(X_i,\X_n)\}
		  \\
		  \cap \{\|X_j - X_i\| = \max_{x \in \cC_{r_n}(X_i,\X_n)}\|x-X_i\|
		  \}
		  ) ,
	  \end{align*}
	  and
	  \begin{align*}
		  N_{n,2} :=
		  \sum_{(i,j,k) \in [n] \times [n] \times [n]: i \neq j \neq k \neq i}
		  {\bf 1} ( {\scr F}_{n}(X_i,\cX_n)
		  \cap
		  \{0< \diam(\cC_{r_n}(X_i,\X_n)) \leq \rho  {r_n}  \}
		  \\
		  \cap 
		  \{\{X_j,X_k\} \subset \cC_{r_n}(X_i,\X_n)\}
		  %\\
		  \cap \{\|X_j - X_i\| = \max_{x \in \cC_{r_n}(X_i,\X_n)}\|x-X_i\|
		  \}
		  ) .
	  \end{align*}
%Let $\delta_9$ be as in
%If $0 < \rho \leq \delta_9$, then 
	  Using
	  Lemma \ref{l:smallK},
	 % and the Mecke formula as in \eqref{e:Jn3}, 
	  we have
	  that
	 \begin{align}
		 \E[K_{n,0,\rho}] & 
		 = n \int_A \Pr[
			 \scr F_n(x,\X_{n-1}^x) \cap
		 \{ 0 < \diam (\cC_{r_n}(x,\eta_n^x)) \leq \rho r_n\}] \nu(dx) ]
	 \nonumber \\
		 &	 = O(  (nr_n^d)^{1-d} I_n).
		 \label{e:EN11}
	 \end{align}
	  %Let $\delta_1$ be as in
 Recall the notation $A_x:=\{y \in A: x \prec y\}$.
	  By assumption $\fmax 
	  \theta_d
	  \rho^{d-1} 
	  < \delta_1 f_0$,
	  so by  Lemma \ref{l:A1}, for $n$ large and
	  $x \in A, y \in B(x,\rho r_n) \cap A_x $
	  we have $\nu(B_{r_n}(y) \setminus B_{r_n}(x)) \geq 2 \delta_1
	   f_0 r_n^{d-1}\|y-x\|$ and $\nu(B_{\|y-x\|}(x))  
	  \leq \fmax \theta_d \|y-x\|^d \leq \delta_1 f_0 r_n^{d-1}  \|y-x\|.$ 

	  For $(i,j) \in [n] \times [n]$ with $i \neq j$,
	  for $(i,j)$ to contribute to the sum in the definition of
	  $N_{n,1}$ we need $\|X_j -X_i\| \leq \rho r_n$ because of
	   the condition 
	   $X_j \in \cC_{r_n}( X_i, \cX_n)$ and  the diameter condition,
	  and we also need $X_j \in A_{X_i}$ because
	  of the condition ${\scr F}_n(X_i, \cX_n)$.
	  Also, for all $k \in [n] \setminus \{i,j\}$
	  we need $X_k \notin (B_{r_n}(X_i) \cup B_{r_n}(X_j))
	  \setminus B_{\|X_j - X_i\|}(X_i)$ because of the condition
	  that $\|X_j - X_i\| = \max_{x \in \cC_{r_n}(X_i, \cX_n)} \|x -X_i\|$.
	  Hence, using the estimates in the previous paragraph we have
	 \begin{align*}
		 \E[N_{n,1}] &
		 \leq  n^2 \int_A \int_{B(x,\rho r_n) \cap A_x}
		 (1- \nu[(B_{r_n}(x) \cup B_{r_n}(y)) \setminus
		 B_{\|y-x\|}(x)])^{n-2} \nu(dy) \nu(dx)
		 \\
		 &
		 \leq  2 n^2 \int_A \int_{B(x,\rho r_n) \cap A_x}
		 e^{- n(\nu(B_{r_n}(x)) + \nu (B_{r_n}(y) \setminus
		 B_{r_n}(x)) - \nu(
		 B_{\|y-x\|}(x)))} \nu(dy) \nu(dx)
		 \\
		 & \leq 2 n 
		 \int_A \left( \int_{B(x,\rho r_n) \cap A_x} n
		 e^{ - n \nu(B_{r_n}(x)) - n \delta_1 f_0 r_n^{d-1} \|y-x\|}
		 \nu(dy) \right) \nu(dx).
	 \end{align*}
	 In the last expression the inner integral can be bounded by
	 \begin{align*}
		 n e^{-n \nu(B_{r_n}(x))} \int_{B(o,\rho r_n)}
		 e^{-\delta_1 f_0 nr_n^{d-1}\|u\|}
		 \fmax du
		 = (nr_n^{d})^{1-d}
		 e^{-n \nu(B_{r_n}(x))} \int_{B(o,\rho  nr_n^d)} 
		 e^{-\delta_1 f_0 \|v\|}
		 \fmax dv
	 \end{align*}
	 and therefore 
	 \begin{align}
		 \E[N_{n,1}] = O((nr_n^d)^{1-d} I_n).
		 \label{e:EN12}
	 \end{align}
	 Next, using Lemma \ref{l:A1} and the inequality
	 $\fmax \theta_d \rho^{d-1} < \delta_1 f_0$ again we have that
	 \begin{align*}
		 \E[N_{n,2}] \leq n^3
		 \int_A \int_{B(x,\rho r_n) \cap A_x}
		 \int_{B(x,\|y-x\|)}
		 %\\ \times 
		 (1- \nu[(B_{r_n}(x) \cup B_{r_n}(y)) \setminus B_{\|y-x\|}(x)])^{n-3}
		 \\
		 \nu(dz) \nu(dy) \nu(dx)
		 \\
		 \leq 2 n \theta_d  \fmax
		 \int_A \left(\int_{B(x,\rho r_n) \cap A_x} n^2 
		 \|y-x\|^d
		 \exp( - n \nu(B_{r_n}(x)) - n \delta_1 f_0 r_n^{d-1} \|y-x\|)
		 \nu(dy) \right) \\
		 \nu(dx).
	 \end{align*}
	 In the last expression the inner integral can be bounded by
	 \begin{align*}
		 n^2 e^{-n \nu(B_{r_n}(x))} 
		 \int_{B(o,\rho r_n)} \|u\|^d e^{-\delta_1 f_0 nr_n^{d-1}\|u\|}
		 \fmax du
		 \\
		 = (nr_n^{d})^{2-2d}
		 e^{-n \nu(B_{r_n}(x))} \int_{B(o,\rho  nr_n^d)} 
		 \|v\|^d
		 e^{-\delta_1 f_0 \|v\|}
		 \fmax dv,
	 \end{align*}
	 and therefore 
	 \begin{align*}
		 \E[N_{n,2}] = O((nr_n^d)^{2-2d} I_n).
	 \end{align*}
	 Combined with (\ref{e:EN11}) and \eqref{e:EN12}
	 this shows that $\E[R_{n,0,\rho }] =
	 O((nr_n^d)^{1-d}I_n)$,
	 which is the statement about 
	  $\E[R_{n,0,\rho}] $ in \eqref{e:V0eps}.
	  The corresponding statement for 
	  $\E[R'_{n,0,\rho}] $ is proved similarly
	  using the multivariate Mecke formula.
 \end{proof}

For $0 < \eps < \rho < \infty$,
recall the definition of event $\scr M_{n,\eps,\rho}(x,\cX)$
at \eqref{e:medevent1}.
	To deal with the medium-sized components,
we shall use the following
estimate for the integral of $\Pr[\scr M_{n,\eps,\rho}(x,\xi_n)]$
with $\xi_n = \eta_n$ or $\xi_n = \X_{n-1}$.
We use notation $\cX^x$ from \eqref{e:addpoint}.

\begin{lemma}[Estimate on medium clusters]
\label{l:E6}
	Let $\rho,\eps \in (0,\infty)$ with  $\rho >\eps $.
	Then there exists
	$\delta =\delta (\eps ) \in
	(0,\infty)$ such that as $n \to \infty$,
	we have
\begin{align}
	n \int_A
	\PP[\scr M_{n,\eps,\rho}(x,\eta_n)] \nu(dx) 
	=O( e^{ -  \delta  n r_n^d} I_n) ;
	\label{e:E6} 
	\\
	n \int_A \PP[\scr M_{n,\eps,\rho}(x,\X_{n-1})] \nu(dx)
	= O( e^{- \delta  n r_n^d} I_n).
	%\le  c\exp( - n \nu(B_r(x)) -n \de  r^d).
	\label{e:E'6} 
\end{align}
\end{lemma}
\begin{proof}
	If $\scr M_{n,\eps,\rho}(x,\eta_n) \setminus \scr F_n(x,\eta_n^x)$ occurs
	then for at least one
	$y \in \eta_n \cap B_{\rho r_n}(x)$ 
	we have that $\diam(\cC_{r_n}(y,\eta_n^{x,y})) \in (\eps r_n,\rho r_n]$,
	and moreover $\scr F_n(y,\eta_n^{x,y})$ occurs and $x \in
	\cC_{r_n}(y,\eta_n^{x,y})$.
	By Markov's inequality and the Mecke formula,
	\begin{align*}
		n \int_A \Pr[\scr M_{n,\eps,\rho}(x,\eta_n) \setminus \scr F_n(x,\eta_n^x)]
		\nu(dx)
		\leq n^2 \int_A  \int_{A \cap B_{\rho r_n}(x)} 
		\Pr[\scr M_{n,\eps,\rho}(y,\eta_n^x ) \cap
		\scr F_n(y,\eta_n^{x,y}) \nonumber \\
		\cap \{x \in \cC_{r_n}(y,\eta_n^{x,y})\}]
		\nu(dy) \nu(dx).
		%\nonumber \\
		%\leq  n \fmax \theta_d  r^d \sup_{y \in B_r(x)} \Pr[\scr M_{n,\eps,\rho}(y,\eta_n)]. 
	%	\label{0803h}
	\end{align*}
	By \eqref{e:0825b} from
	Lemma \ref{l:medK},
	there exists $\delta >0$
	such that for $n$ large the
	probability inside the integral on the right of
	the last display is bounded above
	by $\exp(-n \nu(B_{r_n}(y)) -  2 \delta n r_n^d)$. 
	Then using Fubini's theorem we obtain
	that for $n$ large
	\begin{align}
		n \int_A \Pr[\scr M_{n,\eps,\rho}(x,\eta_n) \setminus \scr F_n(x,\eta_n^x)]
		\nu(dx)
		& \leq n^2 \int_A  \nu(B_{\rho r_n}(y))  
		\exp(- n \nu(B_{r_n}(y)) - 2 \delta n r_n^d)
		\nu(dy)
		\nonumber \\
		& = O(nr_n^d I_n e^{-2 \delta nr_n^d}) 
		= O( e^{- \delta nr_n^d} I_n).
		\label{e:0525a}
	\end{align}
	Also using Lemma \ref{l:medK} we obtain that
	\begin{align*}
		n \int_A \Pr[\scr M_{n,\eps,\rho}(x,\eta_n) \cap \scr F_n(x,\eta_n^x)]
		\nu(dx)
		& \leq n \int_A  
		\exp(- n \nu(B_{r_n}(x)) -  \delta n r_n^d)
		%\Pr[\scr M_{n,\eps,\rho}(y,\eta_n ) \cap \scr F_n]
		\nu(dx)
		\nonumber \\
		& = e^{- \delta nr_n^d} I_n, 
		%\label{e:0525c}
	\end{align*}
	and combined with \eqref{e:0525a} this yields \eqref{e:E6}.

	The proof of \eqref{e:E'6} is similar.
\end{proof}

We are now ready to estimate the asymptotic expected values of $R_n$
	and $R'_n$.

\begin{proposition}[Approximation of $R_n$, $R'_n$ by $S_n, S'_n$]
%	\label{p:EgiantP}
 \label{p:exp_R_bin}
	As $n \to \infty$ we have that
\begin{align}
	 \EE[|R_n - S_n|] = O( (nr_n^d)^{1-d} I_n);   
	 \label{e:RvS}
	 %\\
	%|\EE[R_n] - I_n| = O((nr_n^d)^{1-d} I_n); \label{e:0903c}
	\\
	\EE[|R'_n - S'_n|]
	= O((nr_n^d)^{1-d} I_n).
	\label{e:0903d}
\end{align}
\end{proposition} 
 \begin{proof}
	 Note that $|R_n- S_n- R_{n,0, (\log n)^2}| \leq n$.
	 Hence  by Lemma \ref{l:RvsR0an},
	 \begin{align}
		 \E[|R_n - S_n - R_{n,0,(\log n)^2}|]
		 \leq n \Pr[R_n \neq S_n + R_{n,0,(\log n)^2}]
		 = O(e^{-nr_n^d} I_n).
	 \label{0621b}
	 \end{align}

	 Using Lemma  \ref{l:V0eps},  choose $\eps  \in (0,1)$
	 such that $\E[R_{n,0,\eps}]
	 = O((nr_n^d)^{1-d}I_n)$.
	 Using Lemma \ref{l:ERmod},  choose $\rho \in (1,\infty)$
	 such that $\E[R_{n,\rho, (\log n)^2}] =
	 %I_n \exp(-\Omega(nr_n^d))$.
	 O(e^{-nr_n^d} I_n)$.
	   By Lemma \ref{l:E6}, there exists $\delta >0$ such that 
	  %as follows:
	  \begin{align*}
		  \E[R_{n,\eps,\rho} ] = n \int_A
		  \Pr[ \eps r < \diam \cC_{r_n}(x, \cX_{n-1}^x 
		  %\cup \{x\}
		  ) \leq \rho r_n]
		  \nu(dx)
		  %\\
		  =O( e^{-n \delta r_n^d} I_n).
	  \end{align*}
	 Combining these estimates shows that  
	 $\E[R_{n,0,(\log n)^2} ] = O((nr_n^d)^{1-d}I_n)$.
	 Then using (\ref{0621b}) yields (\ref{e:RvS}).

	The proof of  \eqref{e:0903d}
	is  similar; the only difference is that in the
	step of the argument corresponding to \eqref{0621b} we
	 use the inequality $|R'_n - S'_n - R'_{n,0,(\log n)^2}|
	 \leq Z_n {\bf 1}\{R'_n \neq S'_n + R'_{n,0,(\log n)^2}\}$
	 and the Cauchy-Schwarz inequality.
 \end{proof}

\subsection{Proof of first order limit theorems}
\label{ss:norates}

\begin{proof}[Proof of Theorem \ref{t:momLLN}]
	By Proposition \ref{p:J}
	we have 
	$$
	\E[K'_n-1] = \E[S'_n] +
	\E[K'_n -1 - S'_n] = I_n(1+O(nr_n^d)^{1-d}).
	$$
	By Proposition \ref{p:J}
	and Lemma \ref{l:diff_I_Itilde}
	we also have $ \E[K_n-1] = I_n(1+ O(nr_n^d)^{1-d})$.
	By Proposition \ref{p:exp_R_bin}
	we have $\E[R'_n] =  I_n(1+ O(nr_n^d)^{1-d})$.
	By Proposition \ref{p:exp_R_bin}
	and Lemma \ref{l:diff_I_Itilde}
	we have $\E[R_n] = \E[S_n ] +
	O((nr_n^d)^{1-d} I_n) = I_n(1+ O(nr_n^d)^{1-d})$.
	Thus we have \eqref{e:meanasymp}.

	For  \eqref{e:basicLLN},
	suppose for now that $\zeta_n$ is $K_n -1$ or $R_n$.  Recalling
	that $\tilde{I}_n := \E[S_n]$, we note that
	\begin{align}
	\E[ |(\zeta_n/I_n) -1|] \leq 
		\E[
		I_n^{-1}
			| \zeta_n - S_n|]  
		%	\E I_n^{-1} [|S_n -S'_n|] 
			+ \E[
		I_n^{-1}
				|S_n-\tilde{I}_n|] +
				I_n^{-1} |\tilde{I}_n-I_n|.
		\label{e:forL1}
		\end{align}
	By Proposition \ref{p:J} (when $\zeta_n = K_n -1$)
	or Proposition \ref{p:exp_R_bin} (when $\zeta_n = R_n$)
	we have $\E[|\zeta_n - S_n|] = O((nr_n^d)^{1-d}I_n)$,
	so the first term in the right hand side of \eqref{e:forL1}
	is $O((nr_n^d)^{1-d})$.
	 Moreover by
	 the Cauchy-Schwarz inequality the 
	second term in the right hand side of \eqref{e:forL1}
	is bounded by $(I_n^{-2}\Var(S_n))^{1/2}$, and by
	 Proposition \ref{p:var_iso_bin} this is $O(I_n^{-1/2})$.
	 The third term in the right hand side of \eqref{e:forL1}
	 is $O((nr_n^d)^{1-d})$
	 by Lemma \ref{l:diff_I_Itilde}.
	 Thus we have 
	 \eqref{e:basicLLN} when $\zeta_n $ is $ K_n-1$ or $R_n$,
	 and the corresponding result when $\zeta_n $ is $K'_n-1$
	 or $R'_n$
	 can be proved similarly.  Finally if we add $1$ to $\zeta_n$
	 then we should add a term of $1/I_n$ on the right hand side
	 of \eqref{e:forL1}, but this term is $o(I_n^{-1/2})$
	 so we have \eqref{e:basicLLN} when $\zeta_n$ is
	 $K_n, R_n+1, K'_n$ or $R'_n+1$ too.
\end{proof}

\begin{proof}[Proof of Theorem \ref{t:exprate}]
	 %Finally we need to prove \eqref{e:loginP1} and \eqref{e:loginP2}.
	% = (\max(f_0,d(f_0-f_1/2)))^{-1}$.
	 Since we assume $I_n \to \infty$, 
	 by Proposition \ref{p:bc} we have
	 $b^+ \leq b_c$.
	 Suppose also $b^+ \leq b'_c $.
	 Let $\delta >0$.  Then by Lemma \ref{l:Ilower},
	 \begin{align}
		 ne^{-n \theta_d f_0 r_n^d (1+ \delta) } = o(I_n).
		 \label{e:0927a}
	 \end{align}
	 For an upper bound on $I_n$,
	 we shall use Lemma \ref{l:IUB}.
	Let $\eps >0$ with $f_1 \eps < f_0 \delta$.
	Since $b^+ \leq b'_c$ we have  $b^+   (f_0 - f_1/2) \leq 1/d $, 
	and hence
	\begin{align*}
		\frac{n^{1-1/d} e^{-n \theta_d r_n^d  f_1 (\frac12 - \eps)}}{
			n e^{-n \theta_d f_0r_n^d(1- 2  \delta)}}
		& = O (n^{-1/d} e^{n  \theta_d r_n^d( f_0 - (f_1 /2) - f_0 
		\delta)}  )
		\\
		& = O( n^{-1/d} e^{(b^+(f_0 - f_1/2)- \frac12 f_0  
		\delta)\log n} ) = o(1).
	\end{align*}
	Therefore both terms in the right hand side of \eqref{e:InUB}
	are $o( n e^{-n \theta_d f_0r_n^d (1- 2 \delta)})$, 
	so by Lemma \ref{l:IUB},
	 $I_n = o(n e^{-n\theta_d r_n^d f_0(1- 2 \delta)})$.
	 Using this, along with \eqref{e:0927a} and the fact that
	 the $L^1$ convergence in \eqref{e:basicLLN} implies
	 convergence in probability also, we obtain  that 
	 with probability tending to one,
	 $n e^{-n\theta_d r_n^d f_0(1+ 2 \delta)} < \zeta_n < n
	 e^{-n\theta_d r_n^d f_0(1- 2 \delta)}$, which
	 gives us \eqref{e:loginP1}.

	 Now suppose $b^- \geq b'_c$.
	 Since also %\leq b^- \leq b_c $. Suppose also that
	 $b^+ \leq b_c$, we have
	 $b'_c \leq b_c < \infty$. Hence
	 $b'_c = (d(f_0-f_1/2))^{-1}$ so
	 $b^- \geq (d(f_0 - f_1/2))^{-1}$ and
	 $b^-((f_1/2) - f_0) \leq -1/d$. Let $\eps >0$. Then
	 \begin{align*}
		 \frac{n e^{-n \theta_d f_0 r_n^d}}{n^{1-1/d}
		 e^{-n r_n^d \theta_d f_1(\frac12 - \eps)}}
		 = n^{1/d} e^{nr_n^d \theta_d ((f_1/2) - f_0 - f_1 \eps)} 
		 \leq n^{1/d}
		 e^{b^-((f_1/2)- f_0 - f_1 \eps/2) \log n} = o(1), 
	 \end{align*}
	 so by Lemma \ref{l:IUB},
	 $I_n = o(n^{1-1/d} e^{-n \theta_d r_n^d f_1 (\frac12 - 2 \eps)})$.
	 Also by Lemma \ref{l:Ilower},
	 $n^{1-1/d} e^{-n r_n^d \theta_d f_1(\frac12 +\eps) } = o(I_n)$.
	 Hence by the convergence in probability of $\zeta_n/I_n$
	 to 1 which follows from \eqref{e:basicLLN},
	 with probability tending to 1 we have
	 $n^{1-1/d} e^{-n r_n^d \theta_d f_1(\frac12 +\eps) } \leq
	 \zeta_n \leq
	 n^{1-1/d} e^{-n \theta_d r_n^d f_1 (\frac12 - 2 \eps)}$,
	 and \eqref{e:loginP2} follows.

	 Now suppose $b^+ = b^- =b$ for some $b \geq 0$.
	 Then if $f_0 b \leq (1/d) + f_1 b/2$ we have
	 $(f_0 -f_1/2) b \leq 1/d$ so $b \leq b'_c$ 
	 and \eqref{e:loginP1} applies. By \eqref{e:loginP1}
	 we have $\zeta_n = n^{1- bf_0 + o_{\Pr}(1)}$. 
	 Conversely if $f_0 b \geq (1/d) + f_1 b/2$
	 we have $(f_0-f_1/2) b \geq 1/d$ and $b \geq b'_c$ so
	 \eqref{e:loginP2} applies and tells us that
	 $\zeta_n = n^{1- (1/d) - f_1 b/2 + o_{\Pr}(1)}$. 
\end{proof}

\begin{proof}[Proof of Theorem \ref{t:PoLim}]
	Here we assume as $n \to \infty$ that 
	$I_n = \Theta(1)$ (which implies $n r_n^d = \Theta(\log n) $ 
	by Proposition \ref{p:bc} and Lemma \ref{l:Ilower}).
	Then by Lemma \ref{l:dTV} %, for some $\delta >0$
	we have $\dtv(S'_n , Z_{I_n}) = O(e^{-\delta_1 f_0
	n r_n^d})$.

	By Proposition \ref{p:J} (when $\xi_n = K'_n -1$)
	or Proposition \ref{p:exp_R_bin} (when $\xi_n = R'_n$)
	and  Markov's inequality,  for both those cases
	$\dtv(\xi_n, S'_n) \leq \Pr[ \xi_n  \neq  S'_n]  \leq
	\E[|\xi_n  - S'_n|] = O((nr_n^d)^{1-d})$,
	and therefore by Lemma \ref{l:dTV} and the triangle inequality,
	$\dtv(\xi_n  , Z_{I_n}) = O((nr_n^d)^{1-d})= O((\log n)^{1-d})$
	in those cases. 

	Now suppose  $\xi_n$ is $K_n -1 $ or $ R_n$.  
	%We can argue similarly except for when $d=2$
	%and $A = [0,1]^2$, since  \eqref{e:dtvSZ} of
	 %Lemma \ref{l:dTV} does not apply in that case.
	 %Therefore we use a different method.
	By Proposition \ref{p:J} (when $\xi_n = K_n -1$)
	or Proposition \ref{p:exp_R_bin} (when $\xi_n = R_n$)
	and  Markov's inequality,  for both those cases
	$\Pr[ \xi_n  \neq  S_n]  \leq
	\E[|\xi_n  - S_n|] = O((nr_n^d)^{1-d})$,
	and therefore it suffices to prove that $\dtv(S_n , Z_{I_n})
	= O((nr_n^d)^{1-d})$.
	By Lemma \ref{l:dTV} we have % for some $\delta >0$ that
	$\dtv(S'_n , Z_{I_n}) = O(e^{-\delta_1 f_0 nr_n^d})$,
	so it suffices to prove that $\E[|S'_n-S_n|] = O((nr_n^d)^{1-d})$.

	 Recall that $\eta_n = \{X_1,\ldots,X_{Z_n}\}$.
	 Let $m=m(n) = \lfloor n^{3/4} \rfloor$.
	 By the Cauchy-Schwarz inequality
	 and the Chernoff bound from Lemma \ref{lemChern}(ii),
	 \begin{align}
		 \E[|S'_n-S_n| {\bf 1}\{|Z_n -n | > m \}
		 & \leq (\E[\max(Z_n,n)^2])^{1/2} (\Pr[|Z_n-n| > m])^{1/2}
		 \nonumber \\
		 & \leq (2n^2+n)^{1/2} \exp(-\Omega(n^{1/2})). 
		% = o((\log n)^{1-d}).
		 \label{e:0924a}
	 \end{align}
	 For $i=1,2,\ldots$ write
	  $Y_i := X_{Z_n+i}$ 
	 and $Y'_i := X_{n+i}$. 
	 Then $Y_1,Y_2,\ldots$ are $\nu$-distributed
	 random vectors, independent of each other and of $\eta_n$.
	 Observe that
	 \begin{align*}
		 |S'_n - S_n| {\bf 1}\{ Z_n \leq n \leq Z_n + m\}
		 \leq \sum_{i=1}^m
		 \big( & {\bf 1}\{ \eta_n \cap B_{r_n}(Y_i) = \emptyset\}
		 \\
		  & + \sum_{x \in \eta_n \cap B_{r_n}(Y_i)} {\bf 1}
		 \{\eta_n \cap B_{r_n}(x) = \{x\}\} \big). 
	 \end{align*}
	 Therefore using the Mecke formula followed by
	 Fubini's theorem we obtain that
	 \begin{align}
		 \E[|S'_n - S_n| {\bf 1}\{ Z_n \leq n \leq Z_n + m\}]
		 \leq
		 m \int_A e^{-n \nu(B_{r_n}(x))}dx
		 \nonumber
		 \\
		 + mn \int_A \int_{B_{r_n}(x)} e^{-n \nu(B_{r_n}(y))} \nu(dy) 
		 \nu(dx)
		 \nonumber
		 \\
		 \leq n^{-1/4}  I_n + (n \fmax \theta_d r_n^d ) n^{-1/4} I_n
		 = O(n^{-1/4}(\log n)). 
		 \label{e:0924b}
	 \end{align}
	 Also  $Y'_1,Y'_2,\ldots$ are $\nu$-distributed
	 random vectors, independent of each other and of $\X_n$.
	 Then since $(1- \nu(B_{r_n}(x)))^{n-1} \leq 2 e^{-n \nu(B_{r_n}(x))}$
	 for all large enough $n$ and all $x \in A$,
	 \begin{align*}
		 \E[|S'_n - S_n| {\bf 1}\{ n \leq Z_n \leq n + m\} ]
		 \leq \E \Big[ \sum_{i=1}^m
		 \big(  {\bf 1}\{ \X_n \cap B_{r_n}(Y'_i) = \emptyset\}
		 \\
		   + \sum_{x \in \X_n \cap B_{r_n}(Y'_i)} {\bf 1}
		 \{\cX_n \cap B_{r_n}(x) = \{x\}\} \big) 
		 \Big]
		 \\
		 \leq m  \int_A (1- \nu(B_{r_n}(x)))^n \nu(dx)
		 + m n \int_A \int_{B_{r_n}(x)} (1- \nu(B_{r_n}(y)))^{n-1} \nu(dy)
		 \nu(dx)
		 \\
		 = O( (nr_n^d) n^{-1/4} I_n) = O(n^{-1/4} \log n). 
	 \end{align*}
Combined with \eqref{e:0924a} and \eqref{e:0924b} this shows
that $\E[|S'_n-S_n|] = O(n^{-1/4} \log n) = O((nr_n^d)^{1-d})$ as required.
\end{proof}

\begin{proof}[Proof of Theorem \ref{t:PoLimUn}]
	Assume the uniform case applies.
	We first show that for any $\gamma \in \R$ we have:
	%and
	%assume for now that $\gamma_n \to \gamma \in \R$ as $n \to \infty$.
	%If $d=2$ then $\mu_n = e^{-\gamma_n} \to e^{-\gamma}$.
	\begin{align}
	{\rm if}~ \lim_{n \to \infty} \gamma_n = \gamma
	~{\rm then}~ 
	\lim_{n \to \infty} \mu_n = \begin{cases}
		e^{-\gamma} ~{\rm if}~ d=2\\
		c_{d,A} e^{-\gamma/2} ~{\rm if}~ d\geq 3.
	\end{cases}
		\label{e:mulim}
	\end{align}
	The case $d=2$ of \eqref{e:mulim} is obvious because
	$\mu_n = e^{-\gamma_n}$ in this case.
	Suppose $d \geq 3$. If $\lim_{n \to \infty} \gamma_n = \gamma$,
	then as $n \to \infty$ the second term
	in the right hand side of \eqref{e:defI'} satisfies
	\begin{align*}
		\theta_{d-1}^{-1} |\partial A|
		r_n^{1-d}
		e^{-n \theta_d f_0 r_n^d/2}
		& \sim 
		\theta_{d-1}^{-1} |\partial A|
		\big( \frac{(2-2/d)\log n}{n \theta_d f_0} \big)^{-1+1/d}
		e^{-\gamma/2} \Big( \frac{n}{\log n} \Big)^{-1+ 1/d}
		%\nonumber
		\\
		& = \theta_{d-1}^{-1} \sigma_A (\theta_d/(2-2/d))^{1-1/d}
		e^{-\gamma/2} = c_{d,A}e^{-\gamma/2},
	\end{align*}
	and moreover the ratio between the two terms 
	in the right hand side of \eqref{e:defI'} satisfies
	\begin{align*}
		\frac{ne^{-n\theta_d f_0 r_n^d}}{\theta_{d-1}^{-1}
		|\partial A| r_n^{1-d} e^{-n \theta_d f_0 r_n^d/2}}
		& = \theta_{d-1}|\partial A|^{-1}  nr_n^{d-1}
		e^{-n\theta_d f_0 r_n^d/2}
		\\
		& = O \Big( (\log n) %\big( \frac{\log n}{n} \big)^{1/d}
		\big( \frac{n}{\log n} \big)^{-1+2/d} \Big) = o(1),
	\end{align*}
	and \eqref{e:mulim} follows.

	Now suppose $|\gamma_n |= O(1)$,
	which implies $nr_n^d = \Theta(\log n)$ as $n \to \infty$.
	By \eqref{e:mulim} and a subsequence argument we have
	that $\mu_n = \Theta(1)$ as $n \to \infty$.
	Let $\xi_n$ be any of $K_n-1, R_n,K'_n -1$ or $R'_n$. 
	By a simple coupling argument for $0 < s<t$
	we have $\dtv(Z_{s},Z_t) \leq t-s$. Hence by the
	triangle inequality 
	$$
	\dtv(\xi_n,Z_{\mu_n}) \leq \dtv(\xi_n,Z_{I_n}) + |I_n-\mu_n|.
	$$
	If $d=2$ then 
	by Proposition \ref{p:average2d} 
	and \eqref{e:defI'}
	$I_n = \mu_n(1+ O(nr_n^2)^{-1/2})$; 
	hence by Theorem \ref{t:PoLim},  
	$$
	\dtv(\xi_n,Z_{\mu_n}) = O( (\log n)^{-1} ) + O((nr_n^2)^{-1/2})
	= O((\log n)^{-1/2}).
	$$
	If  $d=3$ then
	by Proposition \ref{p:average3d+}
	and \eqref{e:defI'}, $I_n = \mu_n \Big(1+ O\big( \frac{\log (nr_n^d)}{
		 nr_n^d} \big)^2 \Big) $;
	hence by Theorem \ref{t:PoLim},  
	$$
	\dtv(\xi_n,Z_{\mu_n}) = O( (\log n)^{1-d}) + O\Big(
	\big(\frac{\log(nr_n^d)}{nr_n^d} \big)^{2} \Big)
	= O\Big( \big(\frac{\log \log n}{\log n} \big)^{2}\Big).
	$$
	Thus we have part (a).
	In particular, for all $d \geq 2 $ we have
	$\dtv(\xi_n,Z_{\mu_n}) \to 0$ so that if
	$\gamma_n \to \gamma$ then by \eqref{e:mulim}
	we have $\xi_n \toD Z_{e^{-\gamma}}$ if $d=2$
	and  $\xi_n \toD Z_{c_{d,A}e^{-\gamma/2}}$ if $d \geq 3$,
	which is part (b).

	For part (c),
	now assume \eqref{e:supcri} and \eqref{e:supcriupper}.
	First suppose $d=2$. By Proposition \ref{p:average2d},
	as $n \to \infty$
	we have $I_n = \mu_n(1+ O((nr_n^2)^{-1/2})) $  and
	\eqref{e:EPoUn2} follows from \eqref{e:meanasymp}.
	Also by \eqref{e:basicLLN} from Theorem \ref{t:momLLN},
	and Proposition \ref{p:average2d},
	\begin{align*}
	\E \Big[ \Big| \frac{\xi_n}{\mu_n} -1 \Big| \Big]
		\leq \E \Big[ \Big| \frac{\xi_n}{I_n} \big(
		\frac{I_n}{\mu_n} -1 \big) \Big| \Big]
		+ \E \Big[ \Big| \frac{\xi_n}{I_n} -1 \Big| \Big]
		= O((nr_n^2)^{-1/2} + I_n^{-1/2}),
	\end{align*}
	and hence
	\eqref{e:L1Un2}.

		Suppose $d \geq 3$.
		By Proposition \ref{p:average3d+}, as $n \to \infty$ we have
	 $I_n = \mu_n \Big(1+ O\Big( \big( \frac{\log(nr_n^d)}{nr_n^d}
	\big)^2 \Big) \Big)$.
	Hence using  \eqref{e:meanasymp} we have \eqref{e:EUn3+}.
	Also by \eqref{e:basicLLN} from Theorem \ref{t:momLLN},
	and  Proposition \ref{p:average3d+}, we have
	\begin{align*}
	\E \Big[ \Big| \frac{\xi_n}{\mu_n} -1 \Big| \Big]
		\leq \E \Big[ \Big| \frac{\xi_n}{I_n} \big(
		\frac{I_n}{\mu_n} -1 \big) \Big| \Big]
		+ \E \Big[ \Big| \frac{\xi_n}{I_n} -1 \Big| \Big]
		= O\Big(\big(\frac{\log(nr_n^d)}{nr_n^d}\big)^{2} + I_n^{-1/2}
		\Big),
	\end{align*}
	and hence \eqref{e:L1BiUn3+}.
\end{proof}

\section{Asymptotics of variances}
\label{s:asympvar}

Throughout this section we make the same assumptions on $d$, $A$
and $f$ that were set out at the start of Section \ref{s:prelims}.
%as before.
%jassume that $d \geq 2$ and
%$A \subset \R^d$ is compact and connected with
%$\partial A \in C^2$ and $A= \overline{A^o}$, or $d=2$ and $A =[0,1]^2$,
%and that $f$ is continuous on $A$ with $f_0 >0$.
We also assume that
\eqref{c:lowerbound} and \eqref{c:upperbound} hold,
i.e. that $nr_n^d \to \infty$ and $\liminf(I_n) >0$
as $n \to \infty$.

We shall prove that if $\xi_n$ denotes any of $K_n-1, K'_n-1, R_n$ or
$R'_n$, then   $\Var [\xi_n]$ is asymptotic to $I_n$
(which was  defined at \eqref{e:def:I_n}) as $n \to \infty$;
in the case of $\Var[K_n-1]$ and $\Var[R_n]$ we require
the extra condition $d \geq 3$.

Later we shall show that the number of non-singleton components has
negligible  variance
compared to the number of singletons.  This goal will be achieved by estimating separately the variance for the number of non-singleton components of
small (i.e., smaller than $\delta r_n$), medium and large
(i.e., larger than $\rho r_n$) diameter, and showing that each of these
three variances is $o(I_n)$;
the constants $\delta, \rho$ will be chosen later.

\subsection{Variances for small components: Poisson input}
\label{ss:varsmallPo}

Next we consider for $G(\eta_n,r_n)$
the number of small non-singleton components $K'_{n,0,\rho}$
and the number of vertices in such components, $R'_{n,0,\rho}$
(as defined at \eqref{e:Knrdef}), 
for suitably small (fixed) $\rho$.

\begin{proposition}
	\label{p:varT_n}
	There exists $\rho_0 >0$ such that if  $0 < \rho < \rho_0$
	then as $n \to \infty$ we have
\begin{align}
	\max(\Var[K'_{n,0,\rho}],\Var[R'_{n,0,\rho}] ) = O(
	(nr_n^d)^{1-d}
	I_n
	).
	\label{e:0529f}
\end{align}
%for some finite $c>0$ depending only on $d$ and $f_0$.
\end{proposition}
%-----------------------------

We divide the proof of this proposition into a series of lemmas. 
Given $\rho >0$ and given $n$,
for $x,y \in A$ define the events 
$
\scr T_{x} :=  \scr M_{n,0,\rho}(x, \eta_n)$
and
$\scr T_{x,y} := \scr M_{n,0,\rho}(x, \eta_n^y),
$
where  $\scr M_{n,\eps,K}(\cX)$ was defined
at \eqref{e:medevent1}.
Also, recalling the definition of $\scr F_n(x,\cX)$ at
\eqref{e:defZ}, 
set $\scr E_{x} := \scr T_x \cap \scr F_n(x, \eta_n^x)$
and 
$\scr E_{x,y} := \scr T_{x,y} \cap \scr F_n(x, \eta_n^{x,y})$.
%{\bf [Eventually drop subscript $0$ in notation.]}
We begin with the following bound based on the Mecke formula.
\begin{lemma} \label{l:VE} Suppose $\rho  \in (0,1)$. Then
\begin{align}\label{e:varT_n,0}
\Var[R'_{n,0,\rho}] -\EE[R'_{n,0,\rho}] \leq 
	n^2 \int_A \int_{A \cap B_{4r_n}(x)}
	 \PP[\scr T_{x,y} \cap \scr T_{y,x}]
	 \nu(dy) \nu( dx);
	 \\
\Var[K'_{n,0,\rho}] -\EE[K'_{n,0,\rho}] \leq 
	n^2 \int_A \int_{A \cap B_{4r_n}(x)}
	 \PP[\scr E_{x,y} \cap \scr E_{y,x}]
	 \nu(dy) \nu( dx).
	 \label{e:varKn0}
\end{align} 
\end{lemma}
\begin{proof}
By the Mecke formula, we have
$
\E[R'_{n,0,\rho} ] = n \int_{A} \Pr[ \scr T_{x} ] \nu (dx).
$
Using this and  the multivariate Mecke formula we obtain that
\begin{align*}
	\E[R'_{n,0,\rho}(R'_{n,0,\rho}-1)] - \E[R'_{n,0,\rho}]^2 = 
	n^2 \int_A \int_A (\Pr[ \scr T_{x,y} \cap \scr T_{y,x} ]-
	\Pr[\scr T_{x}] \Pr[\scr T_{y}]) \nu(dy) \nu(dx).
	\end{align*}
	For $\|y-x\| > 4r_n > 2 (1 + \rho) r_n$ we have
	$\Pr[\scr T_{x,y}]= \Pr[\scr T_{x}] \Pr[\scr T_{y}]$,
	and \eqref{e:varT_n,0} follows.

	The proof of  \eqref{e:varKn0} is identical,
	with $K'_{n,0,\rho}$ replacing $R'_{n,0,\rho}$
	and $\scr E_{x,y}$ replacing $\scr T_{x,y}$
	throughout.
\end{proof}

The rest of the proof of Proposition \ref{p:varT_n}
is devoted to estimating the double integral at
\eqref{e:varT_n,0}. 
We deal separately with the integrals over pairs $(x,y)$ satisfying
%distinguish three cases according to the distance between $x$ and $y$,
%namely 
(i)  $\|y-x\| >r_n$; (ii) 
$\rho r_n < \|y-x\| \leq r_n$, and (iii)
$\|y-x\| \leq \rho r_n$.
Let $\delta_1$ be as in Lemma \ref{l:A1}.

\begin{lemma}\label{l:6.4}
	Suppose
	$0 < \rho < \min( (f_0 \delta_1/(2\theta_d \fmax))^{1/d},1)$.
	Then as $n \to \infty$ we have
	\begin{align}
	n^2 \int_A \int_{A \cap B_{4r_n}(x) \setminus B_{r_n}(x)} 
	\Pr[\scr T_{x,y} \cap \scr T_{y,x}] \nu (dy) \nu(dx)
		= O(I_n \exp(- (\delta_1 f_0/3) nr_n^d)).
		\label{e:intannu}
		\end{align}
\end{lemma}
\begin{proof}
	Since $\Pr[ \scr T_{x,y} \cap \scr T_{y,x}]
	{\bf 1}\{r_n < \|y-x\| \leq 4r_n\}$ is symmetric in $x$ and $y$
	it suffices to prove the estimate for the integral
	restricted to $(x,y) \in A \times A$ with $x \prec y$, i.e. 
	$y \in A_x$.
	For such $(x,y)$,
%Let $x,y \in A$ with $x \prec y$.
	if
	$\scr T_{x,y}\cap \scr T_{y,x}$ occurs, then
	$\eta_n \cap  
	(B_{r_n}(x)\cup B_{r_n}(y))\setminus ( B_{\rho r_n}(x) \cup B_{\rho r_n}(y))
	= \emptyset$.
	Hence 
\begin{align*}
	\PP[\scr T_{x,y}\cap \scr T_{y,x} ] & \le 
	\exp(-  n \nu[ (B_{r_n}(x)\cup B_{r_n}(y))\setminus (B_{\rho r_n}(x) \cup 
	B_{\rho r_n}(y) )]) \\
	 & \leq \exp( -n \nu(B_{r_n}(x)) -n \nu(B_{r_n}(y)
	 \setminus B_{r_n}(x))
	+ 2n \theta_d \fmax (\rho r_n)^d).
\end{align*}
	By Lemma \ref{l:A1},
	if $\|x-y\|  \geq r_n$ then
	$\nu(B_{r_n}(y)\setminus B_{r_n}(x))
	\geq 2 f_0 \delta_1 r_n^d$. Therefore if we take $\rho$ to be
	so small that $2  \theta_d \fmax \rho^d < f_0 \delta_1$,
	the third (positive) term in the exponent is less
	than half the second (negative) term. 
	Hence
	$
\PP[\scr T_{x,y}\cap \scr T_{y,x}]\le
	e^{- n \nu(B_{r_n}(x)) - \delta_1 f_0   n  r_n^d}.
	$
It follows that
\begin{align*}
	n^2 \int_A \int_{A_x \cap B_{4r_n}(x) \setminus B_{r_n}(x)} 
	\Pr[\scr T_{x,y} \cap \scr T_{y,x}] \nu (dy) \nu(dx)
	& \leq n I_n \theta_d \fmax (4r_n)^d e^{-\delta_1 f_0 nr_n^d},
\end{align*}
	and \eqref{e:intannu} follows.
\end{proof}

\begin{lemma}
	\label{l:P0}
	Let $x,y\in A$ with $\|x-y\|\in (\rho r_n, r_n]$. Then
	$\PP[\scr T_{x,y}\cap \scr T_{y,x}]=0$.
\end{lemma}
\begin{proof}
	The condition on $\|x-y\|$ implies 
	%This set of relative locations imposes
	that $y \in \cC_{r_n}(x,\eta_n^y)$
	%=\cC_{r_n}(y,\eta_n^x)$ and that also
and $\diam(\cC_{r_n}(x,\eta_n^y)) >\rho r_n$, which 
	negates the event $\scr T_{x,y}$.
\end{proof}
\begin{lemma}
	\label{l:close}
	Suppose
	 $0 < \rho < \min((\delta_1 f_0/( \theta_d \fmax))^{1/(d-1)},1)$.
	 %\big( \frac{\delta_1 f_0}{2 \theta_d \fmax} \big)^{1/(d-1)}$,
	%there exists $c >0 $ such that
	Then as $n \to \infty$ we have
	\begin{align}
	n^2	\int_A \int_{A \cap B_{\rho r_n}(x)}
		\PP[\scr T_{x,y}\cap \scr T_{y,x}] \nu(dy) \nu(dx)
		= O(  (nr_n^d)^{1-d} I_n  ).
		\label{e:0529d}
	\end{align}
\end{lemma}
\begin{proof} Let $x,y \in A$ with $\|x-y \| \in (0,\rho r_n]$.
	Then $\scr T_{x,y} = \scr T_{y,x}$. Define event
	$$
	\scr N_{x,y} := 
	\{ \eta_n ((B_{r_n}(x) \cup B_{r_n}(y)) \setminus B_{\|y-x\|}(x)) =0 \}.
	$$
	By assumption $\fmax \rho^{d-1} \theta_d <  \delta_1 f_0 $. 
	If $x \prec y$, using Lemma \ref{l:A1} yields
\begin{align}
	\PP[  \scr N_{x,y} ]
	& \leq \exp(-n  \nu(B_{r_n}(x)) - 2n \delta_1 f_0 r_n^{d-1}\|y-x\| +
	n \fmax \theta_d \|y-x\|^d)
	\nonumber
	\\
	& \leq \exp(-n  \nu(B_{r_n}(x)) -n \delta_1 f_0 r_n^{d-1}\|y-x\| ).
	\label{e:0529a}
\end{align}
Similarly, if $y \prec x$ then
	\begin{align}
	\PP[  \scr N_{x,y} ]
		\leq \exp(-n  \nu(B_{r_n}(y)) -n \delta_1 f_0 
		r_n^{d-1}\|x-y\|).
		\label{0529b}
\end{align}
	Hence, recalling $A_x  := \{y \in A: x \prec y\}$  and
	using Fubini's theorem we obtain that
	\begin{align}
		&	n^2 \int_A \int_{A \cap B_{\rho r_n}(x)} \Pr[ \scr N_{x,y} ] \nu(dy)
		\nu(dx)
		\nonumber \\
		& \leq  n^2 \int_A \int_{A_x \cap B_{\rho r_n}(x) } 
	%	\Pr[ \scr N_{x,y} ] 
		e^{-n\nu(B_{r_n}(x)) - n \delta_1 f_0r_n^{d-1}\|y-x\|}
		\nu(dy ) \nu(dx)
		\nonumber \\
		& +  n^2 \int_A \int_{A_y \cap B_{\rho r_n}(y) } 
	%	\Pr[ \scr N_{x,y} ] 
		e^{-n\nu(B_{r_n}(y)) - n \delta_1 f_0r_n^{d-1}\|x-y\|}
		\nu(dx ) \nu(dy)
		\nonumber \\
		& \leq 2 n I_n  \fmax 
		\int_{ B_{\rho r_n}(o)} e^{-n \delta_1 f_0 r_n^{d-1}
		\|u\|} du 
		\nonumber \\
		& = 2 n \fmax I_n (nr_n^{d-1})^{-d} \int_{B_{n \rho r_n^d(o)}}
		e^{- n \delta_1 f_0 \|v\|} dv
		%\nonumber \\
		%&
		= O( (n r_n^d)^{1-d} I_n).
		\label{e:0529c}
	\end{align}

	Next, let $z$ denote 
	the furthest point from $x $ in $\cC_{r_n}(x,\eta_n^{x,y})$.
	If $z=y$ then
	$\cN_{x,y}$ occurs. Thus if
	 $\scr T_{x,y} \setminus \cN_{x,y} $ occurs
	 then $z \neq y$
	 %the furthest point from $x$ in $\cC_r(x, \eta_n^{x,y})$
	 %is not $y$, and calling this furthest point $z$,
	 and hence 
	 %we see that
	 %thus there is a point
	 $z \in \eta_n$ with $\|y-x\| < \|z-x\| \leq \rho r_n$, and 
	 moreover $\eta_n \cap (B_{r_n}(x) \cup B_{r_n}(z))
	 \setminus B_{\|z-x\|}(x))
	 = \emptyset $.  That is,
	$$
	\{ \scr T_{x,y} \setminus \scr N_{x,y} \} \subset 
	\{
		\exists z \in \eta_n \cap B_{\rho r_n}(x) \setminus
		B_{\|y-x\|}(x):
	\eta_n((B_{r_n}(x) \cup B_{r_n}(z)) \setminus B_{\|z-x\|}(x) ) =0
	\}.
	$$
	Hence by Markov's inequality, the Mecke formula and
	Fubini's theorem,
	\begin{align*}
		& 	n^2 \int_A \int_{A \cap B_{\rho r_n}(x)}
		\Pr[\scr T_{x,y} \setminus \scr N_{x,y}]
		\nu(dy) \nu(dx) \\
		& \leq n^3 \int_A %(\fmax \theta (\rho r)^d)
		\int_A
		\int_{B_{\rho r_n}(x) \setminus B_{\|y-x\|}(x) }
		e^{- n \nu [(B_{r_n}(x) 
		\cup B_{r_n}(z)) \setminus B_{\|z-x\|}(x)] }
		\nu(dz) \nu(dy) \nu(dx)
		\\
		& \leq n^3 \int_A \int_{B_{\rho r_n}(x) }
		e^{-n \nu[(B_{r_n}(x) \cup B_{r_n}(z))\setminus B_{\|z-x\|}(x)]}
		(\fmax \theta_d \|z-x\|^d) \nu(dz) \nu(dx).
	\end{align*}
	By the same estimates as at \eqref{e:0529a} and \eqref{0529b}
	(now with $z$ instead of $y$), the last expression is bounded
	by
	\begin{align*}
		&	n^3 \fmax \theta_d \int_A \int_{A_x \cap B_{\rho r_n}(x)}
		\|z-x\|^d e^{-n \nu(B_{r_n}(x))- n \delta_1 f_0 
		r_n^{d-1}\|z-x\|}
		\nu(dz) \nu(dx)
		\\
		&+ n^3 \fmax \theta_d \int_A \int_{A_z \cap B_{\rho r_n}(z)}
		\|x-z\|^d e^{-n \nu(B_{r_n}(z))- n \delta_1 f_0 
		r_n^{d-1}\|x-z\|}
		\nu(dx) \nu(dz) \\
		&	\leq 2 n^2 \fmax^2 \theta_d I_n \int_{B_{\rho r_n}(o)}
		e^{- n \delta_1 f_0 r_n^{d-1} \|u\|} \|u\|^d du
		\\
		&	= O(n^2 (nr_n^{d-1})^{-2d} I_n ) 
		= O((nr_n^d)^{2-2d} I_n).
	\end{align*}
	Combining this with \eqref{e:0529c} yields \eqref{e:0529d}.
\end{proof}

%We are ready to finish the proof of  the variance estimate.  

\begin{proof}[\it Proof of Proposition \ref{p:varT_n}.]
	Applying Lemmas 
	\ref{l:6.4}, \ref{l:P0} and \ref{l:close} we obtain that
	provided $\rho$ is taken small enough, we have
	as $n \to \infty$ that
	\begin{align}
	n^2 \int_A \int_{A \cap B_{4r_n}(x)} \Pr[ \scr T_{x,y} 
	\cap \scr T_{y,x} ] \nu(dy) \nu(dx) = O((nr_n^d)^{1-d} I_n ).
		\label{e:0903e}
		\end{align}
	Hence by Lemma  \ref{l:VE}, we obtain that
	\begin{align}
	(\Var[R'_{n,0,\rho}] - \E[R'_{n,0,\rho}])^+ = O((nr_n^d)^{1-d} I_n ).
		\label{e:0529e}
	\end{align}
	Also, by Lemma \ref{l:V0eps},
	provided
	$\rho$ is small enough we have 
	$ \E[R'_{n,0,\rho}] = O((nr_n^d)^{1-d} I_n)$. Combining this
	with  \eqref{e:0529e} and using the nonnegativity of variance,
	we obtain the statement about $R'_{n,0,\rho}$ in
	\eqref{e:0529f}.

	Since  $\scr E_{x,y} \subset \scr T_{x,y}$ we still have
	\eqref{e:0903e} with $\scr T_{x,y}$ replaced by
	$\scr E_{x,y}$. We can then derive the statement about
	$K'_{n,0,\rho}$ by a similar argument;
	instead of Lemma \ref{l:V0eps} we now use part of the proof of
	Proposition \ref{p:J}.
\end{proof}

\subsection{Variances for small components: binomial input}

Next we consider for $G(\X_n,r(n))$
the number of small non-singleton components $K_{n,0,\rho}$
and the number of vertices in such components, $R_{n,0,\rho}$
(as defined at \eqref{e:Knrdef}),
%\eqref{e:defVn}),
for suitably small (fixed) $\rho$.

While the asymptotic variance for small components in a Poisson sample
was obtained above by computing the first two moments and
exploiting the spatial independence of the Poisson process,
we shall bound the variance
for small components in a binomial sample
by a very different argument, namely, 
the Efron-Stein inequality from Lemma \ref{l:Poinc.ES}.
This does not work so well, in the sense that  
our bound does the job only in dimension $d\ge 3$.

\begin{proposition}[Variance estimates for small non-singleton
	components: binomial input]
	\label{p:varRsmall}
	If $d \geq 3$ then there exists
	$\delta_{11} >0 $
	such that if $0 < \rho \leq \delta_{11}$ then
	$\Var(K_{n,0,\rho}) = O((nr_n^d)^{2-d}I_n)$ as $n \to \infty$, and
	$\Var(R_{n,0,\rho}) = O((nr_n^d)^{2-d}I_n)$ as $n \to \infty$.
\end{proposition}
\begin{proof}
	By the Efron-Stein inequality  \eqref{e:p4.2},
	\begin{align*}
		\Var[R_{n,0,\rho}] &
		\leq n \int_A   \E [ (D_x R_{n,0,\rho}(\X_{n-1}))^2] \nu(dx)
	\\
		& = n \int_A   \E [ (D_x^+ R_{n,0,\rho}(\X_{n-1}))^2] \nu(dx)
		+ n \int_A   \E [ (D_x^- R_{n,0,\rho}(\X_{n-1}))^2] \nu(dx).
	\end{align*}
	 Similarly
	\begin{align*}
		\Var[K_{n,0,\rho}] &
		 \leq n \int_A   \E [ (D_x^+ K_{n,0,\rho}(\X_{n-1}))^2] \nu(dx)
		+ n \int_A   \E [ (D_x^- K_{n,0,\rho}(\X_{n-1}))^2] \nu(dx).
	\end{align*}
	Moreover for all finite  $\X \subset \R^d$ and $x \in \R^d \setminus \X$
	we have $D_x^+ K_{n,0,\rho}(\X) \leq D_x^+ R_{n,0,\rho}(\X)$ and
	 $D_x^- K_{n,0,\rho}(\X) \leq D_x^- R_{n,0,\rho}(\X)$.
	Therefore the result follows from the next two
	lemmas.
\end{proof}
\begin{lemma}
  Let $\rho$ be as in Lemma \ref{l:V0eps}.
	%Let $\delta_1$ be as in Lemma \ref{l:A1}.
	%Suppose $0 < \rho < \min( (\delta_1 f_0/(\theta_d \fmax))^{1/(d-1)}, 
	%\frac12)$.
	Then as $n \to \infty$ we have
		 \begin{align}
			 n \int_A   \E [ (D_x^+ R_{n,0,\rho}(\X_{n-1}))^2] \nu(dx)
			 = O((nr_n^d)^{1-d} I_n).
			 \label{e:ED+}
		 \end{align}
\end{lemma}
\begin{proof}
	Note
	$ D_x^+ R_{n,0,\rho}(\X_{n-1}) $ is non-zero only if
	$0 < \diam \cC_{r}(x,\X_{n-1}^x) \leq \rho r_n$,
	in which case
	 $ D_x^+ R_{n,0,\rho}(\X_{n-1})  $ is either 1 (if
	 $ \#(\X_{n-1} \cap B_{\rho r_n}(x)) > 1$) or
	 2 (if 
	 $\# (\X_{n-1} \cap B_{\rho r_n}(x)) = 1$). 
	 Hence
	 \begin{align*}
		 n \int_A   \E [ (D_x^+ R_{n,0,\rho}(\X_{n-1}))^2] \nu(dx)
		 & \leq 4 n \int_A \Pr[ 0 < \diam \cC_{r_n}(x,\X_{n-1}^x) \leq 
		 \rho r_n]
		 \nu(dx)
		 \\
		 & = 4 \E[R_{n,0,\rho}(\X_n)]
		 %\label{e:D+}
	 \end{align*}
	 Then the result follows from Lemma \ref{l:V0eps}.
	 \end{proof}

	 \begin{lemma}
		 Suppose $0< \rho < \min((\delta_1 f_0/(2 \fmax \theta_d))^{1/(d-1)},1)$, where
	 $\delta_1 $ is as in Lemma \ref{l:A1}. 
		 Then as $n \to \infty$,
		 we have
		 \begin{align}
			 n \int_A \E[D_x^-R_{n,0,\rho}(\X_{n-1})]
			 \nu(dx)
			 = O( (nr_n^d)^{2-d} I_n).
			 \label{e:D-}
		 \end{align}
	 \end{lemma}
\begin{proof}
	%$\rho \in (0,1)$ with $4 \fmax \theta_d \rho^{d-1} < \delta_1 f_0$.
	For  $x \in A$, observe that $D_x^-R_{n,0,\rho}(\X_{n-1})$
		  is bounded above
		  by $N_{1,x}$, where $N_{1,x}$ denotes
		 the number of vertices $y \in \X_{n-1}$ such
		 that $\|y-x\| \leq 2r_n$
		 and $ 0 < \diam (\cC_{r_n}(y,\X_{n-1})) \leq \rho r_n$.
		 Therefore
		 \begin{align}
		 (D_x^-R_{n,0,\rho}(\X_{n-1}))^2 \leq N_{1,x}^2 = N_{1,x} + N_{1,x}(N_{1,x}
		 -1).
			 \label{e:DNN}
		 \end{align}

		 Let $N_{2,x}$ be the number of ordered pairs $(y,z)$ of
		 distinct points of $\X_{n-1} \cap B_{2r_n}(x)$ such that
		 $0 < \diam (\cC_{r_n}(y,\X_{n-1})) \leq \rho r_n$,
		 and $y \prec u$
		 for all $u \in \cC_{r_n}(y,\X_{n-1}) \setminus \{y\}$,
		 and $z$ is
		 the point in $\cC_{r_n}(z,\X_{n-1})$ furthest from $y$.
		 Let $N_{3,x}$ be the number of ordered triples $(z,u,y)$
		 of distinct points of $\X_{n-1} \cap B_{2r_n}(x)$ such
		 that $0 < \diam(\cC_{r_n}(z,\X_{n-1})) \leq \rho r_n$, and
		 $z \prec v$ for all $v \in \cC_{r_n}(z,\X_{n-1}) \setminus \{z\}$,
		 and $u$ is the point of $\cC_{r_n}(z,\X_{n-1})$ furthest
		 from $z$, and $y$ is another point of $\cC_{r_n}(z,
		 \X_{n-1})$.

		 Then $N_{1,x} \leq  2 N_{2,x} + N_{3,x} $.
		 For $n$ large we have
		 \begin{align*}
			 \E[N_{2,x}] & \leq n^2
			 \int_{A \cap B_{2r_n}(x)} \int_{A_y \cap 
			 B_{\rho r_n}(y)}
			 (1- \nu[(B_{r_n}(y) \cup B_{r_n}(z)) \setminus
			 B_{\|z-y\|}(y)])^{n-3} \nu(dz) \nu(dy)
			 \\
			 & \leq 2 n^2 \int_{A \cap B_{2r_n}(x)}
			 \int_{A_y \cap B_{\rho r_n}(y)}
			 e^{-n \nu(B_{r_n}(y)) - \delta_1 f_0 nr_n^{d-1}\|z-y\|}
			 \nu(dz) \nu(dy).
		 \end{align*}
		 Therefore using Fubini's theorem we obtain that
		 \begin{align}
			 n \int_A \E[N_{2,x}] \nu (dx)
			 & \leq 2 n^3 \int_A e^{-n \nu(B_{r_n}(y))} 
			 \int_{A_y \cap B_{\rho r_n}(y)}
			 e^{-\delta_1 f_0 nr_n^{d-1}
			 \|z-y\|}  \int_{B_{2r_n}(y)} \nu(dx) \nu(dz) \nu(dy)
			 \nonumber \\
			 & \leq 2^{d+1} \theta_d \fmax^2 n^2 r_n^d I_{n} \int_{\R^d}
			 e^{-\delta_1 f_0 nr_n^{d-1} \|u\|} du
			 \nonumber \\
			 & %\leq c' n^2 r_n^d I_n (nr^{d-1})^{-d}
			 = O( (nr_n^d)^{2-d} I_n).
			 \label{e:N2}
		 \end{align}

		 Next, we have % for a new constant $c$
		 that for $n$ large
		 \begin{align*}
			 \E[N_{3,x}] \leq n^3 \int_{B_{2r_n}(x)}
			 \int_{B_{\rho r_n}(z) \cap A_z} \int_{B_{\|u-z\|}(z)}
			 (1- \nu[(B_{r_n}(z) \cup B_{r_n}(u)) \setminus B_{\|u-z\|}(z)]
			 )^{n-4} \\
			 \nu(dy) \nu(du) \nu(dz) 
			 \\
			 \leq 2 \theta_d \fmax
			 n^3 \int_{B_{2r_n}(x)} \int_{B_{\rho r_n}(z)
			 \cap A_z} \|u-z\|^d
	%\exp(-n\nu(B_{r_n}(z)) - (\delta_1 f_0/2) n r_n^{d-1} \|u-z\|) 
			 e^{-n\nu(B_{r_n}(z)) - \delta_1 f_0 n r_n^{d-1} \|u-z\|} 
			 \nu(du) \nu(dz).
		 \end{align*}
		 Then using Fubini's theorem and a change of
		 variable $v=u-z$ we obtain that
		 \begin{align}
			 n \int_A \E[N_{3,x}] \nu(dx) 
			 & \leq 2^{d+1} \theta_d^2 \fmax^2
			 n^4 r_n^d \int_A e^{-n \nu(B_{r_n}(z))}
			 \nu(dz)
			 \int_{\R^d} e^{-\delta_1 f_0 nr_n^{d-1}\|v\|}
			 \|v\|^d dv 
			 \nonumber \\
			 & = O \big( n^3 r_n^d I_n (nr_n^{d-1})^{-2d} \big)
			 = O \big( (nr_n^d)^{3-2d} I_n \big).
			 \label{e:N3}
		 \end{align}
		 Combined with \eqref{e:N2} this shows that
		 \begin{align}
			 n \int_A \E[N_{1,x}] \nu(dx) 
			 = O \big((nr_n^d)^{2-d}I_n \big).
			 \label{e:N1x}
		 \end{align}

		 Next consider $N_{1,x}(N_{1,x}-1)$, which equals
		 the number of ordered pairs $(y,z)$ of distinct points of
		 $\X_{n-1} \cap B_{2r_n}(x)$ such that both
		 $\cC_{r_n}(y,\X_{n-1})$ and 
		 $ \cC_{r_n}(z,\X_{n-1})$  have Euclidean  diameter  in the
		 range $(0,\rho r_n]$. 
		 For such $(y,z)$ we cannot have 
		 $\rho r_n < \|y-z\| \leq r_n$; 
		 we distinguish between the cases
		 where $\|y -z\| \leq \rho r_n$ and where $\|y-z\| > r_n$.

		 Let $N_{4,x}$ be
		 the number of ordered pairs $(y,z)$ of distinct points of
		 $\X_{n-1} \cap B_{2r_n}(x)$ such that
		 $\|y-z\| \leq \rho r_n$ and
		 $\diam(\cC_{r_n}(y,\X_{n-1})) \leq \rho r_n$.

		 Let $N_{5,x}$ be the number of
		 ordered quadruples $(u,v,y,z)$  of distinct points of
		 $\X_{n-1} \cap B_{2r_n}(x)$ such that $u \prec w$
		 for all $w \in \cC_{r_n}(u,\X_{n-1})$, and $v$ is
		 the furthest point from $u$ in $\cC_{r_n}(u,\X_{n-1})$
		 and $y,z$ are two further points in $\cC_{r_n}(u,\X_{n-1})$
		 and $\diam (\cC_{r_n}(u,\X_{n-1})) \leq \rho r_n$. Then
		 $$
		 N_{4,x} \leq 2 N_{2,x} + 4 N_{3,x} + N_{5,x}.
		 $$

		 For $n \geq 4$ we have that
		 \begin{align*}
			 \E[N_{5,x}] \leq n^4 \int_{A \cap B_{2r_n}(x)}
			 \int_{A_u \cap B_{\rho r_n}(u)} 
			 & (\nu(B_{\|v-u\|}(u)))^2 
			 \\
			  \times & (1- \nu[(B_{r_n}(u) 
			 \cup B_{r_n}(v)) \setminus B_{\|v-u\|}(u)])^{n-4}
			 \nu(dv) \nu(du),
		 \end{align*}
		 and hence by Fubini's theorem,   for $n$ large
		 \begin{align*}
			 n \int_A \E[N_{5,x}] \nu(dx) & \leq
			 \theta_d^3 2^{1+d} \fmax^3
			 n^5 r_n^d \int_A e^{-n \nu(B_{r_n}(u))} \nu(du)
			 \int_{\R^d} e^{-f_0 \delta_1  n r_n^{d-1} \|w\|}
			 \|w\|^{2d} dw
			 \\
			 & = O(n^4 r_n^d I_n (nr_n^{d-1})^{-3d}) 
			  = O( (nr_n^{d})^{4-3d}I_n). 
		 \end{align*}
		 Combined with \eqref{e:N2} and \eqref{e:N3} this shows that
		 \begin{align}
			 n \int_A \E[N_{4,x} ] \nu(dx) =
			 O((nr_n^d)^{2-d} I_n ) .
			 \label{e:N4x}
		 \end{align}

		 Let $N_{6,x}$ be the number of ordered pairs $(y,z)$
		 of distinct points of $\X_{n-1} \cap B_{2r_n}(x)$
		 such that $\|y-z\| > r_n$, $y \prec z$ and
		 both $\diam(\cC_{r_n}(y,\X_{n-1}))$ and
		  $\diam(\cC_{r_n}(z,\X_{n-1}))$ lie in the range $(0,\rho r_n]$.
		  Then $N_{1,x}(N_{1,x}-1) = N_{4,x} + N_{6,x}$ and
		 \begin{align*}
			 \E[N_{6,x}] \leq n^2 \int_{A \cap B_{2r_n}(x)}
			 \int_{A_y \cap B_{2r_n}(x) \setminus B_{r_n}(y)} 
			 (1- 
			 \nu[(B_{r_n}(y) \cup B_{r_n}(z)) \setminus (B_{\rho r_n}(y) 
			 \cup B_{\rho r_n}(z)) ])^{n-3}
			 \\
			 \nu(dz) \nu(dy).
		 \end{align*}
		 By our choice of $\rho$ we have
		 $2 \fmax \theta_d \rho^d \leq \delta_1 f_0$.
		 Then by Lemma \ref{l:A1}, for $n$ large
			 and $y \in A$, $z \in A_y$ with $\|z-y\| > r_n$,
			 \begin{align*}
				 \nu[(B_{r_n}(y) \cup B_{r_n}(z)) \setminus (B_{\rho r_n}(y) 
				 \cup B_{\rho r_n}(z)) ]
			 & \geq \nu(B_{r_n}(y)) + 2 \delta_1 f_0 r_n^d - 2
			 \fmax \theta_d (\rho r_n)^d
			 \\
			 & \geq \nu(B_{r_n}(y)) + \delta_1 f_0 r_n^d.
		 \end{align*}
		 Hence for $n$ large,
		 \begin{align*}
			 \E[N_{6,x}] \leq 2 n^2 \int_{A \cap B_{2r_n}(x)}
			 e^{-n \nu(B_{r_n}(y))- \delta_1 f_0 nr_n^d}
			 \fmax \theta_d (2r_n)^d \nu(dy),
		 \end{align*}
		 so by Fubini's theorem, for $n$ large
		 \begin{align*}
			 n \int_A \E[N_{6,x}] \nu(dx) & \leq
			 2^{1+2d} \fmax^2 \theta_d^2
			 n^3 r_n^{2d} \int_{A} e^{-n \nu(B_{r_n}(y))- \delta_1
			 f_0 nr_n^d }
			 \nu(dy)
			 \\
			 & = O((nr_n^d)^2e^{- \delta_1 f_0 nr_n^d} I_n )
			  = O(e^{-(\delta_1 f_0/2)nr_n^d}I_n ).
		 \end{align*}
		 Combined with \eqref{e:N4x} this shows that
		 \begin{align*}
			 n \int_A \E[N_{1,x}(N_{1,x}-1)] \nu(dx)
			 = O( (nr_n^d)^{2-d} I_n)
		 \end{align*}
		 and combined with \eqref{e:N1x} and \eqref{e:DNN}
		 this gives us \eqref{e:D-}.
\end{proof}

\subsection{Variance estimates for medium components}

We now consider  the `medium-size' component count, denoted
$K_{n,\eps,\rho}$ or
$K'_{n,\eps,\rho}$ (as defined at \eqref{e:Knrdef})
%\eqref{e:Knr'def}) 
with 
$0< \eps < \rho < \infty$. We also consider the
number of vertices in medium-sized components,
denoted $R_{n,\eps,\rho}$ or $R'_{n,\eps,\rho}$.
	We shall bound
	the variances  of all four of these quantities
	using Lemma \ref{l:Poinc.ES}, i.e. using
	the Poincar\'e or Efron-Stein inequality.

\begin{proposition}[Variance estimates for medium-sized components]
	\label{p:varRmed}
	Let $0 < \eps < 1 < \rho < \infty$, and let
	$\delta_2 = \delta_2(d,A, \eps, 2\rho)$ be as in 
	Lemma \ref{l:A2}.
	Let $\xi_n$ stand for any of $R_{n,\eps,\rho}$,
	$R'_{n,\eps,\rho}$, $K_{n,\eps,\rho}$ or
	$K'_{n,\eps,\rho}$.
	Then 
	%there exists $\delta_{12} = \delta_{12}(d,A,\eps, \rho)
	%>0$ such that
	$\Var(\xi_{n}) = O(e^{- (\delta_{2}/2) nr_n^d} I_n)$ 
	as $n \to \infty$.
\end{proposition}
\begin{proof}
	Note that
	$D_x^+K_{n,\eps,\rho}(\X) \leq D_x^+R_{n,\eps,\rho}(\cX)$
	and
	$D_x^-K_{n,\eps,\rho}(\X) \leq D_x^-R_{n,\eps,\rho}(\cX)$,
	for all $x,\X$.
	Analogously to the proof of Proposition \ref{p:varRsmall},
	 but using the Poincar\'e inequality instead of
	 the Efron-Stein inequality in the case of
	 the results for $R'_{n,\eps,\rho}$ and $K'_{n,\eps,\rho}$,
	 we can obtain
	the result from the next two lemmas.
\end{proof}
\begin{lemma}
	\label{l:DF}
	Let $0 < \eps  < 1 < \rho < \infty$, and 
	$\delta_2 = \delta_2(d,A,\eps,2\rho)$ as in Lemma \ref{l:A2}.
	Then as $n \to \infty$,
%	there
%	exists $\delta_{13} >0$ such that
	\begin{align}
	n \int_A \E[(D_x^- R_{n,\eps ,\rho}(\X_{n-1}))^2] \nu(dx) = O(
		e^{- (\delta_{2}/2) nr_n^d} I_n); 
		\label{e:ED-}
		\\
	n \int_A \E[(D_x^- R_{n,\eps ,\rho}(\eta_{n}))^2] \nu(dx) = O(
		e^{- (\delta_{2}/2) nr_n^d} I_n). 
		\label{e:PoED-}
	\end{align}
\end{lemma}
\begin{proof}
	Observe that
	$D_x^- R_{n,\eps ,\rho}(\X_{n-1})$  
	is bounded above by the number of vertices $y \in \X_{n-1} \cap
	B_{2 \rho r_n}(x)$ such that $\diam(\cC_{r_n}(y,\X_{n-1}) ) \in
	(\eps  r,\rho r_n]$. We denote this quantity by $N_{7,x}$.

	Let $N_{8,x}$ be the number of ordered pairs $(y,z)$ of distinct
	points of $\X_{n-1} \cap B_{3 \rho r_n}(x)$ such
	that $\diam(\cC_{r_n}(y,\X_{n-1}))\in
	(\eps r,\rho r_n]$ and $y \prec u$ for all $u \in \cC_{r_n}(y,\X_{n-1})
	\setminus \{y\}$. Then
	$
	N_{7,x} \leq 2 N_{8,x} 
	$

	Let $\delta_2 = \delta_2(d,A,\eps, 2\rho)$ be as in
	Lemma \ref{l:A2}.
	Fix $\delta >0$
	small,
	%as in the proof of \cite[Lemma 4.3]{PY21}
	%{\bf (Check this is the right lemma)}, 
	and
	discretize $\R^d$ into cubes  of side $\delta r_n$ as in that proof.
	Assume $4 \delta d^{3/2} < \min(\delta_2/\theta_d, 1)$,
	and also $\frac54(1- \delta) > \frac98$, and
	$2 \delta \theta_d \fmax < \delta_2 f_0/8$.
	%Assume $(1- (1 - \sqrt{d}\delta)^d ) \theta_d \fmax < \delta_2 f_0/9$,
	%where of the present paper.
	%and also $(1- \sqrt{d} \delta)^d \geq 3/4$.
	Then
	$$
	\E[N_{8,x} ] \leq n^2 \int_{B_{3 \rho r_n}(x)}
	\int_{B_{3 \rho r_n}(x)} \sum_\sigma (1- \nu(
	[(\sigma \cap A_y) \oplus B_{r(1- 
	\sqrt{d} \delta)}(o)] \setminus (\sigma \cap A_y) 
	))^{n-3}  \nu(dz) \nu(dy),
	$$
	where the sum is over a finite (and uniformly bounded)
	number of possible shapes $\sigma$ that could arise as
	the union of those cubes in the discretization containing
	points of $\cC_{r_n}(y,\X_{n-1})$.

	Using Lemma \ref{l:A2},
	%okkkthe continuity of $f$,
	\eqref{0703c}
	and the bound $(1-u)^d \geq 1-du$, we
	have  for $n$ large  that
	\begin{align*}
		& \nu([(\sigma \cap A_y) \oplus B_{(1-\sqrt{d} \delta)r_n} (o) ]
	\setminus (\sigma \cap A_y) ) \\
		& \geq
		(1-\delta) f(y) [ \lambda(B_{(1- \sqrt{d}\delta)r_n}
		(y) \cap A) + 2 \delta_2 (1- \sqrt{d} \delta )^d r_n^d ]
		\\
		& 	\geq
		(1-\delta )f(y) [
			\lambda(B_{r_n}(y) \cap A)
			- (1- (1- \sqrt{d} \delta)^d) \theta_d r_n^d
			+ (3/2) \delta_2 r_n^d]
		\\
		& 	\geq 
		(1- 2 \delta) \nu(B_{r_n}(y) )
		%- 2\delta f(y) \lambda(B_{r_n}(y) \cap A)
		+ (5/4)(1- \delta) f(y) \delta_2 r_n^d
		\\
		& 	\geq 
		\nu(B_{r_n}(y) )
		- 2 \theta_d \delta \fmax r_n^d + (9/8) \delta_2 f(y) r_n^d  
		\\
		& 	\geq \nu(B_{r_n}(y)) +
	\delta_2 f_0 r_n^d,
	\end{align*}
	and thus there exists a constant $c' >0$ such that  for $n$ large
	\begin{align}
	\E[N_{8,x}] \leq c' n^2
	\int_{B_{3 \rho r_n}(x)} \int_{B_{3 \rho r_n}(x)}
		e^{-n \nu(B_{r_n}(y)) - \delta_2 f_0 n r_n^d} \nu(dz) \nu (dy).
		\label{e:EN8}
	\end{align}
	Hence by Fubini's theorem there is a constant $c''$ such that
	for $n$ large
	\begin{align*}
		n \int_A \E[N_{8,x}] \nu(dx)  & \leq c'' n^3 r_n^{2d}
		\int_A e^{-n \nu(B_{r_n}(y)) - \delta_2 f_0 nr_n^d}  \nu(dy) 
		\\
		& = O((nr_n^d)^2 e^{-\delta_2 f_0 nr_n^d} I_n).
	\end{align*}

	Next, let $N_{9,x}$
	denote the number of ordered triples $(y,z,u)$ of distinct points
	of $\X_{n-1} \cap B_{3 \rho r_n}(x)$
	such that $\diam (\cC_{r_n}(y,\X_{n-1})) \in (\eps  r_n,\rho r_n]$
	and $y \prec v$ for all $v \in \cC_{r_n}(y,\X_{n-1}) \setminus \{y\}$.
	Then
	$$
	N_{7,x}(N_{7,x}-1) \leq 
	2 N_{8,x} + N_{9,x}.
	$$
	Using Lemma \ref{l:A2} again, we can find a new constant $c'>0$
	such that for $n$ large
	\begin{align*}
		\E[N_{9,x}] & \leq n^3 \int_{B_{3 \rho r_n}(x)}
		\sum_{\sigma} (1- \nu([(\sigma \cap A_y) \oplus 
		B_{(1- \sqrt{d} 
		\delta)r_n}(o) ] \setminus (\sigma \cap A_y)))^{n-4}
		(\nu(B_{3 \rho r_n}(x)))^2 \nu(dy)
		\\
		& \leq c' n^3 r_n^{2d} \int_{B_{3  \rho r_n}(x)}
		e^{- n \nu(B_{r_n}(y)) - \delta_2 f_0 nr_n^d} \nu(dy),
	\end{align*}
and hence by Fubini's theorem there is a further new constant $c''$ such that
	\begin{align*}
		n \int_A \E[N_{9,x}] \nu(dx) & \leq
		c'' n^4 r_n^{3d} \int_A e^{-n \nu(B_{r_n}(y))
		- \delta_2 f_0 nr_n^d} \nu(dy)
		\\
		& = O((nr_n^d)^3 e^{-\delta_2 f_0 nr_n^d} I_n).
	\end{align*}
	Combined with \eqref{e:EN8} this shows that
	\begin{align*}
		n \int_A \E[(D^-_x R_{n,\eps ,\rho}(\X_{n-1}))^2 ] \nu (dx)
		& \leq n \int_A \E[N_{7,x} + N_{7,x}(N_{7,x}-1)] \nu(dx)
		\\
		& = O((nr_n^d)^3 e^{-\delta_2 f_0 nr_n^d} I_n),
	\end{align*}
	and \eqref{e:ED-} follows.
	The proof of
	 \eqref{e:PoED-} is similar, using the Mecke formula.
\end{proof}
\begin{lemma}
	\label{l:DFplus}
	Let $0 < \eps  < 1 < \rho < \infty$, and let
	$\delta_2 = \delta_2(d,A,\eps,2\rho)$ be as in Lemma \ref{l:A2}.
	Then as $n \to \infty$,
	%there exists $\delta_{14} >0$ such that
	\begin{align}
	n \int_A \E[(D_x^+ R_{n,\eps ,\rho}(\X_{n-1}))^2] \nu(dx) = O(
		e^{- (\delta_{2}/2) nr_n^d} I_n); 
		\label{e:EDF+}
		\\
	n \int_A \E[(D_x^+ R_{n,\eps ,\rho}(\eta_n))^2] \nu(dx) = O(
		(^{- (\delta_{2}/2) nr_n^d} I_n). 
		\label{e:PoEDF+}
	\end{align}
\end{lemma}
\begin{proof}
	Let $\delta_2$ and $\delta$ be as in the previous proof.
	If $D_x^+R_{n,\eps ,\rho}(\X_{n-1}) >0$ then
	$\diam(\cC_{r_n}(x,\X_{n-1}^x) ) \in (\eps r_n,\rho r_n]$.
	We discretize $\R^d$ into cubes of side $\delta r_n$ as before.
	For each possible
	shape $\sigma$ (i.e., a union of cubes of side $\delta r_n$), 
let	$E_{x,\sigma}$ be the event that $\sigma $ is the
	shape induced by $\cC_{r_n}(x,\X_{n-1}^x)$, i.e. the union
	of those cubes in the discretization which contain at least one point
	of $\cC_{r_n}(x,\X_{n-1}^x)$. 
	Given $\cX, D \subset \R^d$ with $\cX$ finite,
	let $\cX(D):= \#(\X \cap D)$.
	%denote the number of elements of $\cX \cap D$.
	Then
	\begin{align*}
		(D_x^+R_{n,\eps ,\rho}(\X_{n-1}))^2 
		\leq \sum_{\sigma} 
		\1_{E_{x,\sigma}} (1+\X_{n-1}(\sigma))^2,
	\end{align*}
	and hence
\begin{align}
	n \int_A \E[(D^+_x R_{n,\eps ,\rho}(\X_{n-1}))^2] \nu(dx)
	\leq n \int_A \sum_{\sigma:x \in \sigma}
	(\Pr[E_{x,\sigma}] + 2 \E[\X_{n-1}(\sigma) \1_{E_{x,\sigma}} ]
	\nonumber \\
	+ \E[\X_{n-1}(\sigma)^2\1_{E_{x,\sigma}}])
	\nu(dx).
		\label{e:EDD}
\end{align}
	If $E_{x,\sigma}$ occurs there is a point $y$
	of $\X_{n-1} \cap \sigma$ with
	$y \prec z$ for all $z \in \X_{n-1} \cap \sigma \setminus \{y\}$, so
	using Lemma \ref{l:A2} as in the preceding proof,
	we obtain for $n$ large that
	\begin{align*}
		\Pr[E_{x,\sigma}] & \leq (n-1) \int_{\sigma }
		(1- \nu([(\sigma \cap A_y)
		\oplus B_{(1-\sqrt{d} \delta)r_n}(o)]
		\setminus (\sigma \cap A_y) ))^{n-2}
		\nu(dy)
		\\
		& \leq 2 n \int_{\sigma} 
		e^{-n\nu(B_{r_n}(y)) - \delta_2 f_0 n r_n^d} \nu(dy),
	\end{align*}
	and hence by Fubini's theorem there exist  constants $c', c''$
	such that
	\begin{align}
		n \int_A \sum_{\sigma: x \in \sigma } \Pr[E_{x,\sigma}] \nu(dx)
		& \leq 2 n^2 \int_A  \sum_{\sigma: x \in \sigma} \int_\sigma
		e^{-n\nu(B_{r_n}(y)) - \delta_2 f_0 n r_n^d} \nu(dy) \nu(dx)
		\nonumber \\
		& = 2 n^2 \int_A 
		\sum_{\sigma:y \in \sigma}
		\int_\sigma 
		%\1\{x,y \in \sigma\}
		e^{-n\nu(B_{r_n}(y)) - \delta_2 f_0 n r_n^d} \nu(dx) \nu(dy)
		\nonumber \\
		& \leq c' n^2 r_n^d \int_A \sum_{\sigma: y \in \sigma}
		e^{-n\nu(B_{r_n}(y)) - \delta_2 f_0 n r_n^d}  \nu(dy)
		\nonumber \\
		& \leq c'' n r_n^d I_n e^{- \delta_2 f_0 nr_n^d},
		\label{e:XE}
	\end{align}
	where for the third line we used the fact that $\lambda(\sigma)$
	is bounded by a constant times $r_n^d$, and in the fourth line
	we used the fact that there are a bounded number of shapes
	$\sigma$ that contain $y$ and are consistent with the diameter 
	condition.

	Next, let
	$N_1(\sigma)$ denote the number of ordered
	pairs $(y,z)$ of distinct points of $\X_{n-1} \cap \sigma$
	such that $y \prec u$ for all points of $\X_{n-1} \cap \sigma
	\setminus \{y\}$. 
	Then $\X_{n-1}(\sigma) \leq 1 +N_1(\sigma)$.
	Therefore
	\begin{align*}
		& \E[(\X_{n-1}(\sigma ) -1)
		\1_{E_{x,\sigma}} ]
		 \leq \E[N_1(\sigma )
		\1_{E_{x,\sigma}} ]
		%~~~~~~~~~~~~~~~~~~~~~~~~
		\\
		& ~~~~~~~~~~  \leq n (n-1) \int_\sigma \int_\sigma
		(1- \nu([(\sigma \cap A_y)
		\oplus B_{(1-\sqrt{d} \delta)r_n}(o)]
		\setminus (\sigma \cap A_y) ))^{n-3} 
		\nu(dz) \nu(dy).
	\end{align*}
	The $z$-integral is bounded by a constant times $r_n^d$, and
	by a similar application of Fubini's theorem to the one at
	\eqref{e:XE} we obtain that
\begin{align}
	n \int_A \sum_{\sigma:x \in \sigma} \E[(\X_{n-1}(\sigma) -1) 
	\1_{E_{x,\sigma}}] \nu(dx)
	= O((nr_n^d)^2 e^{-\delta_2 f_0 nr_n^d} I_n ).
	\label{e:XXE}
\end{align}

Next, let $N_2(\sigma)$ denote the number of ordered triples $(y,z,u)$
of distinct points of $\X_{n-1} \cap \sigma$ such that $y \prec v$
for all $v \in \X_{n-1} \cap \sigma \setminus \{y\}$.

Then provided $\X_{n-1}(\sigma) \neq 0$,
$(\X_{n-1}(\sigma) -1)(\X_{n-1}(\sigma)-2)$ is the number of ordered
pairs of vertices of $\X_{n-1} \cap \sigma$, other than the
first one in the $\prec $ order, and equals $N_2(\sigma)$.
If $E_{x,\sigma}$ occurs the $\X_{n-1}(\sigma) \neq 0$.
Hence
	\begin{align*}
		& \E[(\X_{n-1}(\sigma ) -1) (\X_{n-1}(\sigma) -2)
		\1_{E_{x,\sigma}} ]
		 = \E[N_2(x,\sigma )
		\1_{E_{x,\sigma}} ]
		%~~~~~~~~~~~~~~~~~~~~~~~~
		\\
		& ~~~~~~~~~~  \leq n^3\int_\sigma \int_\sigma \int_\sigma
		(1- \nu([(\sigma \cap A_y)
		\oplus B_{(1-\sqrt{d} \delta)r_n}(o)]
		\setminus (\sigma \cap A_y) ))^{n-4} 
		\nu(du) \nu(dz) \nu(dy).
	\end{align*}
	The $(z,u)$-integral is bounded by a constant times $r_n^{2d}$, and
	by a similar application of Fubini's theorem to the one at
	\eqref{e:XE} we obtain that
\begin{align*}
	n \int_A \sum_{\sigma:x \in \sigma} \E[(\X_{n-1}(\sigma) -1)
	(\X_{n-1}(\sigma) -2)
	\1_{E_{x,\sigma}}] = O((nr_n^d)^3 I_n e^{-\delta_2 f_0 nr_n^d}).
\end{align*}
Combining this with \eqref{e:EDD}, \eqref{e:XE} and \eqref{e:XXE}
we obtain \eqref{e:EDF+}.

The proof of  
\eqref{e:PoEDF+} is similar, using the Mecke formula.
\end{proof}

\subsection{Variance estimates for large components}

\begin{proposition}[Variance estimates for moderately large components]
	\label{p:varRmod}
	There exists $\rho \in (4,\infty)$ 
	such that if $\xi_n$ stands for any
	of $R_{n,\rho,(\log n)^2}$, $R'_{n,\rho,(\log n)^2}$,
	 $K_{n,\rho,(\log n)^2}$, or $K'_{n,\rho,(\log n)^2}$,
	then $\Var(\xi_n) = O(e^{-  nr_n^d} I_n)$ 
	as $n \to \infty$.
\end{proposition}
\begin{proof}
	Analogously to Proposition \ref{p:varRmed}
	the result follows from the next two
	lemmas.
\end{proof}
\begin{lemma}
	\label{l:EDml+}
	There exists $\rho_0 >1$ 
	such that for any fixed $\rho \geq \rho_0$ we have
	as $n \to \infty$ that
	\begin{align}
		n \int_A \E[ (D_x^+ R_{n,\rho,(\log n)^2}(\X_{n-1}))^2 ] \nu(dx) & 
		= O(e^{- nr_n^d} I_n);
		\label{e:mod+1}
		\\
		n \int_A \E[ (D_x^+ R_{n,\rho,(\log n)^2}(\eta_{n}))^2 ]
		\nu(dx) 
		& = O(e^{- nr_n^d} I_n).
		\label{e:Pomod+1}
	\end{align}
\end{lemma}
\begin{proof}
Let $\rho > 4$.
For $y \in \X_{n-1}$,
 adding a point at $x$ can only increase the
 diameter of the component containing $y$.
 Therefore if adding a point at $x$
 causes $y$ to  be in a component
of diameter in the range $(\rho r_n,(\log n)^2r_n]$ when it was not
	before, then
$y$ must previously have been in a component of
	diameter at most $\rho r_n$, and since also the added
point at $x$ affects this component we must have
	$\| y-x\| \leq (\rho +1 )r_n \leq 2 \rho r_n$.
	Also
	%, since $\rho >4$,
	event $\scr M^*_{n,\rho/4,(\log n)^2}(x,\cX_{n-1})$,
	defined at \eqref{e:medevent2},
	must occur.  Therefore defining
	$N_x := \#(\cX_{n-1} \cap B_{2 \rho r_n}(x))$, we have
	$D_x^+R_{n,\rho,(\log n)^2}(\cX_{n-1}) \leq
	N_x {\bf 1}_{
		%\tilde{\scr E}_9(x,\rho/4,n)}$.
		\scr M^*_{n,\rho/4,(\log n)^2}(x,\cX_{n-1})}$.
	Hence  by the Cauchy-Schwarz inequality,
	Lemma \ref{l:E9}
	and a standard moment estimate on the Binomial distribution,
	\begin{align*}
		\E[
			(D_x^+R_{n,\rho, (\log n)^2}(\cX_{n-1}))^2 ] \leq
		(\E[N_x^4])^{1/2}
		%(\Pr [\tilde{\scr E}_9(x,\rho/4,n)])^{1/2}
		(\Pr [\scr M^*_{n,\rho/4,(\log n)^2}(x,\X_{n-1})])^{1/2}
		\\
		= O( n^2 r_n^{2d} \exp(- (\delta_4 \rho/8) nr_n^d)),
	\end{align*}
	where $\delta_4$ is as in Lemma \ref{l:E9}.  Choosing $\rho
	$ so that $\delta_4 \rho > 8(\theta_d f_0 +3)$, and 
using Lemma \ref{l:Ilower}, we obtain that 
	$$
	n \int_A \E[ (D_x^+ R_{n,\rho, (\log n)^2}(\X_{n-1}))^2 ] \nu(dx) = 
	O(n e^{-( \theta_d f_0 +2)  nr_n^d} ) = O(e^{- nr_n^d} I_n),
	$$
	as required for \eqref{e:mod+1}. The proof of \eqref{e:Pomod+1}
	is similar.
	 \end{proof}
\begin{lemma}
	\label{l:EDml-}
	There exists $\rho_0 >1$ such that
	if $\rho \geq \rho_0$ then as $n \to \infty$,
	\begin{align}
		n \int_A \E[ (D_x^- R_{n,\rho,(\log n)^2}(\X_{n-1}))^2 ] \nu(dx) = O(e^{- nr_n^d} I_n); \label{e:EDF'-}
		\\
		n \int_A \E[ (D_x^- R_{n,\rho,(\log n)^2}(\eta_{n}))^2 ] \nu(dx) = O(e^{- nr_n^d} I_n).
		\label{e:PoEDF'-}
	\end{align}
\end{lemma}
\begin{proof}
	Let $\rho > 1$.
	For this proof, given $n$ and given $x \in A$ let $N_x$ denote
	the number of vertices $y \in \X_{n-1} \cap B_{2(\log n)^2r_n}(x)$
	such that $\cC_{r_n}(y,\X_{n-1}) \cap B_{r_n}(x) \neq \emptyset$
	and $\diam \cC_{r_n}(y,\X_{n-1}) \in (\rho r_n,(\log n)^2 r_n]$.
	Then 
	$ D_x^- R_{n,\rho,(\log n)^2}(\X_{n-1}) \leq N_x$.

	We have that $\E[N_x] \leq J_{1,x} + J_{2,x}$, where we set
	\begin{align*}
		%\E[ N_x] & \leq
		J_{1,x} :=  &	 \int_{B_{3 \rho r_n}(x)} n
		\Pr[ \diam (\cC_{r_n}(y,\X_{n-2}^y))
		\in (\rho r_n, (\log n)^2 r_n]] \nu(dy)
		\\
		J_{2,x} := &   \int_{B_{2(\log n)^2r_n}(x)
		\setminus B_{3 \rho r_n}(x)} n \Pr[ \diam (\cC_{r_n}(y,\X_{n-2}^y))
		\in (\|y-x\|/2, (\log n)^2 r_n]] \nu(dy).
	\end{align*}
	Let $\delta_4 $ be as in Lemma \ref{l:E9}.
	By that result, %(with $\X_{n-2} $ rather than $\X_n$),
	\begin{align}
		J_{1,x} \leq n \fmax \theta_d (3 \rho r_n)^d \exp (- \delta_4 \rho
		n r_n^d ).
		\label{e:J1}
	\end{align}
	Also by Lemma \ref{l:E9},
	\begin{align*}
		J_{2,x} & \leq n \int_{A \setminus B_{3 \rho r_n}(x)}
		\exp(-\delta_4 (\|y-x\|/2) nr_n^{d-1})
		\nu(dy)
		\\
		& \leq n \fmax \int_{\R^d \setminus B_{3 \rho r_n}(o)}
		\exp(-\delta_4 (\|u\|/2) n r_n^{d-1}) du
		\\
		& = n\fmax \int_{\{v: \|v\| > 3 \rho r_n  (\delta_4/2) nr_n^{d-1} \}} 
		e^{-\|v\|}
		(\delta_4 nr_n^{d-1}/2)^{-d}
		dv
		\\
		& = (2/\delta_4)^d \fmax (nr_n^d)^{1-d} \int_{\rho \delta_4
		nr_n^d} ^\infty e^{-t} d \theta_d t^{d-1} dt
		\\
		& \leq c \delta_4^{-1} \fmax  \rho^{d-1} e^{-\rho \delta_4 nr_n^d},
	\end{align*}
	where the constant $c$ depends only on $d$. Combined with
	\eqref{e:J1}, this
	shows that if we take $\rho \geq (\theta_d f_0 +3)/ \delta_4$
	then for $n$ large 
	$\E[N_x] \leq \exp(- (  \theta_d f_0 +2) nr_n^d)$ for all $x \in A $,
	and then using Lemma \ref{l:Ilower} we obtain that
	\begin{align}
	n \int_A \E[ N_x ] \nu (dx) = O(e^{- nr_n^d} I_n).
		\label{e:ENx}
	\end{align}

	Next, observe that
	$\E[N_x(N_x-1)] \leq J_{3,x} + 2 J_{4,x}$ where we set
	\begin{align*}
		%\E[ N_x] & \leq
		J_{3,x} :=  &	 \int_{B_{3 \rho r_n}(x)}
		\int_{B_{3 \rho r_n}(x)} n^2 \Pr[ \diam (\cC_{r_n}(y,\X_{n-3}^{y,z}))
		\in (\rho r_n, (\log n)^2 r_n]] \nu(dz) \nu(dy);
		\\
		J_{4,x} := &   \int_{B_{(\log n)^2r_n}(x)
		\setminus B_{3 \rho r_n}(x)} 
		\int_{B_{\|y-x\|}(x)}
		n^2 \Pr[
			\diam (\cC_{r_n}(y,\X_{n-3}^{y,z}))
		\in (\|y-x\|/2, (\log n)^2 r_n]] \\ 
	& ~~~~~~~~~~~~~~~~~~ 
	 ~~~~~~~~~~~~~~~~~~ 
	 ~~~~~~~~~~~~~~~~~~ 
	 ~~~~~~~~~~~~~~~~~~ 
		 \nu (dz) \nu(dy).
	\end{align*}
By Lemma \ref{l:E9},
	\begin{align}
		J_{3,x} \leq n^2  (\fmax   \theta_d (3 \rho r_n)^d)^2
		e^{-\delta_4 \rho nr_n^d}.
		\label{e:J3}
	\end{align}
	 Also by Lemma \ref{l:E9},
	\begin{align*}
		J_{4,x} & \leq n^2 \int_{A \setminus B_{3 \rho r_n}(x)}
		\exp(-\delta_4 (\|y-x\|/2) nr_n^{d-1})
		(\fmax \theta_d \|y-x\|^d) \nu(dy)
		\\
		& \leq n^2 \fmax^2 \theta_d \int_{\R^d \setminus B_{3 \rho r_n}(o)}
		\exp(-\delta_4 (\|u\|/2) n r_n^{d-1}) \|u\|^d du
		\\
		& = n^2\fmax^2 \theta_d \int_{\{v: \|v\| > 3 \rho r_n  (\delta_4/2)
		nr_n^{d-1} \}} 
		e^{-\|v\|} \|v\|^{d}
		(\delta_4 nr_n^{d-1}/2)^{-2d}
		dv
		\\
		& = (2/\delta_4)^{2d} \fmax^2 (nr_n^d)^{2-2d} \int_{\rho \delta_4
		nr_n^d} ^\infty e^{-t} d \theta_d t^{2d-1} dt
		\\
		& \leq c \delta_4^{-1} \fmax^2 \rho^{2d-1} nr_n^d e^{-\rho \delta_4 nr_n^d},
	\end{align*}
	where the constant $c$ depends only on $d$. Combined with
	\eqref{e:J3}, this
	shows that if we take $\rho \geq ( \theta_d f_0 +3)/ \delta_4$
	then for $n$ large 
	$\E[N_x(N_x-1)] \leq \exp(- (\theta_d f_0 +2) nr_n^d)$ for all $x \in A $,
	and then using Lemma \ref{l:Ilower} we obtain that
	$$
	n \int_A \E[ N_x (N_x-1) ] \nu (dx) = O(e^{- nr_n^d} I_n).
	$$
	Combined with \eqref{e:ENx} this shows that \eqref{e:EDF'-} holds.
	The proof of \eqref{e:PoEDF'-} is similar. 
	\end{proof}

%Given $x \in \R^d$,  $i \in \N$, $r >0$, and finite $\X \subset \R^d$,
%let $\ell_i(x,\cX,r)$ denote
%the (Euclidean) diameter of the $i$th largest (in diameter) component
%of $G(\X,r)$ 
%intersecting $B_r(x)$.
%%{\bf [different from earlier use of notation $\ell_i$ - probably OK.]}.
%If $i$ exceeds
%the number of components of $G(\X,r)$ intersecting 
%$B_r(x)$ then set $\ell_i(x,\cX,r):=0$.

%Given also $\rho \in (0,\infty)$, and $n >0$,
%setting $r=r(n)$ define the event
%$
%	\scr T_{n,\rho}(x,\X,\rho) := \{\ell_2(x, \X,r) >  \rho r  \}
%	$

\subsection{Variance estimates: conclusion}

Putting together the preceding estimates, we obtain the asymptotic
variance for $K'_n$ and (when $d \geq 3$) for  $K_n$:

\begin{proposition}
	\label{p:varK}
	Assume that $nr_n^d \to \infty$ and $\liminf(I_n) >0$
as $n \to \infty$. 
Then
%as $n \to \infty$, we have 
\begin{align}
 &	\Var[K'_n] = I_n 
	%ne^{-nf_0 \theta_d r^d} 
	(1+O((nr_n^d)^{(1-d)/2})); 
	\label{0804a} 
	\\
	{\rm if} ~ d \geq 3~ {\rm then}~
	&\Var[K_n] =% ne^{-nf_0 \theta_d r^d}
	I_n (1+O((nr_n^d)^{1-d/2})).
	\label{0804b}
\end{align}
\end{proposition}
\begin{proof}
   Note  $K'_n = S'_n + K'_{n,0,\infty}$, where $S'_n$
	%(the number of singletons)
	and $K'_{n,\eps ,\rho}$ were defined at
	\eqref{e:def_Sn}, \eqref{e:Knrdef}.

	Let $\rho \in (4,\infty)$ be as in Proposition 
	\ref{p:varRmod}.  Let $\rho_0$
	be as in Proposition \ref{p:varT_n}.  Let $\eps  = \rho_0$.

	Let $W_n := K'_{n,(\log n)^2,\infty} $.
	Since $|W_n -1| $ is bounded by
	$ Z_n+ 1$ (where $Z_n = \#(\eta_n)$),
	the Cauchy-Schwarz inequality and Lemma \ref{l:RvsR0an}
	yield that
	\begin{align}
		\Var[W_n] = \Var[W_n-1] \leq \E[(W_n-1)^2]
		& \leq (\E[(Z_n+1)^4])^{1/2}  
		(\Pr[W_{n} \neq 1])^{1/2}
		\nonumber
		\\
		& = O( e^{-\frac12 n r_n^d} I_n).
		\label{e:varW}
	\end{align}
	Then $ K'_{n,0,\infty} = 
		K'_{n,0,\eps } + K'_{n,\eps ,\rho}  
		+ K'_{n,\rho,(\log n)^2}
		 + W_n  $.
	By the estimate $(u+v+w+x)^2 \leq 4 (u^2+v^2+w^2 +x^2)$
	(a consequence of Jensen's inequality), Propositions
	\ref{p:varT_n},
\ref{p:varRmed}
	and
	%\ref{p:var_lar},
	\ref{p:varRmod}, along with
	 \eqref{e:varW},
	\begin{align}
		\Var[K'_{n,0,\infty}]
		 & \leq 4( \Var[K'_{n,0,\eps }] + \Var[ K'_{n,\eps ,\rho} ]
		+ \Var[ K'_{n,\rho,(\log n)^2}] + \Var[W_n]) 
		\nonumber \\
		& = O((nr_n^d)^{1-d} I_n).
		\label{0805a}
	\end{align}
	By Proposition \ref{p:Sn},
	%(\ref{0804c}),  
	$\Var[S'_n] = I_n(1+ e^{- \Omega(nr_n^d)})$.
	Hence by the Cauchy-Schwarz inequality,
	$\Cov(S'_n,K'_{n,0,\infty}) = O( (nr_n^d)^{(1-d)/2} I_n)$, and thus
	$$
	\Var(K'_n) = \Var(S'_n) + \Var(K'_{n,0,\infty})
	+ 2 \Cov(S'_n,K'_{n,0,\infty}) =
	I_n + O( (nr_n^d)^{(1-d)/2} I_n),
	$$
	which is (\ref{0804a}). The proof of (\ref{0804b})
	is similar, but now using Proposition 
	%\ref{p:varsmallbin}
	\ref{p:varRsmall}
	instead of Proposition
	\ref{p:varT_n},
	which accounts for the different power
	of $nr_n^d$ in (\ref{0804b}). 
\end{proof}

We can now also determine the  asymptotic variance for $R'_n$ 
and (if $d \geq 3$) for  $R_n$.
\begin{proposition} 
	\label{p:varR'}
	Under assumptions \eqref{c:lowerbound} and \eqref{c:upperbound},
	as $n \to \infty$ we have
	%\eqref{e:varR}.
\begin{align}
	& \Var [R'_n]  = I_n(1+ O((nr_n^d)^{(1-d)/2}));
	\label{e:varR'}
	\\
	{\rm if} ~ d \geq 3, ~~~~ &
	\Var [R_n]  = I_n(1+ O((nr_n^d)^{1-d/2})).
	\label{e:varR}
\end{align}
\end{proposition}
\begin{proof}
	Let $ 0 < \eps < \rho$ with $\eps < \rho_0$ and
	$\rho_0$ as in Proposition \ref{p:varT_n}.
	By Jensen's inequality and Propositions  
	\ref{p:varT_n}, 
	\ref{p:varRmed} and \ref{p:varRmod},
	\begin{align}
		\Var [R'_{n,0,(\log n)^2}]
		& \leq  3( \Var [R'_{n,0,\eps} ] +
		\Var [R'_{n,\eps,\rho} ] + \Var [ R'_{n,\rho,(\log n)^2}])
		\nonumber
		\\
		& = O( (nr_n^d)^{1-d} I_n).
		\label{e:0903a}
	\end{align}

	Since
	 $|R'_n- S'_n- R'_{n,0, (\log n)^2}| \leq Z_n$,
	 by the Cauchy-Schwarz inequality
	 and Lemma \ref{l:RvsR0an},
	 \begin{align*}
                 \E[|R'_n - S'_n - R'_{n,0,(\log n)^2}|^2]
		 & \leq (\E[Z_n^2])^{1/2}
		 (\Pr[R'_n \neq S'_n + R'_{n,0,(\log n)^2}])^{1/2}
		 \\
		 & %= O(n^{-2b_c}) 
		 = O(e^{-  nr_n^d/2}I_n).
         %\label{e:Rtailvar}
         \end{align*}
	 Then using \eqref{e:0903a} and Jensen's inequality again yields
	 \begin{align}
		 \Var[R'_n-S'_n ] \leq 2(\Var[R'_n - S'_n -
		 R'_{n,0,(\log n)^2} ]  + \Var[ R'_{n,0, ( \log n)^2}
		 ]) = O( (nr_n^d)^{1-d} I_n).
		 \label{e:0903b}
	 \end{align}
	 By Proposition \ref{p:Sn}, $\Var[S'_n]= I_n (1+ e^{-\Omega(nr_n^d)})$.
	 Using this along with \eqref{e:0903b} and
	 the Cauchy-Schwarz inequality gives us \eqref{e:varR'}.

	 The proof of \eqref{e:varR} is similar. We use Proposition
	 \ref{p:varRsmall} instead of Proposition
	\ref{p:varT_n}, 
	and Proposition \ref{p:var_iso_bin}
	instead of Proposition \ref{p:Sn}.
\end{proof}

\subsection{Proof of convergence in distribution results}

\begin{proof}[Proof of Theorem \ref{t:momCLT}]
	By Proposition \ref{p:varK}
	we have $\Var[K'_n] = I_n (1+ (nr_n^d)^{(1-d)/2})$.
	By Proposition \ref{p:varR'}
	we have $ \Var[R'_n]  = I_n(1+ (nr_n^d)^{(1-d)/2})$.
	Thus we have \eqref{e:VarPo}.
	If $d \geq 3$ then
	by Proposition \ref{p:varK} we have
	$\Var[K_n] =  I_n(1+O(nr_n^d)^{1-d/2})$,
	and by Proposition \ref{p:varR'}
	we have $\Var[R_n]= I_n(1+ O((nr_n^d)^{1-d/2}))$.
	Thus we have
	\eqref{e:EVarsim}.

%\label{p:var_iso_bin}

	By \eqref{0805a}
	in the proof of Proposition \ref{p:varK}
	if $\xi'_n = K'_n-1$,
	or \eqref{e:0903b}
	in the proof of Proposition
	\ref{p:varR'} if $\xi'_n = R'_n$,
%Since $K'_n - S'_n = K'_{n,0,\eps }+ K'_{n,\eps ,\rho} + K'_{n,\rho,\infty}$,
%using Propositions  %\ref{p:varsmallbin},
	%\ref{p:varT_n},
%\ref{p:var_med} and \ref{p:var_lar}
%we have
$$
\Var [\xi'_n-S'_n] = O((nr_n^d)^{1-d}I_n).
$$
%By Proposition \ref{p:var_iso_bin} we have
	%$\Var[S'_n]= I_n(1+O(e^{-cnr^d})).$ Then using the Cauchy-Schwarz 
	%inequality we obtain that $\Var[K'_n] = I_n(1+ O((nr^d)^{(1-d)/2}))$.
	%Using Proposition \ref{l:CLTS} we obtain that $\Var[S'_n]^{-1/2}
	%(S_n - \E[S'_n]) \toD N(0,1)$, 
	%so by Slutsky's theorem
	 %$$
	 %I_n^{-1/2} (S'_n - \E[S'_n]) \toD N(0,1).
	 %$$
	 Hence
	 $\Var(I_n^{-1/2}(\xi'_n - S'_n - \E[\xi'_n-S'_n])) = O((nr_n^d)^{1-d})$.
	 %so that
	 % $I_n^{-1/2}(\xi'_n - S'_n - \E[\xi'_n-S'_n]) \toP 0$.
	 Hence by Lemma
	 \ref{l:QSlut},
	 %Slutsky's theorem, 
	  \begin{align*}
		  \dk(I_n^{-1/2}(\xi'_n -  \E[\xi'_n]) , N(0,1))
		  = O(\dk(I_n^{-1/2} (S'_n - I_n),N(0,1)) + 
		  (nr_n^d)^{(1-d)/3}),
	  \end{align*}
	  and \eqref{e:CLTPo} then follows by
	  \eqref{e:CLTS'}. 
	  %and
	  %$$
	  %\Var[K'_n]^{-1/2}(K'_n  - \E[K'_n]) \toD N(0,1).
	  %$$

	  When $d \geq 3$
	  we prove \eqref{e:CLTbin} similarly. In 
	  the binomial setting we get
	  $(nr_n^d)^{2-d}$ instead of $(nr_n^d)^{1-d}$ in
	 \eqref{0805a}
	or \eqref{e:0903b}, and therefore 
	$\Var(I_n^{-1/2} (\xi_n -S_n -\E[\xi_n -S_n])) = O((nr_n^d)^{2-d})$.
	Therefore using Lemma  \ref{l:QSlut}
	and \eqref{e:CLTS} we have
	  \begin{align*}
		  \dk(\tilde{I}_n^{-1/2}(\xi_n -  \E[\xi_n]) , N(0,1))
		  & = O(\dk(\tilde{I}_n^{-1/2} (S_n - \tilde{I}_n),N(0,1)) + 
		  (nr_n^d)^{(2-d)/3})
		  \\
		  & = O( (nr_n^d)^{(2-d)/3} + I_n^{-1/2}).
	  \end{align*}
	Using the fact that
	$\tilde{I}_n = I_n(1+ O(e^{-c' nr_n^d}))$ for some further 
	constant $c'$ by Lemma \ref{l:diff_I_Itilde},
	and using Lemma  \ref{l:QSlut} again
	we obtain \eqref{e:CLTbin}.
%
%
	%  The argument for the central limit theorem for
	%  $K_n$ and for $R_n$ is similar. We use 
	%the fact that
	%$\tilde{I}_n \sim I_n$ by Lemma \ref{l:diff_I_Itilde}.
%
	%The $L^2$ convergence $\xi'_n/I_n \to 1$
	%comes from \eqref{e:meanasymp} and \eqref{e:VarPo}, and the
	%final assertion \eqref{e:basicLLN}
	%comes from \eqref{e:meanasymp} and \eqref{e:EVarsim}. 
\end{proof}

\begin{proof}[Proof of Theorem \ref{t:momCLTunif}]
	We assume \eqref{e:supcri},  \eqref{e:supcriupper}
	and that $\nu$ is uniform on $A$. 
	For $n \geq 1$ define $\gamma_n$ as at \eqref{e:supcriupper}
	and set $a_n:= - \gamma_n$, so
$
a_n:= (2-2/d) (\log n - {\bf 1}\{d \geq 3\}\log \log n)
- n \theta_d f_0 r_n^d.
$
	By %\eqref{e:supcri} and
	\eqref{e:supcriupper},
	$a_n \to \infty$
as $n \to \infty$.
We claim $I_n \to \infty$. Indeed, if $d=2$ then
$$
ne^{-n \pi f_0 r^2} = ne^{a_n - \log n} \to \infty,
$$
so that $I_n \to \infty$ by Proposition \ref{p:average2d}.
	If instead $d \geq 3$ then
\begin{align*}
e^{-n \theta_d f_0 r_n^d/2} r_n^{1-d}
	= e^{a_n/2} \Big( \frac{\log n}{n} \Big)^{1-1/d}r_n^{1-d}
	= e^{a_n/2} \Big( \frac{nr_n^d}{\log n} \Big)^{(1/d)-1}
\end{align*}
which tends to infinity because,
by \eqref{e:supcriupper}, for $n$ large we have
$n \theta_d f_0 r_n^d \leq 2 \log n $. 
Therefore by Proposition 
\ref{p:average3d+},
we have $I_n \to \infty$
in this case too, justifying our claim.

Suppose $d=2$.
	By Proposition \ref{p:average2d} and \eqref{e:defI'},
	as $n \to \infty$
	we have $I_n = \mu_n(1+ O((nr_n^2)^{-1/2})) $  and
	then \eqref{e:VPoUn2} follows from \eqref{e:VarPo}.
	Also by Lemma \ref{l:QSlut}, \eqref{e:VarPo}
		and \eqref{e:CLTPo},
		\begin{align}
			\dk \Big(\frac{\xi'_n - \E[\xi'_n]}{\mu_n^{1/2}},
			N(0,1) \Big)
			&	\leq 
			\dk \Big( \frac{\xi'_n - \E[\xi'_n]}{
				I_n^{1/2}}	,N(0,1)\Big)
			+ \Big(  
			\Var (( \mu_n^{-1/2} - I_n^{-1/2}) (\xi'_n - 
			\E[\xi'_n])) \Big)^{1/3} \nonumber \\
			& = O( (nr_n^2)^{-1/3} + I_n^{-1/2}) + O \Big(
			\Big( \big( \frac{I_n}{\mu_n} \big)^{1/2} -1 \Big)^{2/3} \Big)
			\nonumber \\ &
			 = O( (nr_n^2)^{-1/3} + \mu_n^{-1/2}), 
			 \label{e:dkC}
		\end{align}
		and hence \eqref{e:CLTPoUn2}.

		Now suppose $d \geq 3$. 
		By Proposition \ref{p:average3d+}, as $n \to \infty$ we have
	 $I_n = \mu_n \Big(1+ O\Big( \big( \frac{\log(nr_n^d)}{nr_n^d}
	\big)^2 \Big) \Big)$.
	Hence %using  \eqref{e:meanasymp} we have \eqref{e:EUn3+},
	 using \eqref{e:VarPo}
	we have \eqref{e:VarPoUn3+}, and using 
	\eqref{e:EVarsim}
	we have \eqref{e:VarBiUn3+}.

	Also using Lemma \ref{l:QSlut},
	we can obtain 
		\eqref{e:CLTPoUn3+}		 
	from \eqref{e:CLTPo}
	and 
		\eqref{e:CLTbiUn3+}
	from 
	\eqref{e:CLTbin}, in both cases by similar steps to those used
	at \eqref{e:dkC} to derive \eqref{e:CLTPoUn2}.
\end{proof}

\appendix

\section{Index of notation}

In Section \ref{secintro} we introduced 
the following notations:  $G(\cX,r)$, $K(G)$, $R(G)$, $\cX_n$, $\Po_n$ 
$K_n, K'_n$, $R_n$, $R'_n$, $A$ , $\theta_d$ and $\gamma_n$ and
$\lambda$. Also
$\mu_n$, $\partial A$, $D^o$, $\overline{D}$, $N(0,1)$, $C^2$, $Z_t$
and $\sigma_A$.

In Section
\ref{s:stateresults} before Subsection \ref{ss:Resultsgenf}
we introduced the notation $f$, $\fmax$, $f_0$, $\nu$, and
$O(\cdot)$, $o(\cdot)$, $\Theta(\cdot)$ and $\sim$; also 
$\dk$ and $\dtv$.

In Subsection \ref{ss:Resultsgenf} we introduced notation 
$I_n$, $b^+$, $b^-$, $b_c$, $b'_c$ and $C^{1,1}$.
In Subsection \ref{ss:ResUnif} we introduced notation $c_{d,A}$

In Section \ref{s:prelims} before Subsection
\ref{ss:geom} we introduced notation $\oplus$ and $\|\cdot\|$,
$D^{(a)}$,
$aD$ and $[n]$, $\prec$, $A_x$ $\diam(\cdot)$ and $\#(\cdot)$.

In Subsection \ref{ss:geom} we introduced notation $\hat{n}_x$,
$\tau(A)$, $a(\cdot)$,  $g(\cdot)$ and $\mathbb H$.

In Subsection \ref{ss:ProbTools}
we introduced the notation $\bN(\mathbb R^d)$, ${\cal S}(\mathbb R^d)$,
$D_xF$, $D_x^+F$,  $D_x^-F$, $\LL(\cdot)$ and $\LL(\cdot|E)$.
In Subsection \ref{ss:perc} we introduced notation ${\cal C}_s(x, \cX)$,
$\scr U_n$, $\tilde{\scr U})_n$, $\cK_{n,k,\alpha}$,
$\scr G_{n,k}$ and $\tilde{\scr G}_{n,k}$.
Also $\cX^x$, $\cX^{x,y}$,
$\cX^{x,y,z}$,
 $\scr M_{n,\eps,K}(\cdot)$
and $\scr M^*_{n,\eps,K}(\cdot)$.
Also ${\cal L}_n(\cdot)$ and ${\cal L}_{n,2}(\cdot)$.

In Subsection \ref{ss:IsoMom} we introduce notation $\tilde{I}_n$
$I_n(\cdot)$, $\mathsf{Cor}_n$ and $p_n(\cdot)$.
Also $J_{1,n}$, $J_{2,n}$

In Section \ref{s:moments} before Subsection \ref{ss:means},
we introduced notation $\scr F_n(\cdot)$, $K_{n,\eps,\rho}(\cdot)$
and $ R_{n,\eps,\rho}(\cdot)$. Also 
$K_{n,\eps,\rho}$, 
$K'_{n,\eps,\rho}$, 
$R_{n,\eps,\rho}$ 
and $R'_{n,\eps,\rho}$. 

In Subsection \ref{ss:varsmallPo} we introduce notation 
$\scr T_x$, $\scr T_{x,y}$, $\scr E_x$, $\scr E_{x,y}$
and $\scr N_{x,y}$.

\end{document}